\documentclass{amsart}
\usepackage{latexsym}
\usepackage{amsmath,amssymb,amsfonts,amsthm,amscd}
\usepackage{tikz-cd}
\usepackage{geometry}
\usepackage{graphicx}
\usepackage{lscape}
\usepackage{enumerate}
\usepackage{mathrsfs}
\usepackage{comment}
\usepackage{arydshln}
\usepackage{stmaryrd}
\usepackage{mathtools}
\usepackage{braket}
\usepackage{color}
\usepackage[luatex,pdfencoding=auto]{hyperref}
\usepackage{aliascnt}
\usepackage{cleveref}
\crefname{section}{}{}
\creflabelformat{section}{\S#2#1#3}
\crefname{part}{Part}{Parts}
\crefname{equation}{}{}
\crefname{align}{}{}
\crefname{enumi}{}{}
\newtheorem{thm}{Theorem}[section]
\crefname{thm}{Theorem}{Theorems}
\newaliascnt{cor}{thm}

\aliascntresetthe{cor}
\crefname{cor}{Corollary}{Corollaries}
\newaliascnt{prop}{thm}
\newtheorem{prop}[prop]{Proposition}
\aliascntresetthe{prop}
\crefname{prop}{Proposition}{Propositions}
\newaliascnt{lem}{thm}
\newtheorem{lem}[lem]{Lemma}
\aliascntresetthe{lem}
\crefname{lem}{Lemma}{Lemmas}
\newaliascnt{clm}{thm}
\newtheorem{clm}[clm]{Claim}
\aliascntresetthe{clm}
\crefname{clm}{Claim}{Claims}

\theoremstyle{definition}
\newaliascnt{defi}{thm}
\newtheorem{defi}[defi]{Definition}
\aliascntresetthe{defi}
\crefname{defi}{Definition}{Definitions}
\newaliascnt{ex}{thm}
\newtheorem{ex}[ex]{Example}
\aliascntresetthe{ex}
\crefname{ex}{Example}{Examples}

\theoremstyle{remark}
\newaliascnt{rem}{thm}
\newtheorem{rem}[rem]{Remark}
\aliascntresetthe{rem}
\crefname{rem}{Remark}{Remarks}

\newcommand{\cf}{\textit{cf.\ }}

\newcommand{\bfA}{\mathbf{A}}
\newcommand{\bfE}{\mathbf{E}}
\newcommand{\bA}{\mathbb{A}}

\newcommand{\bC}{\mathbb{C}}

\newcommand{\bE}{\mathbb{E}}

\newcommand{\bG}{\mathbb{G}}

\newcommand{\bL}{\mathbb{L}}

\newcommand{\bP}{\mathbb{P}}
\newcommand{\bQ}{\mathbb{Q}}
\newcommand{\bR}{\mathbb{R}}

\newcommand{\bZ}{\mathbb{Z}}

\newcommand{\cB}{\mathcal{B}}

\newcommand{\cD}{\mathcal{D}}

\newcommand{\cO}{\mathcal{O}}

\newcommand{\sB}{\mathscr{B}}
\newcommand{\sC}{\mathscr{C}}

\newcommand{\sS}{\mathscr{S}}

\newcommand{\sW}{\mathscr{W}}

\newcommand{\ol}{\overline}

\newcommand{\wt}{\widetilde}

\newcommand{\id}{\mathrm{id}}
\newcommand{\pr}{\mathrm{pr}}
\newcommand{\Gm}{\mathbb{G}_{\mathrm{m}}}
\newcommand{\sm}{\mathrm{sm}}
\newcommand{\ad}{\mathrm{ad}}
\newcommand{\et}{\mathrm{et}}
\newcommand{\sing}{\mathrm{sing}}
\newcommand{\Ab}{\mathrm{Ab}}

\newcommand{\an}{\mathrm{an}}
\newcommand{\Zar}{\mathrm{Zar}}
\newcommand{\equi}{\mathrm{equi}}
\newcommand{\rat}{\mathrm{rat}}

\DeclareMathOperator{\Spec}{Spec}

\DeclareMathOperator{\Hom}{Hom}
\DeclareMathOperator{\Ext}{Ext}

\DeclareMathOperator{\Ker}{Ker}

\DeclareMathOperator{\Ima}{Im}
\DeclareMathOperator{\Res}{Res}
\DeclareMathOperator{\Tr}{Tr}
\DeclareMathOperator{\Tor}{Tor}
\DeclareMathOperator{\CH}{CH}

\DeclareMathOperator{\Pic}{\mathrm{Pic}}

\DeclareMathOperator{\Gr}{Gr}

\DeclareMathOperator{\Ind}{Ind}
\DeclareMathOperator{\Cone}{Cone}
\DeclareMathOperator{\Kom}{Kom}

\DeclareMathOperator{\pVMHS}{pVMHS}
\DeclareMathOperator{\VMHS}{VMHS}
\DeclareMathOperator{\AVMHS}{AVMHS}
\DeclareMathOperator{\LocSys}{LocSys}

\title{Asymptotics of local height pairing}
\author{Yuta Nakayama}
\date{\today}
\begin{document}

\begin{abstract}
    We discuss the asymptotics of the Archimedean part of the Arakelov intersection number.
    The theorem is motivated by recent conjectures and their proof strategy by Gao and Zhang on the Northcott property of the Beilinson--Bloch height pairing.
    Our method involves a homological algebra interpretation of the Archimedean height by Hain.
    This interpretation allows us to introduce motivic viewpoints using Deligne cohomology, cycle class maps and higher Chow groups.

    Especially, we compare the biextension by Hain and Brosnan--Pearlstein over \(\bC\) based on Poincar\'e line bundle and Hodge theory with the \(\Gm\)-biextension of Bloch and Seibold defined by two families of homologically trivial cycles on a generically smooth family of projective varieties over a smooth curve.
    Our comparison, a relative version of the work of Gorchinskiy, enhances his derived viewpoint on these biextensions.
    Especially when the family of varieties are smooth, the two constructions are related via derived regulator maps to Deligne cohomology, reinterpreted similarly to Beilinson's absolute Hodge cohomology, as well as the derived description of Hardouin's biextension that generalizes Poincar\'e line bundle by Hain.
    The comparison when the family defined over a smooth curve has a strongly semistable reduction further involves a simple monodromy computation using mixed Hodge modules.

    Along our discussion, we simplify the discussion of Bloch and Seibold, partly in the style of Gorchinskiy.
    For example, the symmetry of their biextension is proved more easily than their work.
\end{abstract}
\maketitle
\setcounter{tocdepth}{1}
\tableofcontents

\section*{Introduction}
Height functions have been a useful tool in arithmetic geometry in the study of rational points and algebraic cycles.
It has been so since the pioneering work of Weil.
He proved extended the Mordell--Weil theorem to the case of abelian varieties over number fields using his height machine.
N\'eron and Tate refined the theory of Weil to the N\'eron--Tate height pairing on abelian varieties.
This helped Faltings to prove the Mordell conjecture on the finiteness of rational points on smooth curves over number fields.
More generally, Beilinson and Bloch independently suggested a conjectural definition of height parings on homologically trivial cycles on smooth projective algebraic varieties over number fields.

Usually, a height pairing on cycles on varieties over a number field \(K\) is the sum of local heights over the places of \(K\).
This description is elementary in the case of the height of Weil.
In the case of the Beilinson--Bloch height, the Archimedean part is unconditionally known.
This definition involves the Green currents, and is analytic in nature.
The non-Archimedean part is conjectural in general, except the good reduction case.
In the last case, the local pairing is the intersection number of the closures of \(W\) and \(Z\) on a suitable model of \(X\) over the local ring at \(s\).

This paper concerns the asymptotic behavior of these local height pairings in a strongly semistable degenerating family over a curve.
We prove our result without any analysis.
The result is motivated by a conjecture of Gao and Zhang \cite{GZFam}.
They are concerned with the Northcott property of the Beilinson--Bloch height.
They provide a conjectural framework to prove this through Arakelov geometry.
Our asymptotic behavior is related with the nefness of their conjectural adelic line bundle.

We mention a few words about the proof before a more detailed introduction.
In the Archimedean case, Hain \cite[\S 3]{HainBiext} algebraically rephrased the analytic definition via Hodge theory and homological algebra.
The definition takes advantage of the notion of biextensions.
Biextensions, introduced by Grothendieck \cite[Expos\'e VII]{SGA7}, is a torsor over a direct product of two groups.
Our biextension is defined by the Deligne and absolute Hodge cohomology.
The proof is by comparing this with another biextension by Bloch \cite{BloBiext} defined by higher Chow groups via a cycle class map.
The biextension by Bloch is closer to the leading coefficient in the asymptotic behavior.
In fact, the coefficient is the non-Archimedean part of the height pairing defined with respect to the base curve, not the number field that we would have started with.

In fact, because our method is algebraic, it extends to the \(\ell\)-adic etale setting.

\subsection*{Overview of the main results}

Now we look more closely into the contents of this paper.
Let \(\pi\colon X\to S\) be a strongly semistable degeneration of smooth projective varieties over a smooth curve over a field \(k\).
Let \(W, Z\) be two generically homologically trivia cycles on \(X\) with respect to \(S\).
Assume that their codimensions \(m,n\) add up to \(1\) plus the relative dimension of \(\pi\).
On the special fiber \(X_s\) where \(W, Z\) have correct dimensions and \(W_s\) and \(Z_s\) do not intersect, there is a conjecturally defined local height pairing.
Also, Beilinson \cite{BeiHeiCyc} unconditionally defined a local height pairing \(\langle W, Z\rangle_{\mathrm{a}, s}\).
This is the easier analog of the conjectural construction, where a number field is replaced by a function field.

The definition of the local height pairing is unconditional when \(k = \bC\), so we begin our introduction in this case.
Take a closed point \(s\in S\).
Let \(t\) be a local parameter of \(S\) at \(s\).
For \(|t|<<1\), set \(h(t)\) to be the local height pairing at the fiber over the point corresponding to \(t\).

\begin{thm} \label{thm:intro}
    Assume the Hodge conjecture.
    Then \(h(t) - \langle W, Z\rangle_{\mathrm{a}, s}\log |t|\) is continuous and bounded.
\end{thm}

The assumption is partly to interpret \(\langle W, Z\rangle_{\mathrm{a}, s}\) motivically rather than inside a cohomology theory.
The assumption also enables us to use the development on mixed Hodge modules by Brosnan--Nie--Fang--Pearlstein \cite{BFNPANF}.
The motivic formulation enables us to prove the theorem by comparing biextensions as we said.
The same theorem is in the work \cite[Theorem 6.7]{SongAsym} of Song with a different analytic proof.
Holmes and de Jong \cite{HdJAsym} proved this result when \(m = n = 1\) using more classical explicit methods including theta functions.
Also, Wang \cite{WangAsym} expressed similar asymptotics in terms of analysis.

\subsection*{The biextension by Bloch}

Toward the overview of our proof, we first review the biextension by Bloch and Seibold.
Suppose that we have a family \(\pi\colon X\to S\) of smooth proper varieties over a smooth base.
Bloch constructed in \cite{BloBiext} a biextension that gives a line bundle on \(S\) out of two homologically trivial cycles on \(X\) using higher Chow groups except that the proofs of relevant moving lemmas turned out false by Suslin.
Later, Seibold \cite{SeiBier} extended the construction to the case of a generically smooth family of projective varieties over a smooth base.

More precisely, let \(\bfA^m(X/S)\) be the sheafification of the Zariski presheaf \((U\mapsto \CH^m_{\hom}(\pi^{-1}(U)))\) on \(S\), where \(\hom\) means the generically homologically trivial part with respect to \(U\).
Define \(\bfA^n(X/S)\) similarly.
Then, Bloch and Seibold constructed a \(\Gm\)-biextension of \(\bfA^m(X/S)\) and \(\bfA^n(X/S)\).
This is a \(\Gm\)-torsor over \(\bfA^m(X/S)\times \bfA^n(X/S)\) with some additional structures.

The Picard group is computed by the first cohomology of \(\Gm\).
Likewise, Grothendieck showed \cite[Expos\'e VII, Corollaire 3.6.5]{SGA7} that the isomorphism classes of \(\Gm\)-biextensions of two groups correspond to the isomorphism classes of \(\Gm\)-extensions of the derived tensor product of the two groups.
We first describe the class in \(\Ext^1(\bfA^m(X/S)\otimes^\bL \bfA^n(X/S), \Gm)\) that corresponds to the construction of Bloch and Seibold as a preliminary to the main portion of this paper.
Roughly, if \(z^m(*_X, -\bullet), z^1(*_S, -\bullet)\) and their analogs are the complex of sheaves defining the higher Chow groups on \(X\) and \(S\) and their open subsets, then the extension is defined by
\begin{align} \label{RoughBiext}    
    \pi_* z^m(*_X, -\bullet)\otimes \pi_* z^n(*_X, -\bullet) \to \pi_* z^{m+n}(*_X, -\bullet)\to z^1(*_S, -\bullet),
\end{align}
where the first morphism is the product of higher Chow groups and the second is the push-forward by \(\pi\), whose target is quasi-isomorphic to \(\Gm [1]\).

The case when \(S\) is the spectrum of a field has already been treated in \cite{GorBiext} by Gorchinskiy.
His construction of \(P_{HC}\), which was suggested by Bloch, recovers our construction in his case.
In fact, we do not fully generalize the construction by Gorchinskiy, resulting in a more elementary presentation.
For example, we do not use \(K\)-theory unlike him.
Also, we do not discuss Picard categories although \cref{prop:Calc} looks similar to some of his statements.
In particular, this proposition could be seen as a positive answer to \cite[Question 4.33]{GorBiext}.
We mention the component of his work on the Poincar\'e biextension later in this section.

We compare our argument with that of \cite{BloBiext, SeiBier}.
Those arguments are essentially the same, and our reasoning was inspired by that of Bloch.
However, we follow the steps to correctly sheafify his construction.
He often omits the details of those sheafifications.
Our work cares about the compatibility of the various constructions that Bloch has made, and the extension of the derived tensor product suits that purpose.

Our construction is slightly less explicit than papers of Bloch and Seibold.
For example, computing the extension class only pins down the isomorphism class of the biextension. 
As another example, aside from our use of \cite[Expos\'e VII, Corollaire 3.6.5]{SGA7}, Bloch uses the relative higher Chow groups by Landsburg \cite{LandRelChow}, which might be understood analogously to relative homology groups.
The relative higher Chow groups are defined just as the homology of the cone of a pull-back map between motivic complexes.
Actually, the relative higher Chow groups cannot be used in his context due to the lack of the regularity of a closed immersion.
Unlike this, our construction involves a cone of a more involved morphism.
In particular, it seems hard to construct sections of our biextensions for disjoint pairs of homologically trivial cycles as done by Bloch.
The difficulty is that we cannot silence some homotopy.
For this reason, we also present the construction by Bloch and Seibold briefly, but in an easier and more symmetric way than Bloch.
For example, we do not need the condition \((*)\) in the proof of Bloch \cite[Proposition 3]{BloBiext}.
Seibold chooses to be more explicit than us, constructing a ``cocycle data'' that gives the biextension.

However, restricting our attention to the \(\Gm\)-extension class has the following additional advantages.
First, our construction is well-defined immediately.
The previous work needed \cite[Proposition 3]{BloBiext} to ensure the well-definedness.
The proposition introduces the small etale site in the proof.
So, our construction abandons the etale site and only uses the small Zariski site as a by-product.
In short, our sheaf-theoretic construction is glued from the beginning on the small Zariski site.

Next, our procedures make it evident that our construction gives a biextension.
Before, we first needed to extend a sheaf of homologically trivial elements of Chow groups by \(\Gm\) for each of homologically trivial cycles with the complementary codimension.
We were then required to rearrange the extension as a biextension.
We combine these procedures together.
This combination also enables us to easily read off the symmetric nature of our construction.
The previous paper gave the author the impression that the symmetry was surprising.
Our shorter symmetry argument suggests that the symmetry arises naturally, although our argument works only for the isomorphism classes.
After all, we bring the symmetric nature of the biextension to the front of the construction.

The last two paragraphs also allow us to do without the Weil reciprocity law and homologically trivial cycles with disjoint supports.
These factors were previously used in the well-definedness argument as well as the symmetry argument.
See \cite[Lemma 2.8]{GorBiext} for another use of the law.
They seem to have been utilized in the discussion of Deligne pairings.
In our context, the Weil reciprocity law appears only when we need explicit sections of the biextension.

\subsection*{Comparison with Archimedean biextensions}

From now on, we shift our focus to another biextension by Hain in \cite[\S 3]{HainBiext} and its variants.
Over \(\bC\), he gave a \(\bC^\times\)-biextension of two intermediate Jacobians of Hodge structures coming from Betti cohomology when \(\pi\) is smooth.
More precisely, he interpreted the Poincar\'e line bundle over the intermediate Jacobians of \(R^{2m - 1}\pi_*\bZ(m)\) and \(R^{2n - 1}\pi_*\bZ(n)\) as a \(\Gm^\an\)-biextension.
Strictly speaking, the word ``biextension'' in \cite{HainBiext} means blended extensions, known as <<extensions panach\'ees>> in \cite[Expos\'e IX, \S 9.3]{SGA7} or as mixed extensions.
What we would call a biextension is the set of the isomorphism classes of blended extensions.
Hain discussed the latter only as a \(\Gm\)-torsor.
However, Hardouin published \cite[Proposition 4.4.1]{HarBiext} that in this kind of situation the set of the blended extensions form a biextension.
The point of Hain's construction is that it offers an algebraic reformulation of the Archimedean height pairing, originally defined using Green currents.

Here is a more precise description.
Suppose that \(S = \Spec \bC\) for simplicity.
Let
\[
    H\coloneqq H^{2m - 1}(X, \bZ(m))/\mathrm{tors}.
\]
The intermediate Jacobian of \(H\) (resp.~\(H^\vee = H^{2n - 1}(X, \bZ(n))/\mathrm{tors}\)) can be interpreted as the set of the isomorphism classes of the mixed Hodge structures \(A\) equipped with isomorphisms
\[
    \Gr_{\mathrm{W}}^0 A \xrightarrow{\simeq} \bZ, \quad \Gr_{\mathrm{W}}^{-1} A \xrightarrow{\simeq} H
\]
\[
    \textup{(resp.~}\Gr_{\mathrm{W}}^{-1} A \xrightarrow{\simeq} H^\vee, \quad \Gr_{\mathrm{W}}^{-2} A \xrightarrow{\simeq} \bZ(1)),
\]
such that other graded pieces of the weight filtration are zero.
Define \(\cB(H)\) to be the set of the isomorphism classes of the mixed Hodge structures \(B\) equipped with isomorphisms
\[
    \Gr_{\mathrm{W}}^0 B \xrightarrow{\simeq} \bZ, \quad \Gr_{\mathrm{W}}^{-1} B \xrightarrow{\simeq} H, \quad \Gr_{\mathrm{W}}^{-2} B \xrightarrow{\simeq} \bZ(1),
\]
such that other graded pieces of the weight filtration are zero.
Then the forgetful map from \(\cB(H)\) to the product of the two intermediate Jacobians exhibits \(\cB(H)\) as a \(\bC/\bZ \cong \bC^\times\)-biextension.
These constructions carry over to the case when \(H\) is a variation of mixed Hodge structures of weight \(-1\).
A similar construction works for mixed \(\bR\)-Hodge structures.
In this case, the intermediate Jacobians becomes a point, and there is a canonical isomorphism between the analog of \(\cB(H)\) and \(\bC/\bR \cong \bR\).
For disjoint homologically trivial cycles \(W, Z\) on \(X\) of codimensions \(m, n\) respectively, Hain constructed a specific element \(B_{W, Z} \in \cB(H)\).
The Archimedean height pairing equals the number obtained by forgetting to the real Hodge structure analog of \(\cB(H)\) and then shifting to \(\bR\).

Later, Lear \cite{LeaExt} studied the asymptotic behavior of the height pairing associated with the blended extensions.
This sparked the interest in the study of so-called Lear extensions.
Brosnan and Pearlstein \cite[Theorem 241]{BPJump} have completely defined the \(\bQ\)-line bundles which extend those obtained out of the biextension by Hain of the intermediate Jacobians of the smooth locus.

We show that biextensions and line bundles by Hain and by Brosnan--Pearlstein coincide with that by Bloch and Seibold.
When \(\pi\) is smooth, we compare the biextensions directly in \cref{sec:ArchCompSm}.
The comparison was claimed by de Jong \cite[Proposition 3.1]{deJongGrossSchoen}, but we write down a detailed argument.
We base ourselves on the derived regulator maps to Deligne cohomology.
Namely, the isomorphism class of the biextension by Hain roughly corresponds to
\begin{align} \label{RoughBiextD}
    R\pi_*\bZ(m)_{X/S}^\cD[2m] \otimes^\bL R\pi_*\bZ(n)_{X/S}^\cD [2n] \to R\pi_*\bZ(m+n)_{X/S}^\cD [2(m + n)] \to \bZ(1)_S^\cD [2],
\end{align}
where the first morphism is the product of Deligne cohomology relative to \(X\) over \(S\) and the second is the pushforward, whose target is quasi-isomorphic to \(\Gm^\an [1]\).
This morphism and \cref{RoughBiext} for the biextension by Bloch are intertwined by the cycle class map.

Historically, Bloch constructed the regulator map from higher Chow groups to Deligne cohomology in \cite[\S 4]{BloBei}.
It might have been well-known that it has a derived version.
However, we could not find a reference for the derived analog building upon purity and \(\bA^1\)-invariance except \cite[\S 8]{BSModulus} by Binda and Saito.
There are other explicit derived regulator maps using currents, especially by Goncharov and Kerr--Lewis--M\"uller-Stach \cite{GonReg, KLM}.
We use the regulator by Binda--Saito for our discussion.

For the application to the asymptotic behavior of Archimedean heights, we need to pin down the isomorphism between the two biextensions by way of comparing specific sections.
This means that we need to work more than just comparing the isomorphism classes of the biextensions.
In terms of homological algebra, this amounts to pinning down the homotopy when comparing \cref{RoughBiext,RoughBiextD}, unlike the original comparison that happened in the derived category.
We achieve this by taking the model of \cref{RoughBiext,RoughBiextD} using Godement resolutions.

We prove the comparison of the biextension and its sections constructed by Deligne cohomology and those by Hain in a more elementary fashion.
Our argument is a variant of that by Esnault--Vieweg \cite[Theorem 7.11]{EVDB} about equivalent definitions of Abel--Jacobi maps.
This part of the proof has a weaker version that works in the derived category, but it requires more preparations and is tangential to our argument here.

We mention relevant part of the work \cite{GorBiext} of Gorchinskiy.
He has two proofs, one close to our strategy and the other using analysis.
We comment on the former here.
First, He only uses \cite{KLM} unlike us.
Second, the argument is a weaker version that works in the derived category.
As a result, he uses the Hodge complex defined by Beilinson \cite{BeiAbs} to compare the biextensions.
We do not need these technicalities, although it is still possible to argue in the derived category without the Hodge complex.

When \(\pi\) is not smooth, the construction by Brosnan and Pearlstein of the \(\bQ\)-line bundle extending the line bundle specializes as follows.
Take generically homologically trivial cycles \(W, Z\) on \(X\) as before.
Set \(S_{\sm}\) to be an open subset of \(\pi\) over which \(\pi\) is smooth, and \(W, Z\) do not meet and have correct dimension in each fiber.
Let \(j\colon S_{\sm} \to S\) be the inclusion.
Set \(H\coloneqq R^{2m - 1}\pi_* \bZ(m)/\mathrm{tors}\), a variation of mixed Hodge structures considered over \(S_{\sm}\).
Let \(\nu, \omega\) be the image of the Abel--Jacobi maps applied to \(W, Z\).
These images are admissible normal functions over \(S_{\sm}\) with respect to its partial compactification \(S\), a notion defined by Saito \cite{SaitoANF}.
They define the sections of the intermediate Jacobians over \(S_{\sm}\).
Then roughly the subsheaf of \(j_*(B(H)_{\nu, \omega})\) consisting of bounded height around points in \(S\setminus S_{\sm}\) is a \(\bQ\)-line bundle over \(S\).

To obtain any meaningful result about the asymptotics at a point \(s\in S\), the vertical components of \(W, Z\) at \(s\) need to be determined by their horizontal components around \(s\).
To carry this out at the level of algebraic cycles, we need Hodge conjectures.
After the modification, we can show the isomorphism of the \(\bQ\)-line bundles by Seibold and those by Brosnan--Pearlstein.
In the process, we rephrase the condition of having bounded height in terms of Hodge theory of admissible variation of Hodge structures.
The key to prove the rephrased statement turns out to run the argument defining \(B_{W, Z}\), but in the correct category of perverse sheaves and mixed Hodge modules.
Then, by the comparison of sections of these line bundles that carry over from the smooth locus, our theorem is reduced to the computation of the divisor associated with the counterpart of \(B_{W, Z}\).
This computation is known by Seibold \cite[6.8 Satz]{SeiBier}.

\subsection*{Organization of this paper}
\cref{chap:Prep} is devoted to the preparations toward the following parts of the paper with the emphasis on higher Chow groups.
We recall notations and basic definitions, especially on higher Chow groups, in \cref{sec:Not}.
In \cref{sec:HomAlg}, we fix conventions around homological algebra.
In \cref{sec:MotDef} we recall the definitions and standard facts around higher Chow groups.
We give a detailed exposition to choose a convention on cubical higher Chow groups among many variants.
The next \cref{sec:Set} fixes the geometric settings for the rest of the paper here.

\cref{chap:BloBiext} starts the discussion corresponding to that of \cite{BloBiext}.
We first recall the vanishing of what is called \(\theta_W\) in the Bloch's paper \cite{BloBiext} in \cref{sec:Van}.
This result eventually allows us to modify \cref{RoughBiext} to \(\bfA^m(X/S)\otimes^\bL\bfA^n(X/S)\to \Gm[1]\) in \cref{sec:Constr}.
This gives an isomorphism class of a \(\Gm\)-biextension of the pair \((\bfA^m(X/S), \bfA^n(X/S))\) of subsheaves of generically homologically trivial elements in the sheaves of Chow groups.
Additionally, we review the calculation of the line bundle obtained after fixing two generically homologically trivial cycles in \cref{sec:Calc}.
This calculation has been done by Seibold \cite[6.8 Satz]{SeiBier}.
Also, we reformulate the construction of Bloch and Seibold \cite[Kapitel 5]{SeiBier} to construct the sections of our biextensions.

\cref{chap:Arch} compares the \(\Gm\)-biextension in \cref{chap:BloBiext} with \cite[\S 3]{HainBiext} by Hain and \cite[Theorem 241]{BPJump} by Brosnan and Pearlstein.
We begin with \cref{sec:pVMHS}, where we review pre-variations of mixed Hodge structures introduced by Fujino and Fujisawa \cite{FFVMHS>0}.
This notion helps us since the intermediate Jacobians appear as a sheaf version of \(\Ext^1\) in the category of pre-variations of mixed Hodge structures. 
In \cref{sec:Arch}, we review what Hain constructed as a biextension, and the work of Brosnan and Pearlstein that continues Hain's.
Building on the previous section, in \cref{sec:ArchCompSm}, we compare the biextensions by Bloch and Hain over \(\bC\) using the derived regulator maps of Binda and Saito.
Since Hain works when our family of projective varieties is smooth, this section pays attention only to that case.
We build on \cref{sec:ArchCompSm} to relate the constructions by Bloch and Seibold, and Brosnan--Pearlstein in \cref{sec:ArchComp}.

\subsection*{Acknowledgments}
The author thanks his advisor S.-W. Zhang for the academic support.
The author also thanks J.~Koizumi for pointing him to references regarding motivic subjects.
The author is grateful to N.~Katz and W.~Sawin for listening to or answering questions about etale cohomology.
The author thanks M.~Bars, A.~Beilinson, B.~Dirks, S.~G. Park and T.~Zhou for listening to or answering questions about Hodge theory.
The author appreciates the patience of S.~Kanda, J.~Krantz, G.~Nahm and A.~Tatsuoka who have listened to or answered questions on topology or complex geometry.
The author especially thanks G.~Nahm for isomorphisms in \cref{rem:homology}.

\part{Largely motivic preparations} \label{chap:Prep}

This part of the paper fixes notations and geometric settings.
The first \cref{sec:Not} is just a bunch of notations.
In \cref{sec:HomAlg}, we fix conventions around homological algebra.
We need this section because later in \cref{chap:Arch}, we need to upgrade various statements in the commutativity in the derived category to those that actually specify homotopy.
More specifically, \cref{lem:NoHtpy} goes into the coincidence of \cref{RoughBiext,RoughBiextD}.
Also, \cref{lem:3x3} goes into the comparison of \cref{RoughBiextD} and the biextension by Hain.
In this comparison, we perform the analogous argument to \cite[Theorem 7.11]{EVDB} that involves more cones.
To keep track of all these distinguished triangles, we need the derived version of the \(3 \times 3\) lemma.
We also include \cref{ex:HomTrivExSeq}, useful in the arguments with mixed Hodge modules.
In \cref{sec:MotDef}, we recall the definitions and standard facts around higher Chow groups.
Especially, we review the compatibilities of simplicial and cubical higher Chow groups.
The compatibilities are necessary because Bloch \cite{BloBiext} used simplicial higher Chow groups, while Binda and Saito use cubical ones.
These compatibilities and other reasons such as \cref{rem:propCalcSign} necessitate various sign conventions, including the quasi-isomorphism between \(z^1(S, -\bullet)\) and \(\Gm[1]\).

\section{General notations} \label{sec:Not}
Let \(\ol{*}\colon \bC\to \bC\) be the complex conjugation.
For a field \(k\), let \(\ol{k}\) denote an algebraic closure of \(k\).
If \(K\) is an extension of \(k\), the (transcendental) degree of the extension is written as \([K:k]\) (resp.~\(\Tr\deg_k K\)).
For a ring \(A\) and an \(A\)-module \(M\) of finite length, we decide that \(l_A(M)\) is the length of \(M\).
The symbol \(\LocSys(M)\) denotes the category of the local systems on \(M\) for a topological space \(M\).
For a map \(f\) of topological spaces, \(f_*\) and \(f^{-1}\) denote the direct and inverse image of sheaves of abelian groups.
We use similar notations for morphisms of sites.

We sometimes omit ``\(\Spec\)'' for affine schemes.
For a scheme \(T\), let \(\cO_T\) be the structure sheaf of \(T\).
For \(t\in T\), let \(k(t)\) be the residue field at \(t\).
If \(T\) is integral, then \(K(T)\) means its function field.
When \(T\) is over a field \(l\) that contains a subfield \(k\), we mean the Weil restriction of \(T\) by \(\Res_{l/k} T\).
For a cycle \(W\) on \(T\), let \(|W|\) be its support.
If \(T\) is a scheme of finite type over \(\bC\), equip \(T(\bC)\) with the classical topology.
Use the same notation for morphisms of schemes of finite type over \(\bC\) and the induced maps on the \(\bC\)-points.
When \(T\) is a complex manifold, equip it with the sheaf of the holomorphic functions \(\cO_T\).
Betti, Deligne and etale cohomology theories are denoted by \(H^\star\), \(H^\star_{\cD}\) and \(H^\star_\et\).
The notation \(H_\star\) stands for singular homology as discussed in \cref{sec:HomAlg}.

For a manifold \(M\) with boundary, let \(\partial M\) be the boundary.

\section{Homological algebra} \label{sec:HomAlg}

We first set the general notations for homological algebra.
Let \(\Ab\) be the category of abelian groups.
We mostly use cohomological notations for complexes.
For an integer \(p\) and a cohomological complex \(K\) in an abelian category, \(Z^p(K)\) (resp.~\(H^p(K)\)) denotes the \(p\)-th group of cocycles (resp.~cohomology group) of \(K\).
For a complex \(K\), we write \(\partial\), \(\partial_K\) or \(\partial_K^n\colon K^n\to K^{n+1}\) when \(K\) is cohomological, for the differential of \(K\).
If \(K\) is cohomological, \(\tau_{<n} A\), \(\tau_{\leqq n} A\) and \(\tau_{\geqq n} K\) (resp.~\(\sigma_{\leqq n} K\) and \(\sigma_{\geqq n} K\)) mean the canonical (resp.~brutal) truncations for an integer \(n\).
If \(K\to L\) is a morphism of cohomological complexes, then \(\Cone (K\to L)\) denotes the cone of the morphism \cite[p.~26]{HartsRes}.
In general, our sign conventions follow \cite{HartsRes}.
The derived tensor products of cohomological complexes are denoted by \(\otimes^\bL\).
We use the same notation for a double complex and its total complex.
For an abelian category \(\sC\), let \(\Kom(\sC), K(\sC), D(\sC)\) be the category, homotopy category and derived category of the complexes with terms in \(\sC\).
Especially, if \(\sC\) is the category of sheaves of \(\cO_T\)-modules for a ringed space \(T\), then the corresponding category of complexes, and homotopy and derived categories are denoted by \(\Kom(\cO_T)\), and \(K(\cO_T)\) and \(D(\cO_T)\).
If \(\sC\) is a category, let \(\Ind (\sC)\) be the category of ind-objects in \(\sC\).

We fix the convention around singular homology and Poincar\'e duality.
Let \(M\) be a compact oriented manifold.
Let \((C_\bullet(M), \partial)\) (resp.~\(C^\bullet(M)\)) be the homological (resp.~cohomological) complex of singular chains (resp.~cochains) on \(M\) with \(\bZ\)-coefficients.
Define the cohomological complex \({}'C^\bullet (M)\) with terms \(C_{-\bullet} (M)\) and differentials \((-1)^{q + 1} \partial \colon C_{-q} (M)\to C_{-q-1} (M)\).
This follows \cite[\S 2.3]{KLM}, not \cite{JannDel}.
Then the cap product at the end of \cite[p.~239]{Hatcher} with \(X\coloneqq M\) becomes a morphism of complexes.
The same is true for other cap products in \emph{loc. cit.}.
Now, suppose that \(\dim M\) is even.
Then the orientation of \(M\) gives
\[
    C^\bullet(M)\to {}'C^\bullet(M),
\]
which defines the Poincar\'e duality.
We also use variants of this construction.

\subsection*{Quasi-isomorphisms as a multiplicative system}
We review some of the fact that quasi-isomorphisms form a multiplicative system in the homotopy category of complexes.
The purpose is not to recall this fact, but to reveal the homotopy that appears in the proof.
We apply the content of this passage in a later passage in \cref{sec:ArchCompSm}, where we compare sections of biextensions, preserving the relationships of those sections by keeping track of various homotopy.

We discuss more specifically.
Let \(\sC\) be an abelian category.
Suppose that we have morphisms \(A\to B, A\to C\) of complexes with terms in \(\sC\).
We assume that \(A\to B\) is a quasi-isomorphism.
The fact that quasi-isomorphisms constitute a multiplicative system includes the following statement: there exists a diagram
\[
    \begin{CD}
        A @>>> B \\
        @VVV @VVV \\
        C @>>> D
    \end{CD}
\]
commutative up to homotopy, where \(C\to D\) is a quasi-isomorphism.
One possible \(D\) is 
\[
    D \coloneqq \Cone(\Cone(A\to B)[-1] \to A\to C).
\]
This construction invokes a canonical morphism \(C\to D\).
Also, we can give \(B\to D\) by
\[
    (0, -\id, 0)\colon B^p\to C^p\oplus B^p\oplus A^{p + 1} = D^p
\]
for each integer \(p\).

The point is that this construction is functorial in the category of complexes, not just in the homotopy category.
We record a precise formulation just in case.
\begin{lem} \label{lem:NoHtpy}
Take quasi-isomorphisms \(A\to B, A'\to B'\) and morphisms \(A\to C, A'\to C'\).
Let \(A\to A', B\to B', C\to C'\) be morphisms of complexes such that
\[
\begin{CD}
    B @<<< A @>>> C \\
    @VVV @VVV @VVV \\
    B' @<<< A' @>>> C'
\end{CD}
\]
commutes in the category of complexes, thereby needing no homotopy.
Put \(D'\coloneqq \Cone(\Cone(A\to B)[-1] \to A\to C)\).
Then
\[
\begin{CD}
    B @>>> D @<<< C \\
    @VVV @VVV @VVV \\
    B' @>>> D' @<<< C'
\end{CD}
\]
commutes as a diagram of complexes, again without homotopy.
\end{lem}

\subsection*{Derived \(3\times 3\)-lemma}
Here, we review an explicit version of \(3\times 3\)-lemma.

\begin{lem} \label{lem:3x3}
Let \(\sC\) be an abelian category as before.
Let
\[
\begin{CD}
    A_1 @>>> B_1 \\
    @VVV @VVV \\
    A_2 @>>> B_2
\end{CD}
\]
be a commutative diagram in the category of the complexes with terms in \(\sC\).
Then an isomorphism of \(\Cone(\Cone(A_1\to B_1) \to \Cone(A_2\to B_2))\) and \(\Cone(\Cone(A_1\to A_2) \to \Cone(B_1\to B_2))\) is given by
\[
    A_1^{p + 2}\oplus B_1^{p+1}\oplus A_2^{p+1}\oplus B_2^p\to A_1^{p + 2}\oplus A_2^{p+1}\oplus B_1^{p+1}\oplus B_2^{p}, (a, b, c, d)\mapsto (-a, c, b, d).
\]
for each integer \(p\).
\end{lem}

\subsection*{Distinguished triangles out of cohomology long exact sequences}
Again, let \(\sC\) be an abelian category.
We examine how long exact sequences obtained from distinguished triangles behave as triangles in \(D(\sC)\).
\begin{ex} \label{ex:HomTrivExSeq}
    Suppose that \(X\to Y\to Z\xrightarrow{+1}\) is a distinguished triangle in \(D(\sC)\).
    Let \(A\in \sC\).
    Take \(A\to X\in D(\sC)\) such that \(A\to H^0(X)\to H^0(Y)\) vanishes.
    By the last assumption, \(A\to X\to Y\) factors through \(A\to \tau_{<0} Y\).
    There exist morphisms
    \[
    \begin{CD}
        X @>>> Y @>>> Z @>{+1}>> \\
        @AAA @AAA @AAA \\
        A @>>> \tau_{<0} Y @>>> Z' @>{+1}>> \\
        @|      @VVV            @VVV \\
        A @>{f}>> H^{-1}(Y)[1] @>>> Z'' @>{+1}>>
    \end{CD}
    \]
    of distinguished triangles.
    The cohomology of the last row tells us that \(Z'' = H^{-1}(Z'')[1]\).
    These triangles induce commuting exact sequences
    \[
        \begin{CD}
            H^{-1}(X) @>>> H^{-1}(Y) @>>> H^{-1}(Z) @>>> H^0(X) @>>> H^0(Y) \\
            @.              @|              @AAA         @AAA   @. \\
            0         @>>> H^{-1}(Y) @>>> H^{-1}(Z') @>>> A @>>> 0 \\
            @.              @|              @VVV         @|     @. \\
            0         @>>> H^{-1}(Y) @>>> H^{-1}(Z') @>>> A @>>> 0.
        \end{CD}
    \]
    By the top and bottom rows, the pull-back of the short exact sequence
    \[
        0 \to H^{-1}(Y)/H^{-1}(X) \to H^{-1}(Z) \to H^0(X) \to H^0(Y)
    \]
    by \(A\to H^0(X)\) is isomorphic to the pushforward of the exact sequence corresponding to \(f\in\Ext^1(A, H^{-1}(Y))\).
\end{ex}

\section{Recollections on higher Chow groups} \label{sec:MotDef}

This section introduces the notation on higher Chow groups.
We first discuss the original simplicial definition.
For a nonnegative integer \(q\), let
\begin{align} \label{Delta^q}
    \Delta^q\coloneqq \Spec \bZ[T_0,\dots, T_q]/\left(\sum_{i = 0}^{q}T_i - 1\right).
\end{align}
For a scheme T, the \emph{faces} of \(\Delta^q_T\) are closed subschemes defined by equations of the form \(T_i = 0\) for \(i\in S\) using a subset \(S\subseteq \{0,1,\dots q\}\).

For an equidimensional scheme \(T\) and a nonnegative integer \(p,q\), the symbol \(\CH^p(T,q)\) denotes the higher Chow group defined as below.
Let \(z^p(T,q)\) be the abelian group of all the cycles of codimension \(p\) on \(\Delta^q_T\) that intersect properly with its faces.
They form a homological complex by carrying an integral closed subscheme \(Z\in z^p(T,\bullet)\) to
\[
    \partial Z\coloneqq\sum_{i = 0}^{q}(-1)^i Z\cap V(T_i=0),
\]
where \(V\) means the vanishing locus inside \(\Delta^q_T\).
The \(q\)-th homology of the complex is \(\CH^p(T,q)\).
We use the cohomological notation \(z^p(T,-\bullet)\) rather than the homological complex.
If \(Z\) vanishes by the differential, then let \([Z]\) denote the corresponding class in \(\CH^p(T,q)\).

\begin{ex} \label{ex:CH1(1)}
We have the isomorphism
\[
    \cO_T(T)^\times\simeq\CH^1(T, 1), \quad f\mapsto [V(T_0 + fT_1)]
\]
when \(T\) is uni-dimensional and regular, with \(V\) meaning the vanishing locus.
We fix this identification throughout the note.
\end{ex}

\begin{ex} \label{codim1}
The group \(\CH^p(T)\coloneqq\CH^p(T,0)\) is the usual Chow group.
We review a relevant construction in a simple case.
Suppose in this paragraph that \(T\) is regular, uni-dimensional and connected and that \(p = 1\).
We have the following construction of a morphism of cohomological complexes from \(z^1(T, -\bullet)\) to the variant
\[
    K_{1,T}\colon\dots \to 0\to K(T)^\times\xrightarrow{\mathrm{div}} z^1(T,0)\to 0 \to \dots
\]
of the Gersten complex, where \(K(T)^\times\) is the \(-1\)-st term.
The morphism is the identity on the \(0\)-th term.
On the \(-1\)-st term, we use
\[
    z^1(T,1)\to z^1(T,1)/\partial z^1(T,2)\to z^1(\Spec K(T), 1)/\partial z^1(\Spec K(T), 2) = K(T)^\times,
\]
whose last equality is \cref{ex:CH1(1)}.
The morphism of the complexes is a quasi-isomorphism.
\end{ex}

For the description of the intersection product, we introduce variants of the above complex.
Suppose \(W\) is a finite set of nonempty closed subsets of \(T\).
Then define \(z^p_W(T,q)\) to be the abelian group freely generated by all the integral subschemes \(Z\) of \(\Delta^q_T\) that belong to \(z^p(T,q)\) and intersect properly with all the elements in \(W\).
We also need sheafified version of these.
Let \(z^p(*_T,q)\) be the small Zariski sheaf
\[
    T\supseteq U\mapsto z^p(U,q)
\]
on \(T\).
Define the subsheaf \(z^p_W(*_T,q)\subseteq z^p(*_T,q)\) by
\[
    T\supseteq U\mapsto z^p_{W\cap U}(U,q),
\]
where \(W\cap U\coloneqq \{W_0\cap U\mid W_0\in W\}.\)

We need more notations for the intersection product.
Let \(z^p(T,q,r)\) be the set of the elements in \(z^p(T\times_\bZ \Delta^q \times_\bZ \Delta^r, 0)\) that intersect transversally with \(T\times_\bZ F \times_\bZ G\) for all the faces \(F\subseteq \Delta^q\) and \(G\subseteq \Delta^r\).
Set \(z^p(T,q,r)_\Delta\) be the set of the elements in \(z^p(T\times_\bZ \Delta^q \times_\bZ \Delta^r, 0)\) that intersect transversally with \(T\times_\bZ F\) for all the faces \(F\subseteq \Delta^q\times_\bZ \Delta^r\).
These form double complexes.
Check \cite[\S 8]{GLBK} for details.
We define sheaf versions \(z^p(*_T,q,r)\), \(z^p(*_T,q,r)_\Delta\) and their double complexes similarly.

If \(T\) is a regular equidimensional scheme of finite type over a field of characteristic \(p\geqq 0\), then the etale cycle class map is
\[
    *_\et\colon \CH^p(T,q)\to\prod_{\ell\neq p} H^{2p-q}_\et(T, \bZ_\ell(p))
\]
for nonnegative integers \(p,q\) and a prime number \(\ell\neq p\).
When \(T\) is a smooth quasi-projective scheme over \(\bC\), then we also have the cycle class maps
\[
    *_\sing \colon \CH^p(T, q)\to H^{2p-q}(T(\bC), \bZ(p)), \quad *_\cD \colon \CH^p(T,q)\to H^{2p-q}_\cD(T(\bC), \bZ(p))
\]
for Betti and Deligne cohomology \cite{BloBei}.

For \cref{sec:ArchCompSm}, we recall cubical higher Chow groups.
Put \(\square\coloneqq \bP^1_{\bZ}\backslash \{1\}\).
If \(T\) is a scheme, \emph{faces} of \(\square^q_T\) are closed subschemes of the form
\[
    T\times \prod_{i\in \{1,2,\dots,q\}} \epsilon(i)
\]
for \(\epsilon\colon \{1,2,\dots,q\}\to \{\{0\},\{\infty\}, \square\}\).
For an equidimensional scheme \(T\) of finite type over a field and a nonnegative integer \(p,q\), let \(\underline{z}_c^p(T,q)\) be the abelian group of the cycles \(Z\) on \(\square^q_T\) of codimension \(p\) that intersect properly with all the faces of \(\square^q_T\).
This defines a cubical object \cite[\S 1.2]{LevCub} of the category of abelian groups by associating \(0, \infty\in \square\) with \(0\) and \(1\) in \emph{loc. cit.}
Let \(z_c^p(T,\bullet)\) be the homological complex obtained by the construction of \cite[Definition 1.4]{LevCub}.
This is quasi-isomorphic to \(z_c^p(T,\bullet)_0\), a complex denoted by \(\underline{A}^0\) around \cite[Lemma 1.3]{LevCub} for \(\underline{A} = \underline{z}_c^p(T,\bullet)\).
We call the elements of the terms of \(\Ker(\underline{z}_c^p(T,\bullet)\to z_c^p(T,\bullet))\) degenerate cycles.
Our convention regarding the sign of the differential is the same as that of \cite{BSModulus, TotMilnorK} unlike \cite[\S 5.2]{KLM}\footnote{Our sign is used in \S 2.1 of the same paper, but probably we should ignore this.} or \cite[p.~276]{LevRevisit}.
As before, we prefer the cohomological notation \(z_c^p(T,-\bullet)\) to the homological notation.
Also, we have sheafified versions \(z_c^p(*_T, -\bullet)\) and \(z_c^p(*_T, -\bullet)_0\) as well.

A zigzag of quasi-isomorphisms between \(z^p(T, -\bullet)\) and \(z_c^p(T, -\bullet)\) is well-known (\cf \cite[Theorem 4.7]{LevRevisit}\footnote{The author suspects that \(\epsilon'\) and \(\epsilon''\) in \cite[p.~294]{LevRevisit} should be the other way around.}, \cite[Theorem 4.3]{BloCub}) for \(T\) of finite type over a field.
In this case, if \(\partial Z = 0\) for \(Z\in z_c^p(T,\bullet)\), then let \([Z]\in \CH^p(T,q)\) be the associated class by the following convention of the quasi-isomorphisms.

For an equidimensional scheme \(T\) of finite type over a field, let \(z^p(T, \Delta^a, \square^b)\) be the set of \(Z\in z^p(T\times \Delta^a\times \square^b, 0)\) whose intersection with \(T\times F\times G\) for all the faces \(F\subseteq \Delta^a\) and \(G\subseteq \square^b\) is proper, and is \(0\) when \(F = \Delta^a\) and the projection of \(G\) to some coordinate gives \(\{\infty\}\).
We have a total complex \(z^p(T, \Delta^{-\bullet}, \square^{-\bullet})\).
By the isomorphism from \(\bA^1\) to \(\square\) given by \(z\mapsto 1-z^{-1}\), the proof of \cite[Theorem 4.3]{BloCub} shows that the inclusions
\begin{align} \label{Delta2Mix}
    z^p(T, -\bullet)\to z^p(T, \Delta^{-\bullet}, \square^{-\bullet})
\end{align}
and
\begin{align} \label{Square2Mix}    
    z_c^p(T, -\bullet)_0\to z^p(T, \Delta^{-\bullet}, \square^{-\bullet})
\end{align}
are quasi-isomorphic.

\begin{ex} \label{ex:CH1(1)Mix}
Let \(p = 1\) in this paragraph.
When \(T\) is the spectrum of a field, \cite[Theorem 1]{TotMilnorK} says that the cubical higher Chow group \(\CH^1(T,1)\) is isomorphic to \(\cO_T(T)^\times\) by carrying \(a\in \cO_T(T)^\times\backslash\{1\}\) to \(\{a\}\in \underline{z}_c^1(T, 1)\).
This is compatible with \cref{ex:CH1(1),Delta2Mix,Square2Mix} by what follows.
For each \(a\in \cO_T(T)^\times\), consider
\[
    V((1-a^{-1})T_0 + a^{-1}X - 1)\in z^1(T, \Delta^1, \square^1),
\]
with \(T_0\) as in \cref{Delta^q} and the coordinate \(X\) for \(\square\backslash\{\infty\}\subset \bP^1\backslash\{\infty\} = \bA^1\).
The differential of the total complex carries this to 
\[
    (\{a\}, -V(T_0 + aT_1))\in z^1(T, \Delta^0, \square^1)\oplus z^1(T, \Delta^1, \square^0).
\]
This means that the element of the cubical higher Chow group defined by \(\{a\}\in \underline{z}_c^1(T, 1)\) coincides with \([V(T_0 + aT_1)]\in \CH^1(T, 1)\) when seen as the element of the \(-1\)-st cohomology of \(z^1(T, \Delta^0, \square^1)\).

More generally, suppose that \(T\) is a regular scheme of finite type over a field.
Then for each \(1\neq f\in \cO_T(T)^\times\), we can associate the closed subscheme of \(\square_T\) of codimension 1 by taking the base change of \(\{f\}\subset \bP^1_T\).
This defines an element of \(z_c^1(T, 1)\) by restricting the embedding \(T\xrightarrow{f} \bP^1_T\) to \(\square_T\).
By the discussion over \(K(T)\) above, this element coincides with the image of \(f\) under \cref{ex:CH1(1)} via \cref{Delta2Mix,Square2Mix}.
\end{ex}

We review that various constructions on these motivic complexes are compatible with \cref{Delta2Mix,Square2Mix}.
First, we discuss the intersection product.
This is briefly mentioned in \cite[p.~181]{TotMilnorK}.
We first define the product structure on \(z^p(T, \Delta^{-\bullet}, \square^{-\bullet})\).

For that, we need an exterior product.
Let \(T\) be an equidimensional scheme of finite type over a field \(k\).
Let \(z^p(T, \Delta^{a}, \Delta^{a'}, \square^{b})\) be the set of \(Z\in z^p(T\times \Delta^a\times \Delta^{a'}\times \square^b, 0)\) whose intersection with \(T\times F\times F'\times G\) for all the faces \(F\subseteq \Delta^a\), \(F'\subseteq \Delta^{a'}\) and \(G\subseteq \square^b\) is proper, and is \(0\) when \(F=\Delta^a\), \(F'=\Delta^{a'}\) and the projection of \(G\) to some coordinate gives \(\{\infty\}\).
Again, \(z^p(T, \Delta^{-\bullet}, \Delta^{-\bullet}, \square^{-\bullet})\) means the total complex.
If \(T'\) is also an equidimensional scheme of finite type over \(k\), then for nonnegative integers \(p,p'\) we have
\[
    z^p(T, \Delta^{-\bullet}, \square^{-\bullet})\otimes_\bZ z^{p'}(T', \Delta^{-\bullet}, \square^{-\bullet})\to z^{p + p'}(T\times_k T', \Delta^{-\bullet}, \Delta^{-\bullet}, \square^{-\bullet})
\]
by \(Z\otimes Z'\mapsto (-1)^{a'b}Z\times_k Z'\) for \(Z\in z^p(T, \Delta^{a}, \square^{b})\) and \(Z'\in z^{p'}(T', \Delta^{a'}, \square^{b'})\).
It induces the following commutative diagrams:
\[
    \begin{CD}
        z^p(T, -\bullet)\otimes_\bZ z^{p'}(T', -\bullet)@>>> z^{p + p'}(T\times_k T', -\bullet, -\bullet)\\
        @V{\cref{Delta2Mix}}VV    @V{\cref{Delta2Mix}}VV\\
        z^p(T, \Delta^{-\bullet}, \square^{-\bullet})\otimes_\bZ z^{p'}(T', \Delta^{-\bullet}, \square^{-\bullet})@>>> z^{p + p'}(T\times_k T', \Delta^{-\bullet}, \Delta^{-\bullet}, \square^{-\bullet})\\
        @A{\cref{Square2Mix}}AA    @A{\cref{Square2Mix}}AA\\
        z_c^p(T, -\bullet)_0\otimes_\bZ z_c^{p'}(T', -\bullet)_0@>>> z_c^{p + p'}(T\times_k T', -\bullet)_0.
    \end{CD}
\]

Define \(z^p(T, \Delta^{a}, \Delta^{a'}, \square^{b})_\Delta\) to be the set of \(Z\in z^p(T\times \Delta^a\times \Delta^{a'}\times \square^b, 0)\) whose intersection with \(T\times F\times G\) for faces \(F\subseteq \Delta^a\times\Delta^{a'}\) in the sense of \cite[\S 8]{GLBK} and faces \(G\subseteq \square^b\) is proper, and is \(0\) when \(F = \Delta^a\times\Delta^{a'}\) and the projection of \(G\) to some coordinate gives \(\{\infty\}\).
\begin{lem}
    Let \(T\) be a equidimensional scheme of finite type over a field.
    Set \(p\) to be a nonnegative integer.
    There the inclusion
    \[
        z^p(T, \Delta^{-\bullet}, \Delta^{-\bullet}, \square^{-\bullet})_\Delta \to z^p(T, \Delta^{-\bullet}, \Delta^{-\bullet}, \square^{-\bullet})
    \]
    is a quasi-isomorphism.
\end{lem}

\begin{proof}
    As in the proof of \cite[Lemma 8.1]{GLBK}, for nonnegative integers \(a,a'\), it suffices to show that
    \[
        z^p(T, \Delta^a, \Delta^{a'}, \square^\bullet)_\Delta\to z_c^p(T\times\Delta^a\times\Delta^{a'}, \bullet)_0
    \]
    and
    \[
        z^p(T, \Delta^a, \Delta^{a'}, \square^\bullet)\to z_c^p(T\times\Delta^a\times\Delta^{a'}, \bullet)_0
    \]
    are quasi-isomorphisms.
    We can use \cite[Lemma 4.3.1]{BloCub} via the isomorphism \(\bA^1\to\square\) given by \(z\mapsto 1-z^{-1}\).
\end{proof}

Let \(T\) be an equidimensional scheme of finite type over a field.
Take a nonnegative integer \(p\).
In the following commutative diagram, all the arrows are inclusions:
\[
\begin{CD}
    z^{p}(T, -\bullet, -\bullet)@<<< z^{p}(T, -\bullet, -\bullet)_\Delta\\
    @VVV   @VVV\\
    z^p(T, \Delta^{-\bullet}, \Delta^{-\bullet}, \square^{-\bullet})@<<< z^p(T, \Delta^{-\bullet}, \Delta^{-\bullet}, \square^{-\bullet})_\Delta\\
    @AAA   @AAA\\
    z_c^{p}(T, -\bullet)_0@= z_c^{p}(T, -\bullet)_0.
\end{CD}
\]

Finally, we define the triangulation
\[
    z^p(T, \Delta^{-\bullet}, \Delta^{-\bullet}, \square^{-\bullet})_\Delta\to z^p(T, \Delta^{-\bullet}, \square^{-\bullet})
\]
as in the paragraph of \cite{GLBK} just before Lemma 8.1.
This is compatible with the simplicial triangulation, namely
\[
    \begin{CD}
        z^p(T, -\bullet, -\bullet)_\Delta @>>> z^p(T,-\bullet, -\bullet)\\
        @VVV   @VVV\\
        z^p(T, \Delta^{-\bullet}, \Delta^{-\bullet}, \square^{-\bullet})_\Delta @>>> z^p(T, \Delta^{-\bullet}, \square^{-\bullet}).
    \end{CD}
\]
is commutative.
So is the following:
\[
    \begin{CD}
        z_c^{p}(T, -\bullet)_0@= z_c^{p}(T, -\bullet)_0\\
        @VVV   @VVV\\
        z^p(T, \Delta^{-\bullet}, \Delta^{-\bullet}, \square^{-\bullet})_\Delta @>>> z^p(T, \Delta^{-\bullet}, \square^{-\bullet}).
    \end{CD}
\]

To sum up, the exterior products defined so far are compatible with each other:
\begin{prop}
    Take equidimensional schemes \(T, T'\) of finite type over a field \(k\).
    Let \(p, p'\) be nonnegative integers.
    We have a commutative diagram
    \[
        \begin{CD}
            z^p(T, -\bullet)\otimes_\bZ z^{p'}(T', -\bullet)@>>> z^{p + p'}(T\times_k T', -\bullet)\\
            @V{\cref{Delta2Mix}}VV    @V{\cref{Delta2Mix}}VV\\
            z^p(T, \Delta^{-\bullet}, \square^{-\bullet})\otimes_\bZ z^{p'}(T', \Delta^{-\bullet}, \square^{-\bullet})@>>> z^{p + p'}(T\times_k T', \Delta^{-\bullet}, \square^{-\bullet})\\
            @A{\cref{Square2Mix}}AA    @A{\cref{Square2Mix}}AA\\
            z_c^p(T, -\bullet)_0\otimes_\bZ z_c^{p'}(T', -\bullet)_0@>>> z_c^{p + p'}(T\times_k T', -\bullet)_0
        \end{CD}
    \]
    in the derived category of abelian groups.
\end{prop}
A sheafified version of this is also true.

Toward the compatibility of the products, we discuss that of pull-backs.
First, we discuss the irrelevant case of a flat morphism \(f\colon T\to T'\) of equidimensional schemes of finite type over a field.
Let \(p\) be a nonnegative integer.
By a similar argument to the simplicial case, pull-backs induce
\[
    f^*\colon z_c^{p}(T', -\bullet)\to z_c^{p}(T, -\bullet),
\]
\[
    f^*\colon z_c^{p}(T', -\bullet)_0\to z_c^{p}(T, -\bullet)_0,
\]
\[
    f^*\colon z^{p}(T', \Delta^{-\bullet}, \square^{-\bullet})\to z^{p}(T, \Delta^{-\bullet}, \square^{-\bullet}).
\]
The first two of these are obviously compatible.
Also clearly the following holds.
\begin{prop}
    Let \(f\colon T\to T'\) be a flat morphism of equidimensional schemes of finite type over a field.
    Set \(p\) to be a nonnegative integer.
    The diagram
    \[
        \begin{CD}
            z^{p}(T, -\bullet)@>{f^*}>> z^{p}(T', -\bullet)\\
            @V{\cref{Delta2Mix}}VV   @V{\cref{Delta2Mix}}VV\\
            z^{p}(T, \Delta^{-\bullet}, \square^{-\bullet})@>{f^*}>> z^{p}(T', \Delta^{-\bullet}, \square^{-\bullet}) \\ 
            @A{\cref{Square2Mix}}AA    @A{\cref{Square2Mix}}AA\\
            z_c^{p}(T, -\bullet)_0@>{f^*}>> z_c^{p}(T', -\bullet)_0
        \end{CD}
    \]
    commutes.
\end{prop}

More relevant is the case of a regular closed immersion.
We need more notations.
Let \(T\) be an equidimensional scheme of finite type over a field.
Take nonnegative integers \(p, q\).
For a finite set \(W\) of closed subschemes of \(T\), define \(\underline{z}_{c,W}^p(T, q)\) to be the abelian group of the cycles \(Z\) on \(\square^q_T\) of codimension \(p\) that intersect properly with all the faces of \(\square^q_{W_0}\), where \(W_0 = T\) or \(W_0\in W\).
This turns into a cubical object of the category of abelian groups, producing \(z_{c,W}^p(T, -\bullet)\) and \(z_{c,W}^p(T, -\bullet)_0\) as before.
Also, for nonnegative integers \(a,b\), set \(z_{W}^p(T, \Delta^a, \square^b)\) to be the set of \(Z\in z^p(T\times \Delta^a\times \square^b, 0)\) whose intersection with \(W_0\times F\times G\) for all \(W_0\in W\cup \{T\}\), all the faces \(F\subseteq \Delta^a\) and \(G\subseteq \square^b\) is proper, and is \(0\) when \(W_0 = T\), \(F = \Delta^a\) and the projection of \(G\) to some coordinate gives \(\{\infty\}\).
We have a total complex \(z_{W}^p(T, \Delta^{-\bullet}, \square^{-\bullet})\).

Let \(f\colon T\to T'\) be a regular closed immersion of pure codimension.
Here, we assume that \(T'\) is smooth and quasi-projective over a field.
Then \cite[Theorem 1.10]{KLAddHighChow} says that the bottom horizontal arrow of the commutative diagram
\[
\begin{CD}
    z^p(T', -\bullet) @<<< z^p_{\{T\}}(T', -\bullet) \\
    @V{\cref{Delta2Mix}}VV @V{\cref{Delta2Mix}}VV \\
    z^p(T', \Delta^{-\bullet}, \square^{-\bullet}) @<<< z^p_{\{T\}}(T', \Delta^{-\bullet}, \square^{-\bullet}) \\
    @A{\cref{Square2Mix}}AA @A{\cref{Square2Mix}}AA \\
    z_{c}^p(T', -\bullet)_0 @<<< z_{c,\{T\}}^p(T', -\bullet)_0
\end{CD}
\]
is quasi-isomorphic, where the horizontal arrows are inclusions.
Also, by \cite[Theorem 4.7]{LevRevisit} with the arguments in the proof of \cite[Theorem 4.4.2]{BloCub} and the Dold--Kan isomorphism, the right column of the diagram is made of quasi-isomorphisms.
Therefore, each arrow in the diagram is a quasi-isomorphism.

Finally, we have the commutative diagram
\[
\begin{CD}
    z^p_{\{T\}}(T', -\bullet) @>>> z^p(T, -\bullet)\\
    @V{\cref{Delta2Mix}}VV @V{\cref{Delta2Mix}}VV \\
    z^p_{\{T\}}(T', \Delta^{-\bullet}, \square^{-\bullet}) @>>> z^p(T, \Delta^{-\bullet}, \square^{-\bullet}) \\
    @A{\cref{Square2Mix}}AA @A{\cref{Square2Mix}}AA \\
    z_{c,\{T\}}^p(T', -\bullet)_0 @>>> z_{c}^p(T, -\bullet)_0,
\end{CD}
\]
where we explain the horizontal arrows now.
They carry an integral closed subscheme \(Z\) of \(\Delta^q_{T'}\), \(T'\times \Delta^a\times \square^b\), \(\square^q_{T'}\) respectively that belongs to the source to
\[
    \sum_{t} \sum_{i = 0}^{\infty} (-1)^i l_{\cO_{Z, t}}(\Tor_i^{\cO_{T',\pr (t)}}(\cO_{T,\pr (t)}, \cO_{Z,t}))\ol{\{t\}},
\]
where \(t\) runs through the maximal points of \(T\times_{T'} Z\) and \(\pr\colon \Delta^q_{T'}~(\mathrm{resp.}~T'\times \Delta^a\times \square^b, \square^q_{T'})\to T'\), with the following remarks.
First, the infinite sum is finite.
Second, these indeed are morphisms of complexes.

The following summary has a sheafified variant.
\begin{prop}
    Let \(f\colon T\to T'\) be a regular closed immersion of pure codimension.
    Here, we assume that \(T'\) is smooth and quasi-projective over a field.
    Set \(p\) to be a nonnegative integer.
    The diagram
    \[
        \begin{CD}
            z^{p}(T', -\bullet)@>{f^*}>> z^{p}(T, -\bullet)\\
            @V{\cref{Delta2Mix}}VV   @V{\cref{Delta2Mix}}VV\\
            z^{p}(T', \Delta^{-\bullet}, \square^{-\bullet})@>{f^*}>> z^{p}(T, \Delta^{-\bullet}, \square^{-\bullet}) \\ 
            @A{\cref{Square2Mix}}AA    @A{\cref{Square2Mix}}AA\\
            z_c^{p}(T', -\bullet)_0@>{f^*}>> z_c^{p}(T, -\bullet)_0
        \end{CD}
    \]
    commutes in the derived category of abelian groups.
\end{prop}

We proceed to the compatibility regarding the pushforward.
Let \(f\colon T\to T'\) be a proper morphism of equidimensional schemes of finite type over a field.
Let \(p\) be a nonnegative integer.
Similarly to the existence of \(f_*\colon z^{p + \dim T - \dim T'}(T, -\bullet)\to z^p(T', -\bullet)\), push-forwards induce
\[
    f_*\colon z_c^{p + \dim T - \dim T'}(T, -\bullet)\to z_c^p(T', -\bullet),
\]
\[
    f_*\colon z_c^{p + \dim T - \dim T'}(T, -\bullet)_0\to z_c^p(T', -\bullet)_0,
\]
\[
    f_*\colon z^{p + \dim T - \dim T'}(T, \Delta^{-\bullet}, \square^{-\bullet})\to z^p(T', \Delta^{-\bullet}, \square^{-\bullet})
\]
by relying on the intersection product of Cartier divisors and cycles.
The compatibility of the first two of these is obvious.
The following is clear as well.
\begin{prop}
    Let \(f\colon T\to T'\) be a proper morphism of equidimensional schemes of finite type over a field.
    Set \(p\) to be a nonnegative integer.
    The diagram
    \[
        \begin{CD}
            z^{p + \dim T - \dim T'}(T, -\bullet)@>{f_*}>> z^p(T', -\bullet)\\
            @V{\cref{Delta2Mix}}VV   @V{\cref{Delta2Mix}}VV\\
            z^{p + \dim T - \dim T'}(T, \Delta^{-\bullet}, \square^{-\bullet})@>{f_*}>> z^p(T', \Delta^{-\bullet}, \square^{-\bullet}) \\
            @A{\cref{Square2Mix}}AA    @A{\cref{Square2Mix}}AA\\
            z_c^{p + \dim T - \dim T'}(T, -\bullet)_0@>{f_*}>> z_c^{p}(T', -\bullet)_0
        \end{CD}
    \]
    commutes.
\end{prop}

\section{Settings} \label{sec:Set}

We use the settings here throughout the rest of the note.
Let \(k\) be a field.
Set \(S\) to be a connected smooth separated quasi-compact \(k\)-scheme.
Let \(X\) be a smooth \(k\)-scheme endowed with a flat projective morphism \(\pi\colon X\to S\) of relative dimension \(d\) that is generically smooth on \(S\).
Bloch \cite{BloBiext} considers the case when \(X\) is smooth over \(S\).

We mean by \(\bfA^p(X/S)\) the small Zariski sheaf on \(S\) for nonnegative integers \(p\) obtained as the sheafification of the presheaf that carries an open subset \(U\subseteq S\) to
\[
    \{[Z]\in \CH^p(\pi^{-1}(U))\mid Z\in z^p(\pi^{-1}(U), 0) \text{ is generically homologically trivial}\},
\]
where generically homologically trivial cycles stand the following.
Take a nonempty open subset \(S_{\sm}\subseteq S\) such that \(\pi\) restricted to \(\pi^{-1}(S_{\sm})\) is smooth.
Let \(U\subseteq S\) be an open subset.
Set \(Z\in z^p(X\times_S U,0)\).
When \(U\subseteq S_{\sm}\), this \(Z\), or \([Z]\in \CH^p(X\times_S U)\), is said homologically trivial if \([Z]_{\et} = 0\in H^{2p}_\et(X_s,\bZ_\ell)\) for every geometric point \(s\to U\) and every prime \(\ell\) different from the characteristic of \(k\).
In general, we say \(Z\) and \([Z]\) are generically homologically trivial if they are homologically trivial after restricted to \(\pi^{-1}(U\cap S_{\sm})\).
The last definition does not depend on \(S_{\sm}\).

\begin{rem} \label{rem:BettiEt}
    When \(k=\bC\), the definition of homologically trivial cycles using the etale cycle class map is equivalent to that using the Betti cycle class map for \(\Spec \bC = s\to U\) as above.
\end{rem}

\part{Cycles and biextensions revisited} \label{chap:BloBiext}

We start the actual discussion regarding various biextensions.
This part of the paper recovers the results in \cite{BloBiext} in a more systematic way described in the introduction.

As we mentioned there, we need \cref{lem:theta} in \cref{sec:Van} to reduce \cref{RoughBiext} to some derived morphism that actually correspond to the biextension by Bloch.
We record the actual construction of the derived morphism in \cref{sec:Constr}.
\cref{lem:HomAlg}, a statement purely in homological algebra, largely controls the process of the reduction using \cref{lem:theta}.
We illustrate the relationship of our discussion with the biextension by Bloch in \cref{rem:Bloch}.

If we just needed the isomorphism class of the biextension of the pair \((\pi_* z^m(*_X, 0), \pi_* z^n(*_X, 0))\), then we could restrict \cref{RoughBiext} to \(\pi_* z^m(*_X, 0)\otimes \pi_* z^m(*_X, 0)\).
We will take up this idea in \cref{sec:Calc}, where we recover sections the biextensions of Bloch and Seibold.
That is, we first construct sections where we do not have to consider descent.
Then, the Weil reciprocity law forces these sections to descend.
Like \(B_{W, Z}\) in the introduction, the sections \(\langle W, Z\rangle\) here are given for each pair of disjoint families of generically homologically trivial cycles.
This means that the sections defines a rational section of the line bundle obtained as the pull-back of the biextension, where we say a rational section because we also have in mind the situation where these families meet somewhere.
The divisor of the rational section was computed by Seibold as in \cref{prop:Calc}.

\section{Vanishing around the multiplication of higher Chow groups} \label{sec:Van}
Let \(n\) be an integer such that \(0\leqq n\leqq d+1\).
Put \(m\coloneqq d-n+1\).
Take \(W\in z^m(X,0)\).
In this section, we prove a vanishing result related with multiplying \([W]\).

Let
\[
    \theta_W\colon \CH^n(X,1)\to \CH^{d+1}(X,1)\to \CH^1(S,1)
\]
be the composition of the map that takes the product with \([W]\) and of \(\pi_*\).
The title refers to the following result.
\begin{lem}[{\cite[4.5 Korollar]{SeiBier}, \cf \cite[Lemma 1]{BloBiext}}] \label{lem:theta}
    Suppose that \(W\) is generically homologically trivial.
    Then \(\theta_W = 0\).
\end{lem}

\begin{proof}
Since the proof is in \cite{SeiBier}, we only make some remarks.
First, the author does not know why a result of Gabber was cited in the proof of \cite[Lemma 1]{BloBiext}.
Second, we talk about the case where \(S = \Spec k = \Spec \bC\) and that \(X\) is smooth over \(S\).
This case is settled in \cite[Lemma 1]{MulBiext} as explained in \cref{lem:thetaDel}.
However, the proof of \cite[Lemma 1]{MulBiext} cites a reference that is not applicable.
Also, the author does not know whether \(H_{\cD}^{2n-1}(X, \bZ(n))\subseteq J(X)\) as stated in that proof.
See \cref{rem:BettiEt} as well.
\end{proof}

\begin{lem}\label{lem:thetaDel}
    Suppose that \(S=\Spec \bC\).
    Then \(H_\cD^{2n-1}(X, \bZ(n))\xrightarrow{\cdot [W]_\cD} H_\cD^{2d+1}(X, \bZ(d+1))\) is \(0\).
\end{lem}

\begin{proof}
    Let \(J(X)\) be the kernel of
    \[
        \bigoplus_{p,q\in\bZ} H^p_\cD(X, \bZ(q))\to \bigoplus_{p,q\in\bZ} H^p(X, \bZ(q)).
    \]
    By \cite[Lemma 1.6]{BloHei}, the multiplication of Deligne cohomology induces
    \[
        H_\cD^{2n-1}(X, \bZ(n))/(J(X)\cap H_\cD^{2n-1}(X, \bZ(n)))\times (J(X)\cap H_\cD^{2m}(X, \bZ(m)))\to H_\cD^{2d+1}(X, \bZ(d+1)).
    \]
    Its source is isomorphic to the image of
    \[
        H_{\cD}^{2n-1}(X, \bZ(n))\to H^{2n-1}(X, \bZ(n))\oplus F^nH^{2n-1}(X, \bC),
    \]
    namely the kernel of
    \[
        H^{2n-1}(X, \bZ(n))\oplus F^nH^{2n-1}(X, \bC)\to H^{2n-1}(X, \bC).
    \]
    The last group is torsion.
    On the other hand, \(J(X)\cap H_\cD^{2m}(X, \bZ(m))\), where \([W]_\cD\) belongs to, is divisible by \cite[p.~129]{BloHei}.
    Thus the lemma follows.
\end{proof}

\section{Construction of a \texorpdfstring{\(\Gm\)}{Gm}-biextension} \label{sec:Constr}
The title means the biextension of \(\bfA^m(X/S)\) and \(\bfA^n(X/S)\).
As before, we stick to cohomological notations.

Set \(\pr_{S,1}, \pr_{S,2}\colon S^2\to S\) and \(\pr_{X,1}, \pr_{X,2}\colon X^2\to X\) to be the first and second projections.
Let \(\Delta_X\colon X\to X^2\) and \(\Delta_S\colon S\to S^2\) be the diagonal.
Put \(\pi^2\colon X^2\to S^2\).
Define
\[
    t\colon \Delta_S^{-1}\pi^2_*z^{d+1}(*_{X^2}, -\bullet, -\bullet)\to z^1(*_S,-\bullet) = \Gm[1]
\]
(\cf \cref{ex:CH1(1)}) to be the morphism in the derived category adjoint to
\begin{align*}
    &\pi^2_*z^{d+1}(*_{X^2}, -\bullet, -\bullet)\xleftarrow\simeq \pi^2_*z^{d+1}(*_{X^2}, -\bullet, -\bullet)_\Delta\to \pi^2_*z^{d+1}(*_{X^2},-\bullet) \xleftarrow\simeq \pi^2_*z^{d+1}_{\{\Delta X\}}(*_{X^2},-\bullet) \\
    \to& \pi^2_*\Delta_{X,*} z^{d+1}(*_X, -\bullet) = \Delta_{S,*}\pi_* z^{d+1}(*_X, -\bullet)\to \Delta_{S,*} z^1(*_S,-\bullet)\xrightarrow\simeq R\Delta_{S,*} z^1(*_S,-\bullet),
\end{align*}
with more details below.
\begin{itemize}
    \item The first arrow comes from \(z^{d+1}(*_{X^2}, -\bullet, -\bullet)_\Delta\to z^{d+1}(*_{X^2}, -\bullet, -\bullet)\).
    This is an isomorphism on any open subset of \(X^2\) by \cite[Lemma 8.1]{GLBK}.
    \item We obtain the second morphism from the triangulation \cite[p.~100]{GLBK}.
    \item The third arrow is an isomorphism due to \cite[Theorem 1.10]{KLAddHighChow} applied to quasi-projective open subsets of \(S^2\).
    \item The fourth morphism results from pushing the restriction by \(\Delta_X\) forward by \(\pi^2\).
    \item The second last arrow means \(\pi_*\) modified by \(\Delta_{S,*}\).
    \item The last isomorphism is by Zariski descent similar to \cite[right after Theorem 1.7]{LevTec}.
        This morphism is not essential in the construction.
        See the biggest diagram in \cref{sec:Calc}.
\end{itemize}

Also, consider the morphism
\begin{align*}
    &(\pi^2)^{-1}(\pr_{S,1}^{-1}\pi_*z^m(*_X, -\bullet)\otimes_\bZ \pr_{S,2}^{-1}\pi_*z^n(*_X, -\bullet)) \\
    =& \pr_{X,1}^{-1}\pi^{-1}\pi_*z^m(*_X, -\bullet)\otimes_\bZ \pr_{X,2}^{-1}\pi^{-1}\pi_*z^n(*_X, -\bullet) \\
    \to& \pr_{X,1}^{-1}z^m(*_X, -\bullet)\otimes_\bZ \pr_{X,2}^{-1}z^n(*_X, -\bullet)\to z^{d+1}(*_{X^2}, -\bullet, -\bullet)
\end{align*}
of complexes.
Its adjoint gives
\[
    \pr_{S,1}^{-1}\pi_*z^m(*_X, -\bullet)\otimes_\bZ \pr_{S,2}^{-1}\pi_*z^n(*_X, -\bullet) \to \pi^2_*z^{d+1}(*_{X^2}, -\bullet, -\bullet).
\]
It then pulls back via \(\Delta_S\) to
\begin{align} \label{prelimTimes}
    \pi_*z^m(*_X, -\bullet)\otimes_\bZ \pi_*z^n(*_X, -\bullet)\to \Delta_S^{-1}\pi^2_* z^{d+1}(*_{X^2}, -\bullet, -\bullet).
\end{align}
Now, define a subcomplex \(z_{\hom}^p(*_{X/S}, -\bullet)\) of \(\pi_*z^p(*_X, -\bullet)\) for nonnegative integers \(p\) by
\[
    z_{\hom}^p(*_{X/S}, q)\coloneqq \pi_*z^p(*_X, q)
\]
for \(0< q\in \bZ\) and by setting \(z_{\hom}^p(*_{X/S}, 0)\) to send an open subset \(U\subseteq S\) to the set of the generically homologically trivial cycles in \(z^p(\pi^{-1}(U), 0)\).
Restrict \cref{prelimTimes} to
\[
    *\times *\colon z_{\hom}^m(*_{X/S}, -\bullet)\otimes_\bZ z_{\hom}^n(*_{X/S}, -\bullet)\to \Delta_S^{-1}\pi^2_*z^{d+1}(*_{X^2}, -\bullet, -\bullet).
\]

\begin{lem} \label{lem:HomAlg}
    Let \(P, Q\) be bounded above complexes of the sheaves of abelian groups on a topological space.
    Let \(a, b\) be integers such that \(P^x = Q^y = 0\) if \(x>a, y>b\).
    Assume that \(P^a, Q^b\) are flat over \(\bZ\).
    The quotient of \(\tau_{\geqq a + b - 1} (P\otimes Q)\) by the image of
    \[
        (\Ker \partial_P^{a - 1}\otimes_\bZ Q^b\oplus P^a\otimes_\bZ \Ker \partial_Q^{b - 1})[1 - a - b]
    \]
    is isomorphic to \(H^a(P)[-a]\otimes_\bZ^\bL H^b(Q)[-b]\) in the derived category.
\end{lem}

\begin{proof}
    The quotient equals \(\tau_{\geqq a + b - 1}\) applied to
    \[
        \dots \to 0\to \Ima \partial_P^{a - 1}\otimes_\bZ \Ima \partial_Q^{b - 1}\to  \Ima \partial_P^{a - 1}\otimes_\bZ Q^b\oplus P^a\otimes_\bZ \Ima \partial_Q^{b - 1}\to P^a\otimes_\bZ Q^b\to 0 \to \dots,
    \]
    where \(P^a\otimes Q^b\) is in degree \(a + b\).
    Note that \(\Ima \partial_P^{a - 1}\) and \(\Ima \partial_Q^{b - 1}\) are flat.
\end{proof}

By \cref{lem:theta}, \(t\circ (*\times *)\) restricted to
\begin{align} \label{subcpx}
    (&\pi_*\Ker (\partial \colon z^m(*_X, 1)\to z^m(*_X, 0))\otimes_\bZ z^n_{\hom}(*_{X/S}, 0)\oplus \\
    &z^m_{\hom}(*_{X/S}, 0)\otimes_\bZ \pi_*\Ker (\partial \colon z^n(*_X, 1)\to z^n(*_X, 0)))[1] \nonumber
\end{align}
vanishes.
Let \(A\) be the quotient by this subcomplex of \(z_{\hom}^m(*_{X/S}, -\bullet)\otimes_\bZ z_{\hom}^n(*_{X/S}, -\bullet)\).
Thus \(t\circ (*\times *)\) uniquely induces \(t\circ (*\times *)\colon A\to \Gm[1]\).
\cref{lem:HomAlg} calculates the source of
\begin{align} \label{biext}    
    \tau_{\geqq -1}(t\circ (*\times *))\colon \bfA^m(X/S)\otimes^\bL_\bZ \bfA^n(X/S) \to \Gm[1].
\end{align}
The morphism \cref{biext} gives a \(\Gm\)-biextension \(\bE\) over \(\bfA^m(X/S)\times \bfA^n(X/S)\) by \cite[Expos\'e VII, Corollaire 3.6.5]{SGA7}.

\begin{rem} \label{rem:Bloch}
Originally, \(\bfE_W\) \cite[before (2.1)]{BloBiext} was obtained by restricting \(t\circ (*\times *)\) to
\[
    z_{\hom}^m(*_{X/S}, 0)\otimes_\bZ \pi_*z^n(*_X, -\bullet)
\]
and by taking the \(-1\)-st cohomology of its mapping cone.
Our construction induces \cite[(2.1)]{BloBiext} roughly as follows.
We can repeat our argument, replacing \(z_{\hom}^m(*_{X/S}, -\bullet)\otimes_\bZ z_{\hom}^n(*_{X/S}, -\bullet)\) with \(z_{\hom}^m(*_{X/S}, 0)\otimes_\bZ z_{\hom}^n(*_{X/S}, -\bullet)\) and \cref{subcpx} with
\[
    (z^m_{\hom}(*_{X/S}, 0)\otimes_\bZ \pi_*\Ker (\partial \colon z^n(*_X, 1)\to z^m(*_X, 0)))[1].
\]
The repeated argument leads us to \(z_{\hom}^m(*_{X/S}, 0)\otimes_\bZ \bfA^n(X/S)\to \Gm[1]\) that is compatible with \cref{biext}.
We get the corresponding extension of \(z_{\hom}^m(*_{X/S}, 0)\otimes_\bZ \bfA^n(X/S)\) by \(\Gm\) by looking at its mapping cone and the relevant long exact sequence of cohomology.
The pull-back of the \(\Gm\)-extension to \(z_{\hom}^m(*_{X/S}, 0)\times \bfA^n(X/S)\) is the biextension corresponding to the above derived morphism to \(\Gm[1]\), where the comparison of biextensions and extensions in \cite[\S 3]{SGA7} in this case is easy because \(z_{\hom}^m(*_{X/S}, 0)\otimes_\bZ \bfA^n(X/S) = z_{\hom}^m(*_{X/S}, 0)\otimes^\bL_\bZ \bfA^n(X/S)\).
The last \(\Gm\)-extension actually comes from the original construction, namely the long exact sequence associated with \(t\circ (*\times *)\) restricted to \(z_{\hom}^m(*_{X/S}, 0)\otimes_\bZ z_{\hom}^n(*_{X/S}, -\bullet)\) and its mapping cone.
\end{rem}

\section{Calculation of the line bundle} \label{sec:Calc}
Take \(W\in z_{\hom}^m(S_{X/S},0)\) and \(Z\in z_{\hom}^n(S_{X/S},0)\).
These induce a global section of \(\bfA^m(X/S)\times \bfA^n(X/S)\).
The section in turn defines a Zariski \(\Gm\)-torsor \(\bE_{W,Z}\) on \(S\).
This line bundle has been determined as follows.

\begin{prop}[{\cite[6.8 Satz]{SeiBier}}] \label{prop:Calc}
    Assume that \(S\) is quasi-projective over \(k\).
    Then \(\bE_{W,Z}\) corresponds to \(\pi_*([W][Z])\) via \(\Pic(S)\simeq \CH^1(S)\).
\end{prop}

\begin{rem} \label{rem:propCalcSign}
    \cref{prop:Calc} depends on the normalization of \cref{ex:CH1(1)}.
\end{rem}

We go back to the case of general \(S\).
Suppose that \(W\) and \(Z\) are disjoint.
Then \cite[6.8 Satz]{SeiBier} actually also shows that \(\bE_{W, Z}\) is trivial.
For more precision, we briefly present the following variant of the construction in \cite{BloBiext,SeiBier}.
We start by modifying \cref{rem:Bloch}, incorporating properly intersecting sections.
Let \(\Psi_1\) be the subcomplex of \(z^m_{\hom}(*_{X/S}, 0)\otimes_\bZ z^n_{\hom}(*_{X/S}, -\bullet)\) whose \(p\)-th term for an integer \(p\) is generated by sections of the form \(W\otimes Z\), where \(Z\) is a section of \(z^n_{\{W\}}(*_{X/S}, -p)\).
The restriction of \(*\times *\) to \(\Psi_1\) factors through \(\Delta_S^{-1}\pi^2_*z^{d+1}(*_{X^2}, -\bullet, -\bullet)_\Delta\).
The resulting \(\Psi_1 \to \Delta_S^{-1}\pi^2_*z^{d+1}(*_{X^2}, -\bullet, -\bullet)_\Delta\to \Delta_S^{-1} \pi^2_*z^{d+1}(*_{X^2},-\bullet)\) in turn factors through \(\Delta_S^{-1}\pi^2_*z_{\Delta X}^{d+1}(*_{X^2},-\bullet)\).
We eventually obtain a morphism \(\Psi_1\to z^1(*_S, -\bullet)\) of complexes rather than that in the derived category.
The long exact sequence of cohomology associated with its mapping cone \(C_1\) gives
\[
    H^{-1}(\Psi_1)\xrightarrow{0} \bG_{\mathrm{m}, S}\to H^{-1}(C_1)\to z^m_{\hom}(*_{X/S}, 0)\otimes_\bZ \bfA^n(X/S) \to 0,
\]
where the leftmost morphism vanishes due to \cref{lem:theta}, and \(H^0(\Psi_1) = z^m_{\hom}(*_{X/S}, 0)\otimes \bfA^n(X/S)\) by \cite[Theorem 1.10]{KLAddHighChow}.
The extension can be seen as a biextension of \((z^m(*_{X/S}, 0), \bfA^n(X/S))\) by \(\bG_{\mathrm{m}, S}\).

Consider a similar subcomplex \(\Psi_2\) of \(z^m_{\hom}(*_{X/S}, -\bullet)\otimes_\bZ z^n_{\hom}(*_{X/S}, 0)\) and a mapping cone \(C_2\).
This gives rise to the extension \(H^{-1}(C_2)\) of \(\bfA^m(X/S)\otimes_\bZ z^n_{\hom}(*_{X/S}, 0)\) by \(\bG_{\mathrm{m}, S}\).
This also gives a biextension.
These extensions coincide when pulled back to \(\Psi_1^0 = \Psi_2^0\), easily checked without concerns about homotopy here.
Similar comparison works for torsors over \((z^m_{\hom}(*_{X/S}, 0)\times z^n_{\hom}(*_{X/S}, 0))\cap \Delta_S^{-1}\pi^2_*z^{d + 1}_{\{\Delta X\}}(*_{X^2}, 0)\).
The last sheaf is a bisubgroup \cite[Definition 3.1]{GorBiext} sheaf of \(z^m_{\hom}(*_{X/S}, 0)\times z^n_{\hom}(*_{X/S}, 0)\).
It makes sense to talk about biextensions over it, and we can strengthen the comparison of torsors to that of biextensions if we want.

Suppose that \(U\subseteq S\) is an open subset.
Also suppose that \(W\in z^m_{\hom}(U_{X/S}, 0), Z\in z^n_{\hom}(U_{X/S}, 0)\) are disjoint.
Then \((0, W\otimes Z)\in \Ker\partial_C^{-1} \subseteq z^1(U, 1)\oplus \Psi_1^0\) defines \(\langle W, Z\rangle\in H^{-1}(C_1)\).
We use the same notation for \([(0, W\otimes Z)]\in H^{-1}(C_2)\) since these cohomology classes in the extensions coincide under the above pull-back to \(\Psi_1^0 = \Psi_2^0\).

\begin{clm} \label{clm:descent}
    There exist
    \begin{itemize}
        \item a biextension of \((\bfA^m(X/S), \bfA^n(X/S))\) by \(\bG_{\mathrm{m}, S}\) that simultaneously descend those induced by \(H^{-1}(C_1), H^{-1}(C_2)\) and that has the isomorphism class given by \cref{biext},
        \item and sections \(\langle W, Z\rangle\) of the biextension for disjoint \(W\in z^m_{\hom}(U_{X/S}, 0), Z\in z^n_{\hom}(U_{X/S}, 0)\) for open subsets \(U\subseteq S\) that descend those in \(H^{-1}(C_1), H^{-1}(C_2)\).
    \end{itemize}
\end{clm}

\begin{proof}
    Sections of \(\bfA^m(X/S)\times \bfA^n(X/S)\) locally come from disjoint sections of \(z^m_{\hom}(*_{X/S}, 0)\) and \(z^n_{\hom}(*_{X/S}, 0)\) by the moving lemma.
    Therefore, we can always trivialize the biextensions induced by \(H^{-1}(C_1), H^{-1}(C_2)\) with \(\langle W, Z\rangle\).
    The biextension of \((z^m_{\hom}(*_{X/S}, 0), z^n_{\hom}(*_{X/S}, 0))\) has been glued to those induced by \(H^{-1}(C_1), H^{-1}(C_2)\) by comparing \(\langle W, Z\rangle\) for rationally equivalent \(W\) and \(Z\).
    These gluing data combine to a new gluing data that gives a biextension of \((\bfA^m(X/S), \bfA^n(X/S))\) and its sections by \cite[(3.5)]{BloBiext} \cite[4.1 Satz]{SeiBier}.

    For a nonnegative integer \(p\), set \(z^p_{\hom, \equi}(*_{X/S}, 0)\) to be the subsheaf of \(z^p_{\hom}(*_{X/S}, 0)\) generated by the cycles that intersect properly with the fibers of \(\pi\).
    Let \(z^p_{\rat, \equi}(*_{X/S}, 0)\coloneqq \Ima \partial^{-1}_{\pi_*z^m(*_X, -\bullet)}\cap z^p_{\hom, \equi}(*_{X/S}, 0)\).
    Consider the double subcomplex \(\Phi\), with terms generated by the tensor products of the disjoint sections, of
    \begin{align*}        
        &(\dots \to 0\to z^m_{\rat, \equi}(*_{X/S}, 0)\to z^m_{\hom, \equi}(*_{X/S}, 0)\to 0\to \dots) \\
        \otimes_\bZ &(\dots \to 0\to z^n_{\rat, \equi}(*_{X/S}, 0) \to z^n_{\hom, \equi}(*_{X/S}, 0)\to 0\to \dots),
    \end{align*}
    where \(z^m_{\hom, \equi}(*_{X/S}, 0)\) and \(z^n_{\hom, \equi}(*_{X/S}, 0)\) are placed in degree \(0\).
    The subcomplex is quasi-isomorphic to \(\bfA^m(X/S)\otimes^\bL_\bZ \bfA^n(X/S)\) by the moving lemma.
    The biextension that we have just made has the isomorphism class corresponding to
    \[
        \begin{CD}
            \dots \to 0 @>>> \Phi^{-2} @>>> \Phi^{-1} = \Phi^{-1, 0}\oplus \Phi^{0, 1} @>>> \Phi^0 @>>> 0\to\dots \\
            @.              @VVV            @VVV                                            @VVV         @. \\
            \dots \to 0 @>>> 0 @>>> \bG_{\mathrm{m}, S} @>>> 0 @>>> 0\to\dots, \\
        \end{CD}
    \]
    with the middle vertical morphism
    \[
        (W\otimes Z, W'\otimes Z')\mapsto \sigma_{Z}(W)\sigma_{W'}(Z')\quad \textup{\cite[(3.2)]{BloBiext} \cite[4.1 Satz]{SeiBier}}.
    \]
    By the cited constructions, \(\Phi \to \bG_{\mathrm{m}, S}[1]\) equals \cref{biext}.
\end{proof}

By construction, our \(\langle W, Z\rangle\) are bilinear.
Also, if \(W\) is rationally trivial, then
\begin{align} \label{TwistedSection}
    \langle W, Z\rangle = \sigma_{Z}(W)\langle 0, Z\rangle.
\end{align}
A similar property holds if \(Z\) is rationally trivial.

\begin{rem}
    In defining sections of the biextension given through \cite[Expos\'e VII, Corollaire 3.6.5]{SGA7}, we cannot help talking about homotopy.
    For example, it would be hard for us to show \cref{TwistedSection} for the both cases of the rationally trivial \(W\) and of rationally trivial \(Z\) while embracing the derived nature of our construction at least for the two steps.
    First, \(\wt{WZ}\) in the proof of \cref{prop:Calc} depend a priori on the choice of \(Z\) and especially of \(W\) in their rational equivalence class.
    Second, even if we can show the independence of the choice of \(Z\), doing the same after switching the role of \(W\) and \(Z\) and gluing the two constructions would require similar arguments to the proof of \cref{clm:descent}.
    We can also think of assuring the independence of \(W\) and \(Z\) at the same time in a symmetric construction, but this makes the first difficulty more pronounced.

    In fact, these kinds of homotopy sometimes define nontrivial isomorphisms between biextensions \cite[Expos\'e VII, (2.5.4.1)]{SGA7}.
    To avoid the nontrivial isomorphisms, we sometimes need to argue indirectly later, constructing sections of biextensions that can be incompatible among different fibers of the projections.
    The subtleties often hide in morphisms of distinguished triangles.
\end{rem}

\cref{clm:descent} actually defines a rational section \(\langle W, Z\rangle\) of \(\bE_{W,Z}\) for \(W\in z^m_{\hom}(U_{X/S}, 0), Z\in z^n_{\hom}(U_{X/S}, 0)\) for open subsets \(U\subseteq S\).
With this in mind, we state \cite[6.8 Satz]{SeiBier} in its full strength.
\begin{prop}[{\cite[6.8 Satz]{SeiBier}}] \label{prop:FullCalc}
    Assume that \(S\) is quasi-projective over \(k\).
    Take \(W\in z_{\hom}^m(S_{X/S},0)\) and \(Z\in z_{\hom}^n(S_{X/S},0)\) that properly intersect.
    Then the divisor associated with the section \(\langle W, Z\rangle\) of \(\bE_{W,Z}\) is \(\pi_*(W\cdot Z)\).
\end{prop}

\part{Comparison with Archimedean biextensions} \label{chap:Arch}

This part of the paper compares the biextension obtained in \cref{chap:BloBiext} with that in \cite[\S 3]{HainBiext}.
First, we define the integral analog of the pre-variations of \(\bQ\)-mixed Hodge structures in \cref{sec:pVMHS}.
These are the variation of mixed Hodge structures, except that they do not satisfy the Griffiths transversality condition.
We need them to discuss the intermediate Jacobians as the sheaf version of \(Ext^1\) groups in the category of pre-variations.
Second, we discuss Archimedean biextensions introduced by Hain and by Brosnan--Pearlstein in \cref{sec:Arch}.

In the next \cref{sec:ArchCompSm}, we first show that the biextension by Bloch and that by Hain coincide.
Namely, we talk about the smooth case.
Toward this goal, we first construct the Deligne cohomology analog of the biextension by Bloch.
Namely, we suitably modify \cref{RoughBiextD} using \cref{lem:thetaDel} as we did in \cref{sec:Constr} for \cref{RoughBiext}.
The modification ends up with a biextension over the product of what we could call Betti (co)homologically trivial parts of the Deligne cohomology, namely intermediate Jacobians.
This Deligne cohomology analog intertwines the biextension \(E\) from \cref{chap:BloBiext} and the Archimedean biextension by Hain.

To show that an isomorphism exists between \(\bE\) and its Deligne cohomology analog, we can use the regulator map from higher Chow groups to Deligne cohomology by Binda--Saito.
However, this is not enough when we eventually want to compare \(\langle W, Z\rangle\) and \(B_{W, Z}\).
In general, the automorphism group of any biextension of a pair \((A, B)\) of abelian groups by an abelian group \(C\) is isomorphic to \(\Hom(A\otimes B, C)\) according to Grothendieck \cite[Expos\'e VII, (2.5.4.1)]{SGA7}.
Specifying the images of \(\langle W, Z\rangle\) suffices to pin down the element of this kind of automorphism group.
Prescribing the element of the automorphism group also means giving a specific homotopy in the commutative diagram
\[
    \begin{CD}
        A\otimes B @>>> C[1] \\
        @|              @| \\
        A\otimes B @>>> C[1].
    \end{CD}
\]
in the derived category.
We need to be careful in the process, since Deligne cohomology is defined as a cone of a derived morphism.
We resolve this issue by writing everything down in terms of the Godement resolution.

After all, we define \(\langle W, Z\rangle_\cD\) as the image of \(\langle W, Z\rangle\) under the regulator map at the level of complexes.
Strictly speaking, the defined sections do not belong to the Deligne cohomology analog of \(\bE\), but to the biextension of \((\bfA^m(X), z^n_{\hom}(*_{X/S}, 0))\).
That is, we break the symmetry that we have advertised.
It may be possible to recover the symmetry of \(\langle W, Z\rangle_\cD\). 
However, we need to abandon the symmetry anyway as soon as we mimic \cite[Theorem 7.11]{EVDB} in the comparison of the Deligne cohomology analog of \(\bE\) and what Hain did.

We switch the focus to \cref{prop:<WZ>DVSBWZ} that the Deligne cohomology analog of \(\bE\) and Hain's biextension agree.
Forgetting the symmetry of \(\langle W, Z\rangle_\cD\) enables a cohomological description of \(\langle W, Z\rangle_\cD\) similar to a description of the Deligne cycle class of \(W\).
Also, \(B_{W, Z}\) resembles the image of Abel--Jacobi maps in some \(\Ext^1\), which appears in the description of the Deligne cycle class.
These suggest that we can follow the argument in \cite[Theorem 7.11]{EVDB} as long as we pay attention to the increased number of distinguished triangles using \cref{lem:3x3}.

Finally, we treat the strongly semistable degeneration over a curve in \cref{sec:ArchComp}.
This means that our \(W, Z\) are generically homologically trivial now.
We have been interested in the asymptotic behavior of the height of \(B_{W, Z}^\circ\) near the singular fiber.
We first discuss in \cref{lem:Hodge} that after assuming the Hodge conjecture, we can modify \(W\) so that it has rational coefficients and is homologically trivial even in the singular fibers.
The argument rests on Hodge theory, especially the Clemens--Schmid exact sequence.

We finally compare the \(\bQ\)-line bundles by Seibold and by Brosnan--Pearlstein in \cref{thm:ArchComp}.
In this paragraph, we will be sloppy about the rational coefficients.
They are not an issue as soon as we multiply \(W\) and the \(\bQ\)-line bundles by appropriate integers.
After modifying \(W\) as in \cref{lem:Hodge}, we locally replace \(W, Z\) by rationally equivalent cycles \(W', Z'\) so that they do not intersect and they are equidimensional over the base.
The sections of the line bundle by Brosnan--Pearlstein are those of the line bundle by Hain in the smooth locus with bounded heights around the singular locus.
Therefore, to compare the two line bundles, it suffices to show that \(B_{W', Z'}^\circ\) has bounded height as we expect from the fact that \(\langle W', Z'\rangle\) is defined over a singular locus.
It turns out that we can repeat the construction of \(B_{W', Z'}^\circ\) in the smooth locus in our singular case, if we work in the category of mixed Hodge modules.
This is thanks to Brosnan--Fang--Nie--Pearlstein \cite[Lemma 2.18]{BFNPANF}.
They state certain exactness of the intermediate extension functor \(j_{!*}\) in the category of mixed Hodge modules.
Recall that \(j_{!*}\) usually preserves the injectivity and surjectivity, but not the exactness.
In the case of the curve \(S\), we have \(j_{!*}(V[1]) = (j_* V)[1]\) at the level of underlying complexes of sheaves for a variation \(V\) of \(\bQ\)-mixed Hodge structures on the smooth locus.
The failure to preserve the exactness comes from the failure of the right exactness of \(j_*\).
Eventually, we prove that applying \(j_{!*}\) to \(B_{W', Z'}^\circ\) in the smooth locus gives us the blended extension of \(\bQ\) by \(j_{!*}(R^{2m-1}\pi_*\bQ(m))\) by \(\bQ(1)\).
Again, this is unusual considering the general theory of mixed Hodge modules.
Indeed, the general discussion \cite[Remark 20]{BPJump} on so-called biextension\footnote{This ``biextension'' is in the sense of blended extensions in this paper.} Hodge structures that are admissible show that the resulting surjectivity of \(j_{*}\) leads to the bounded height.

We apply the comparison to get the asymptotic behavior in \cref{thm:intro}, namely \cref{thm:Apply}.
The leading term has the form \(q\log |t|\) for a rational number \(q\) as known by Lear \cite[(6.2.2)]{LeaExt}.
This \(q\) is related with the order of vanishing of the section \(B_{W, Z}\) \cite[Lemma 6.1]{HdJAsym}, where \(W\) is suitably modified as in \cref{lem:Hodge}.
We easily calculate this using \(\langle W, Z\rangle\), getting the non-Archimedean local height pairing by the compatibility of maps in motivic and etale cohomology used to define the non-Archimedean pairing.

\section{Pre-variations of mixed Hodge structures} \label{sec:pVMHS}

Our discussion will center around the pre-variations, not variations, of mixed Hodge structures because of its closer relationships with the intermediate Jacobians.
As in \cite[Definition 3.10]{FFVMHS>0}, the pre-variations stands for the variation of mixed Hodge structures except that we do not require the Griffiths transversality condition.

\begin{defi}
Let \(T\) be a complex manifold.
A \emph{pre-variation of mixed Hodge structures} on \(T\) is a pair of a local system \(L\) of finitely generated abelian groups on \(T\) and a pre-variation of \(\bQ\)-mixed Hodge structure \cite[Definition 3.9]{FFVMHS>0} which \(L\otimes_\bZ \bQ\) underlies.
\end{defi}

Let \(\pVMHS(T)\) (resp.~\(\VMHS(T)\)) be the category of the pre-variations (resp.~variations) of mixed Hodge structures on \(T\).
These are abelian categories \cite[Lemma 3.14]{FFVMHS>0}.
For \(H\in \pVMHS (T)\), define \(H_\bZ, H_\bQ, H_\cO\) to be the underlying \(\bZ\)-local system, \(\bQ\)-local system and finite free locally free \(\cO_T\)-module.
Also, let \(W_\bullet \coloneqq W_\bullet H_\bQ\) and \(F^\bullet \coloneqq F^\bullet H_\cO\) be the weight and Hodge filtrations of \(H\).

If \(G, H\) are pre-variations of Hodge structure on \(T\), then let \(Ext_{\pVMHS}^1(G, H)\) be the sheafification of the presheaf \(U\mapsto \Ext^1_{\pVMHS(U)}(G|_U, H|_U)\) on \(T\).
Similarly define \(Hom_{\pVMHS}\) as well as \(Ext_{\VMHS}^1\) for appropriate \(G, H\).
When \(G, H\) are pre-variations on \(T\) of mixed Hodge structures, the definition of \(Hom_{\pVMHS}(G, H)\) does not need a sheafification.
Furthermore, \(Hom_{\pVMHS}(G, H)\) is a local system by the identity theorem.
\begin{ex}[{\cite[Lemma 184]{BPJump}}] \label{ex:ExtJac}
    Without the need to sheafify,
    \[
        Ext_{\VMHS}^1(\bZ, \bZ(1)) = Ext_{\pVMHS}^1(\bZ, \bZ(1)) = \cO_{T}^\times.
    \]
\end{ex}

Now, let \(H\) be a variation of Hodge structure on \(T\) of weight \(-1\) and with torsion-free stalks of the underlying local system.
Then \(Ext_{\pVMHS}^1(\bZ, H)\) is precisely the sheaf of the holomorphic sections of the intermediate Jacobian of \(H\) \cite[\S C.3]{GZFam}, although we will reprove this later to fix our convention.
Also, set \(\cB(H)\) to be the sheaf \cite[Lemma 183]{BPJump} that carries an open subset \(U\subseteq T\) to the set of the isomorphic classes of the blended extensions in \(\VMHS(U)\) of \(\bZ\) by \(H|_U\) by \(\bZ(1)\).
By \cite[Proposition 4.4.1]{HarBiext}, we can check that \(\cB(H)\to Ext_{\VMHS}^1(H, \bZ(1))\times Ext_{\VMHS}^1(\bZ, H)\) is a biextension \cite[Corollary 185]{BPJump}.
The pre-variation analog follows from a surjectivity of a map between intermediate Jacobians.

Put \(H^\vee \coloneqq Hom_{\pVMHS}(H, \bZ(1)_T)\).
Note
\begin{align} \label{Dual}    
    Ext_{\VMHS}^1(H, \bZ(1)_T) = Ext_{\VMHS}^1(\bZ_T, H^\vee)
\end{align}
by the duality multiplied by \(-1\) following \cite[(3.1.7), (3.1.8)]{HainBiext}.

\section{Review of Archimedean biextensions} \label{sec:Arch}

This section reviews part of \cite[(3.4)]{HainBiext} and \cite[Theorem 241]{BPJump}.

\subsection*{The smooth case}
First, we deal with the former.
Let \(m, n\) be nonnegative integers such that \(m + n = d + 1\).
Suppose that \(\pi\) is smooth until the end of \cref{rem:sign}.
We specialize to the case when \(H = R^{2m-1}\pi_* \bZ(m)/\mathrm{tors}\), where \(\mathrm{tors}\) is the torsion part.
Suppose that \(W\in z_{\hom}^m(S_{X/S}, 0)\) and \(Z\in z_{\hom}^n(S_{X/S}, 0)\) intersect properly with the fibers of \(\pi\).
Also, assume that \(|W|\cap |Z| = \emptyset\).
First, let \(i_Z\colon |Z|\to X\) and \(j_Z\colon X\backslash |Z|\to X\) be the inclusions.
We have the distinguished triangle
\[
    R\pi_*i_{Z, *}Ri_Z^!\bZ(n)\to R\pi_*\bZ(n)\to R\pi_{Z, *}\bZ(n)\xrightarrow{+1},
\]
where \(\pi_Z\coloneqq \pi\circ j_Z\).
This induces the first of the three rows
\begin{equation} \label{E_Z}
\begin{tikzcd}[cramped, column sep=small]
R^{2n-1}\pi_* i_{Z,*} Ri_Z^!\bZ(n) \arrow[r] &
R^{2n-1}\pi_* \bZ(n) \arrow[r] &
R^{2n-1}\pi_{Z,*}\bZ(n) \arrow[r, "-1"] &
R^{2n}\pi_* i_{Z,*} Ri_Z^!\bZ(n) \arrow[r] &
R^{2n}\pi_* \bZ(n) \\
0 \arrow[r] \arrow[u, equal]&
R^{2n-1}\pi_* \bZ(n) \arrow[r] \arrow[u, equal] \arrow[d] &
E'_Z \arrow[r] \arrow[u] \arrow[d] &
\bZ_{S(\bC)} \arrow[r] \arrow[u, "{-[Z]}_{\mathrm{sing}}"] &
0 \arrow[u] \\
0 \arrow[r] &
\displaystyle\frac{R^{2n-1}\pi_* \bZ(n)}{\mathrm{tors}} \arrow[r] &
E_Z \arrow[r] &
\bZ_{S(\bC)} \arrow[r] \arrow[u, equal] &
0
\end{tikzcd}
\end{equation}
of compatible exact sequences, where we define the second row by the pullback by
\[
    \bZ_{S(\bC)}\xrightarrow{-[Z]_{\sing}} H^{2n}_{|Z|(\bC)}(X(\bC), \bZ(n))_{S(\bC)}\to R^{2n}(\pi\circ i_Z)_* \bZ(n).
\]
The morphism labeled by \(-1\) in the top row of \cref{E_Z} is the negative of the one that induces the usual construction in the snake lemma \cite[(1.3.3)]{ConBaseChange}.
We use \(-[Z]_{\sing}\) for \cref{rem:homology}.
The cup product
\[
    R^{2n-1}\pi_{Z, *}\bZ (n)\times R^{2m-1}\pi_{Z, !}\bZ (m)\to R^{2d}\pi_{Z, !} \bZ(d + 1)\to \bZ(1)_{S(\bC)}
\]
defines
\begin{align} \label{2dual}
    R^{2m-1}\pi_{Z, !}\bZ (m)\to (R^{2n-1}\pi_{Z, *} \bZ(n))^\vee \to E_Z^\vee.
\end{align}

Second, set \(i_W\colon |W|\to X\) and \(j\colon X\backslash (|W|\cup |Z|)\to X\backslash |Z|\) to be the inclusions.
We switch our focus to the distinguished triangle
\[
    R\pi_*i_{W, *}Ri_W^!\bZ(m)\to R\pi_{Z, !}\bZ(m)\to R\pi_{Z, !}Rj_*\bZ(m)\xrightarrow{+1}.
\]
As before, this induces the top exact row of the diagram
\[
\begin{tikzcd}[cramped]
    0 \arrow[r] & R^{2m-1}\pi_{Z, !}\bZ(m) \arrow[r] \arrow[d, equal] & R^{2m-1}\pi_{Z, !}Rj_*\bZ(m) \arrow[r, "-1"] & R^{2m}\pi_* i_{W, *}Ri_W^!\bZ(m) \arrow[r] & R^{2m}\pi_{Z, !} \bZ(m) \\
    0 \arrow[r] & R^{2m-1}\pi_{Z, !}\bZ(m) \arrow[r]                  & B_W(Z) \arrow[u] \arrow[r]                  & \bZ_{S(\bC)} \arrow[r] \arrow[u, "{-[W]_{\sing}}"] & 0. \arrow[u]
\end{tikzcd}
\]
The second row is defined by the pullback as in \cref{E_Z}, where the rightmost square is commutative as below.
\begin{rem} \label{rem:PL?}
    Our argument here is different from \cite[(3.3)]{HainBiext}.
    There the argument from PL topology is used.
\end{rem}
Consider the triangle
\[
    R\pi_{Z, !} \bZ(m) \to R\pi_*\bZ (m) \to R\pi_*i_{Z, *}\bZ (m)\xrightarrow{+1}.
\]
This induces the exact sequence
\[
    R^{2m-1}\pi_*i_{Z, *}\bZ (m)\to R^{2m}\pi_{Z, !} \bZ(m) \to R^{2m}\pi_*\bZ (m) \to R^{2m}\pi_*i_{Z, *}\bZ (m).
\]
Because of the proper intersection of \(Z\) with the fibers of \(\pi\), the most left and right terms are \(0\).
This invokes the desired commutativity.

Now, consider the pushout
\[
\begin{CD}    
    0 @>>> R^{2m-1}\pi_{Z, !}\bZ(m)@>>> B_W(Z) @>>> \bZ_{S(\bC)} @>>> 0 \\
    @.    @V{\cref{2dual}}VV           @VVV        @|              @.\\
    0 @>>> E_Z^\vee                @>>> B_{W, Z} @>>> \bZ_{S(\bC)} @>>> 0.
\end{CD}
\]
Note the short exact sequence
\[
    0\to \bZ(1)_{S(\bC)}\to E_Z^\vee \to R^{2m-1}\pi_*\bZ (m)/\mathrm{tors} = H \to 0
\]
by Poincar\'e duality.
We equip \(B_{W, Z}\) with the following structure of a blended extension:
\[
    0\to \bZ(1)_{S(\bC)}\xrightarrow{-1} E_Z^\vee \to B_{W, Z},
\]
where the sign is introduced due to the sign in \cref{Dual}.

The Poincar\'e line bundle minus its zero section over \(Ext_{\pVMHS}^1(\bZ, H)\times Ext_{\pVMHS}^1(\bZ, H^\vee)\) is isomorphic to the biextension
\[
    \cB (H)\to Ext_{\VMHS}^1(\bZ, H)\times Ext_{\VMHS}^1(\bZ, H^\vee)
\]
after the pullback, as summarized in, for example, \cite[\S D.2]{GZFam}.
We have \(B_{W, Z}\in \cB (H)(S(\bC))\).
It maps to \((E_W, E_Z)\in (Ext_{\pVMHS}^1(\bZ, H)\times Ext_{\pVMHS}^1(\bZ, H^\vee))(S(\bC))\).

\begin{rem} \label{rem:homology}
The original construction of \cite[\S 3]{HainBiext} uses homology rather than cohomology.
This coincides with our exposition due to the following isomorphisms.

We only describe the case \(S = \Spec\bC\) because \emph{loc. cit.} is the most detailed in that case.
By \cite{HofTri} and \cite[Theorem 3.5]{BryPL}, there are open neighborhoods \(\wt W\) and \(\wt Z\) of \(|W|(\bC)\) and \(|Z|(\bC)\) whose closures are disjoint and are homeomorphic to the mapping cones of some \(\partial \ol{\wt W}\to |W|(\bC)\) and \(\partial \ol{\wt Z}\to |Z|(\bC)\) so that \(\partial \ol{\wt W}\), \(|W|\), \(\partial \ol{\wt Z}\) and \(|Z|\) are identically preserved.
Some of the isomorphisms are Poincar\'e duality including
\[
    H^{2n}(X(\bC), X(\bC)\backslash |Z|(\bC), \bZ(n))\simeq H_{2m - 2}(|Z|(\bC), \bZ(1 - m))
\]
and isomorphisms involving Poincar\'e duality, excision and homotopy equivalences:
\begin{align*}
    &H^{2n - 1}(X(\bC)\backslash |Z|(\bC), \bZ(n))\xrightarrow{\simeq} H^{2n - 1}(X(\bC)\backslash \wt Z, \bZ(n))\xrightarrow{\simeq} H_{2m - 1}(X(\bC)\backslash \wt Z, \partial \wt Z, \bZ(1 - m)) \\
    &\xleftarrow{\simeq} H_{2m - 1}(X(\bC), \ol{\wt Z}, \bZ(1 - m))\xleftarrow{\simeq} H_{2m - 1}(X(\bC), |Z|(\bC), \bZ(1 - m)), \\
    &H^{2m - 1}(X(\bC)\backslash |W|(\bC), |Z|(\bC), \bZ(m))\xrightarrow{\simeq} H^{2m - 1}(X(\bC)\backslash \wt W, \ol{\wt Z}, \bZ(m)) \\
    &\xrightarrow{\simeq} H^{2m - 1}(X(\bC)\backslash (\wt W\sqcup \wt Z), \partial \wt Z, \bZ(m)) \xrightarrow{\simeq} H_{2n - 1}(X(\bC)\backslash (\wt W\sqcup \wt Z), \partial \wt W, \bZ(1 - n)) \\
    &\xleftarrow{\simeq} H_{2n - 1}(X(\bC)\backslash \wt Z, \ol{\wt W}, \bZ(1 - n)) \xleftarrow{\simeq} H_{2n - 1}(X(\bC)\backslash |Z|(\bC), |W|(\bC), \bZ(1 - n)).
\end{align*}
\end{rem}

\begin{rem} \label{rem:sign}
    We use \(B^\circ_{W, Z} \coloneqq B_{-W, -Z}\) rather than \(B_{W, Z}\) in what follows.
    This is because \(E_{-W}\) seen as an element of \(Ext^1_{\pVMHS}(\bZ_{S(\bC)}, H) \simeq H_\bZ\backslash H_\cO/F^0H_\cO\) is the image of \(W\) by the Abel--Jacobi map according to \cite[Proposition 3.2]{KLM} since the isomorphism as in \cite[(2.2)]{HainBiext} is the negative \cite[(2.2)]{HainBiext} of that of \cite{CarExt}.
    Also, \(B^\circ_{W, Z}\) and \(B_{W, Z}\) gives the same height in terms of \cite[Definition D.1]{GZFam}.
\end{rem}

\subsection*{The general case}
Next, we allow \(\pi\) to have singular fibers.
Let \(\emptyset \neq S_{\sm}\subseteq S\) be an open subset over which \(\pi\) is smooth.
Let \(j_{\sm}\colon S_\sm \to S\) be the inclusion.
For simple notations, we consider a torsion-free variation \(H\) of Hodge structures of weight \(-1\) on \(S_{\sm}\).
Also for simplicity, we assume that \(S\backslash S_\sm\) is a normal crossing divisor.
The article \cite{BPJump} discusses the extension of the fibers of \(\cB(H)\) as a line bundle.

For an open subset \(U\subseteq S(\bC)\), let \(\cB^{\ad}(H)(U)\) be the set of the isomorphism classes of blended extensions of \(\bZ_{U\cap S_\sm}\) by \(H|_{U\cap S_\sm}\) by \(\bZ(1)_{U\cap S_\sm}\) in the category \(\AVMHS(U)\) of the admissible \cite[Proposition 3.13]{SZVMHS}, \cite[1.8, 1.9]{KasVMHS} variations of mixed Hodge structures on \(U\cap S_\sm\subseteq U\).
This defines a sheaf \(\cB^{\ad}(H)\) on \(S(\bC)\) \cite[Lemma 183]{BPJump}.
Define \(Ext^1_{\AVMHS}\) as before.
In fact, \(\cB^{\ad}(H)\) is a biextension of \((Ext^1_{\AVMHS}(H, \bZ(1)_{S(\bC)}), Ext^1_{\AVMHS}(\bZ_{S(\bC)}, H))\) by \(j_{\sm, *}\cO_{S_\sm(\bC)}^\times\) \cite[Corollary 185, Theorem 81]{BPJump}.

\begin{ex}
Take nonnegative integers \(m, n\) such that \(m + n = d + 1\).
Take generically homologically trivial cycles \(W\in z_{\hom}^m(S_{X/S}, 0)\) and \(Z\in z_{\hom}^n(S_{X/S}, 0)\).
Let \(S_\sm\) be the locus over which \(\pi\) is smooth, and \(W, Z\) intersect properly with the fibers of \(\pi\).
Suppose that \(H = R^{2m-1}\pi_*\bZ_{\pi^{-1}(S_\sm)(\bC)}\).
Then \(B_{W, Z}\in \cB^\ad (H)(S(\bC))\) \cite{GNPPHyp}.
\end{ex}

To pin down a holomorphic line bundle on \(S(\bC)\) extending the fibers of \(\cB(H)\), we need the notion of a torsion pairing.
For a further simplicity, we assume that \(S\) is a curve.
Take \(B\in \cB^{\ad}(H)(S)\).
It induces a blended extension \(B_\bZ\) of the local system \(\bZ_{S_{\sm}(\bC)}\) by \(H_{\bZ}|_{S_{\sm}(\bC)}\) by \(\bZ(1)_{S_{\sm}(\bC)}\).
Take a point \(s\in S\backslash S_\sm\).
Let \(s\in \Delta\subset S_{\sm}(\bC)\cup \{s\}\) be a small disc around \(s\).
The induced extensions of \(\bQ_{\Delta\backslash \{s\}}\) by \(H_{\bQ}|_{\Delta\backslash \{s\}}\) and of \(H_{\bQ}|_{\Delta\backslash \{s\}}\) by \(\bQ(1)_{\Delta\backslash \{s\}}\) are trivial \cite[Theorem 2.11]{BFNPANF}.
The moduli of such blended extensions of \(\bZ\)-local systems over \(\Delta\backslash \{s\}\) form \cite[Proposition 188]{BPJump} \cite[Proposition 4.4.1]{HarBiext} a biextension of \((\Ext^1(H_{\bZ}|_{\Delta\backslash \{s\}},\bZ(1)_{\Delta\backslash \{s\}})_{\mathrm{tors}}, \Ext^1(\bZ_{\Delta\backslash \{s\}}, H_{\bZ}|_{\Delta\backslash \{s\}})_{\mathrm{tors}})\).
The biextension is classified \cite[Expos\'e VII, Corollaire 3.6.5]{SGA7} by
\begin{align} \label{biextTors}
    \Ext^1(H_{\bZ}|_{\Delta\backslash \{s\}},\bZ(1)_{\Delta\backslash \{s\}})\otimes_\bZ \Ext^1(\bZ_{\Delta\backslash \{s\}}, H_{\bZ}|_{\Delta\backslash \{s\}}) \to \Ext^1(\bZ_{\Delta\backslash \{s\}}, \bZ(1)_{\Delta\backslash \{s\}})[1]
\end{align}
in the derived category, where the derived tensor product is not needed since there are no homomorphisms from the corresponding \(\Tor_1\) to \(\bZ(1) = \Ext^1(\bZ_{\Delta\backslash \{s\}}, \bZ(1)_{\Delta\backslash \{s\}})\).
Especially, the mapping cone of \cref{biextTors} is the above moduli.
Our \cref{biextTors} induces
\begin{align} \label{TorsPair}
    \tau_s\colon \Ext^1(H_{\bZ}|_{\Delta\backslash \{s\}},\bZ(1)_{\Delta\backslash \{s\}})\otimes_\bZ \Ext^1(\bZ_{\Delta\backslash \{s\}}, H_{\bZ}|_{\Delta\backslash \{s\}}) \to \bQ(1)/\bZ(1)
\end{align}
via the short exact sequence \(0\to \bZ(1)\to \bQ(1)\to \bQ(1)/\bZ(1)\to 0\) and the torsionness.

The source of \cref{TorsPair} contains an element induced by \(B_\bZ\).
Its image has a lift \(\wt\tau_s(B)\in \bQ(1)\) defined as follows \cite[Proposition 192]{BPJump}.
By \cite[Theorem 2.11]{BFNPANF} cited above, \(B_\bQ\) induces the element \(\wt\tau(B)\) of the moduli \(\Ext^1(\bQ_{\Delta\backslash \{s\}}, \bQ(1)_{\Delta\backslash \{s\}}) = \bQ(1)\) of the biextensions of \(\bQ_{\Delta\backslash \{s\}}\) by \(H_{\bQ}|_{\Delta\backslash \{s\}}\) by \(\bQ(1)_{\Delta\backslash \{s\}}\) with trivial subextension and trivial quotient extension.

Now, we can formulate the extension of the fiber of \(\cB(H)\) to \(S(\bC)\) \cite[Theorem 241, Remark 242]{BPJump}.
Namely, take \(\nu\in Ext^1_{\AVMHS}(\bZ_{S_\sm(\bC)}, H)(S_\sm(\bC)), \omega\in Ext^1_{\AVMHS}(H, \bZ(1)_{S_\sm(\bC)})(S_\sm(\bC))\).
Then the holomorphic line bundle \(\cB(H)_{\omega, \nu}\) extends to the holomorphic \(\bQ\)-line bundle
\[
    \ol{\cB}(H)_{\omega, \nu} \coloneqq \frac{1}{N}(U\mapsto \{b\in \cB^{\ad}(H)_{\omega, N\nu}(U)\mid \wt\tau_s(b) = 0 \quad (s\in U\backslash S_\sm(\bC))\})
\]
on \(S\), where \(N\) is any integer such that \(N\tau_s(\omega, \nu) \in \bZ(1)\).

\section{The smooth case} \label{sec:ArchCompSm}

When \(\pi\) is smooth, we compare the biextension of \(\bfA^m(X/S)\) and \(\bfA^n(X/S)\) in \cite{BloBiext} with that in \cite[\S 3]{HainBiext}, although we do not impose the smoothness assumption until we explicitly do.
Our first main technical tool is the derived cycle map to Deligne cohomology in \cite[\S 8]{BSModulus} although maybe their construction was known in the case without modulus.

We first describe the cubical variant of \cref{biext} since \cite[\S 8]{BSModulus} uses cubical higher Chow groups.
Our \(t\) in \cref{biext} may be replaced by
\begin{align} \label{DefTCub}
\Delta_S^{-1}\pi^2_*z_c^{d+1}(*_{X^2}, -\bullet)_0\to z_c^1(*_S,-\bullet)_0
\end{align}
adjoint to
\begin{align*}
    &\pi^2_*z_c^{d+1}(*_{X^2}, -\bullet)_0\xleftarrow\simeq \pi^2_*z^{d+1}_{c, \{\Delta X\}}(*_{X^2},-\bullet)_0 \xrightarrow{\Delta_X^*} \pi^2_*\Delta_{X,*} z_c^{d+1}(*_X, -\bullet)_0 = \Delta_{S,*}\pi_* z_c^{d+1}(*_X, -\bullet)_0 \\
    &\xrightarrow{\pi_*} \Delta_{S,*} z_c^1(*_S,-\bullet)_0 \to R\Delta_{S,*} z_c^1(*_S,-\bullet)_0,
\end{align*}
where the first arrow is an isomorphism due to \cite[Theorem 1.10]{KLAddHighChow} applied to quasi-projective open subsets of \(S^2\).
Also, the last morphism is isomorphic thanks to Zariski descent \cite[Theorem 4.1.8]{ParLoc}.
We change \(z_{\hom}^m(*_{X/S}, -\bullet)\) by \(z_{c, \hom}^m(*_{X/S}, -\bullet)_0\), defined by switching the \(0\)-th term of \(\pi_*z_c^m(*_X, -\bullet)_0\) to its subsheaf of the generically homologically trivial cycles.
We similarly understand \(z_{c, \hom}^n(*_{X/S}, -\bullet)_0\).
As a substitute of \(*\times *\), use
\begin{align} \label{prelimTimesCub}
    &z_{c, \hom}^m(*_{X/S}, -\bullet)_0\otimes_\bZ z_{c, \hom}^n(*_{X/S}, -\bullet)_0\to \pi_*z_c^m(*_X, -\bullet)_0\otimes_\bZ \pi_*z_c^n(*_X, -\bullet)_0 \\
    =& \Delta_S^{-1}(\pr_{S,1}^{-1}\pi_*z_c^m(*_X, -\bullet)_0 \otimes_\bZ \pr_{S,2}^{-1}\pi_*z_c^n(*_X, -\bullet)_0) \nonumber \\
    \to& \Delta_S^{-1} \pi^2_*(\pi^2)^{-1}(\pr_{S,1}^{-1}\pi_*z_c^m(*_X, -\bullet)_0\otimes_\bZ \pr_{S,2}^{-1}\pi_*z_c^n(*_X, -\bullet)_0) \nonumber \\
    =& \Delta_S^{-1} \pi^2_*(\pr_{X,1}^{-1}\pi^{-1}\pi_*z_c^m(*_X, -\bullet)_0\otimes_\bZ \pr_{X,2}^{-1}\pi^{-1}\pi_*z_c^n(*_X, -\bullet)_0) \nonumber \\
    \to& \Delta_S^{-1} \pi^2_*(\pr_{X,1}^{-1}z_c^m(*_X, -\bullet)_0\otimes_\bZ \pr_{X,2}^{-1}z_c^n(*_X, -\bullet)_0) \nonumber \\
    \to& \Delta_S^{-1} \pi^2_*z_c^{d+1}(*_{X^2}, -\bullet)_0. \nonumber
\end{align}
Composing these and taking \(\tau_{\geqq -1}\) gives \cref{biext} by the compatibilities checked in \cref{sec:MotDef}.
We can define \(\langle W, Z\rangle\) similarly to and compatibly with the construction at the end of \cref{sec:Calc}.

\subsection*{The derived cycle map to Deligne cohomology}
We discuss the compatibility of \cite[\S 8.3]{BSModulus} with operations on cubical higher Chow groups and Deligne cohomology.
Although these compatibility have not been explicitly stated in the literature, they are consequences of the fact that Deligne cycle classes map to relevant Deligne cycle classes by those operations.
That, in turn, follows by checking the corresponding statements for the Betti and de Rham cycle classes.
For a smooth variety \(T\) over \(\bC\), we introduce the morphism \(\epsilon\colon T_{\an}\to T_{\Zar}\) of sites.
Also, \(\bZ (p)^{\cD}_T\) means the Deligne complex \cite[\S 8.2]{BSModulus}.
When \(F\subset T\) is a simple normal crossing divisor, we introduce the variant
\[
    \bZ (p)^\cD_{(T, F)}\to R\tau_* \bZ(p)_{(T\backslash F)(\bC)}\oplus \sigma_{\geqq p}\Omega_{T(\bC)}(\log F(\bC))^\an\to R\tau_*\bC_{(T\backslash F)(\bC)}\xrightarrow{+1},
\]
where \(\tau\colon T\backslash F\to T\) is the open immersion.

We begin with the outer product.
\begin{prop} \label{prop:CubDelProd}
    Let \(T, T'\) be smooth varieties over \(\bC\).
    Set \(p, p'\) to be nonnegative integers.
    Let \(\pr_1\colon T\times_\bC T'\to T\) and \(\pr_2\colon T\times_\bC T'\to T'\) be projections.
    The diagram
    \[
    \begin{tikzcd}
        \epsilon^{-1} (\pr_1^{-1}z_{c}^{p}(*_T, -\bullet)\otimes_\bZ \pr_2^{-1}z_{c}^{p'}(*_{T'}, -\bullet)) \arrow[r] \arrow[d, equal] &\epsilon^{-1}z_{c}^{p + p'}(*_{T\times_\bC T'}, -\bullet) \arrow[dd]\\
        \pr^{-1}\epsilon^{-1}z_{c}^{p}(*_T, -\bullet)\otimes_\bZ \pr_2^{-1}\epsilon^{-1}z_{c}^{p'}(*_{T'}, -\bullet) \arrow[d] & \\
        \pr^{-1}\bZ (p)^\cD_T[2p] \otimes_\bZ \pr_2^{-1}\bZ (p')^\cD_{T'}[2p'] \arrow[r] & \bZ (p + p')^\cD_{T\times_\bC T'}[2(p + p')]
    \end{tikzcd}
    \]
    commutes, where the two vertical morphisms are \cite[\S 8.3]{BSModulus}.
\end{prop}

\begin{proof}
    Take nonnegative integers \(n, n'\).
    Let \(W, W'\) be integral closed subvarieties of \(T\times\square^n\) and \(T'\times\square^{n'}\) of codimension \(p, p'\), respectively.
    Set \(\ol W, \ol{W'}\) be their closures in \(T\times (\bP^1)^{n}\) and \(T'\times (\bP^1)^{n'}\), respectively.
    Then
    \begin{align*}
        &H^{2p}_{\ol W(\bC)}((T\times_\bC (\bP^1_\bC)^n)(\bC), \bZ(p)_{(T\times (\bP^1)^n, T\times ((\bP^1)^n\backslash\square^n))}^\cD) \\
        &\otimes_\bZ H^{2p'}_{\ol{W'}(\bC)}((T'\times_\bC (\bP^1_\bC)^{n'})(\bC), \bZ(p')_{(T'\times (\bP^1)^{n'}, T'\times ((\bP^1)^{n'}\backslash\square^{n'}))}^\cD) \\
        \to& H^{2(p + p')}_{(\ol W\times_\bC \ol{W'})(\bC)}((T\times_\bC T'\times_\bC (\bP^1_\bC)^{n + n'})(\bC), \bZ(p + p')_{(T\times_\bC T'\times (\bP^1)^{n + n'}, T\times_\bC T'\times ((\bP^1)^{n + n'}\backslash\square^{n + n'}))}^\cD)
    \end{align*}
    carries \(cl_\cD^p(W)\otimes cl_\cD^{p'}(W')\) to \(cl_\cD^{p + p'}(W\times_\bC W')\) because of \cite[Proposition 3.7]{EVDB} and because those Deligne cycle classes are defined using the Betti and de Rham cycle classes \cite[Theorem 8.5]{BSModulus}\footnote{The author is not entirely sure about whether the sign in the construction \cite[\S 6]{EVDB} of the de Rham cycle classes is correct. We normalize those classes so that they are compatible with the corresponding Betti cycle classes.}, which satisfy similar compatibilities \cite[\S 3.2, Proposition 2]{ElZeiFond}.
    Our proposition follows by copying the proof of \cite[Proposition 4.7]{GLBK}, where we can skip the arguments around triangulation.
\end{proof}

We now focus on the compatibility of the pullback.
\begin{prop} \label{prop:CubDelPull}
    Let \(f\colon T'\to T\) be a closed immersion of smooth varieties over \(\bC\).
    Take a nonnegative integer \(p\).
    The following diagram, where the vertical morphisms are \cite[\S 8.3]{BSModulus}, commutes:
    \[
        \begin{tikzcd}
            f^{-1}\epsilon^{-1}z_{c}^{p}(*_T, -\bullet) \arrow[d] \arrow[r, equal] & \epsilon^{-1}f^{-1}z_{c}^{p}(*_T, -\bullet) \arrow[r, "f^*"] & \epsilon^{-1}z_{c}^{p}(*_{T'}, -\bullet) \arrow[d] \\
            f^{-1}\bZ^{\cD}_T(p)[2p]             \arrow[rr, "f^*"] & &   \bZ^{\cD}_{T'}(p)[2p].
        \end{tikzcd}
    \]
\end{prop}

\begin{proof}
    We take an integer \(n\geqq 0\).
    We may replace \(T\) by a quasi-projective open subscheme and consider \(W\in \underline{z}_{c, \{T'\}}^p(T, n)\) by \cite[Theorem 1.10]{KLAddHighChow}.
    Then we can define \(T'\cap W\in \underline{z}_{c}^p(T', n)\).
    Let \(\ol W\) be the closure of \(W\) in \(T\times (\bP^1)^n\).
    By
    \[
        H^{2p}_{\ol W(\bC)}(T(\bC), \bZ(p)_{(T\times (\bP^1)^n, T\times ((\bP^1)^n\backslash\square^n))}^\cD)\to H^{2p}_{(\ol W\cap T')(\bC)}(T'(\bC), \bZ(p)_{(T'\times (\bP^1)^n, T'\times ((\bP^1)^n\backslash\square^n))}^\cD),
    \]
    we have \(cl_\cD^p(W)\mapsto cl_\cD^p(T'\cap W)\) because of similar statements in Betti \cite[Proposition 19.2]{FulInt} and de Rham \cite[Lemma 7.3]{BSModulus} theory.
    This proves the desired compatibility.
\end{proof}

Next, we discuss the compatibility with pushforward.
\begin{prop} \label{prop:CubDelPush}
    Assume that \(\pi\) is smooth.
    Take an integer \(p\geqq 0\).
    The diagram
    \[
        \begin{tikzcd}
            \pi_*\epsilon^{-1} z_{c}^{p + d}(*_X, -\bullet) \arrow[d] & \epsilon^{-1}\pi_*z_{c}^{p + d}(*_X, -\bullet)  \arrow[l] \arrow[r, "\pi_*"] & \epsilon^{-1}z_c^p(*_S, -\bullet) \arrow[d]\\
            R\pi_*\bZ^{\cD}_X(p + d)[2p + 2d] \arrow[rr]              &                                                                              &   \bZ^{\cD}_S(p)[2p]
        \end{tikzcd}
    \]
    commutes, where the morphisms are as follows.
    \begin{itemize}
        \item   The downward maps are \cite[\S 8.3]{BSModulus}.
        \item   The bottom map is induced by the Verdier duality and integration along fibers, forming a well-defined map due to \cite[\S 5,6]{EHNPoin}, for example.
    \end{itemize}
\end{prop}

\begin{proof}
    As in the last two propositions and \cite[Lemma 4.4]{GLBK}, we check the compatibility of the Deligne cycle classes with the pushforward.
    For the Betti side, this follows from that of \cite[(7.12)]{BSModulus} and the Betti cycle classes \cite[Lemma 19.1.2]{FulInt}.
    The de Rham side is easily checked when we represent the sheaves of differential forms and cycle classes via currents.
    The compatibility of integration along fibers with the pushforward of currents is also known as projection formula. 
    The claimed commutativity follows.
\end{proof}

Finally, we compute the cycle class map on \(\CH^1(S, 1)\).

\begin{prop} \label{prop:CubDelGm}
    The following diagram commutes:
    \[
        \begin{CD}
            \cO_S(S)^\times @>>> \cO_{S(\bC)}(S(\bC))^\times \\
            @V{\textup{\cref{ex:CH1(1)Mix}}}VV    @| \\
            \CH^1(S, 1) @>{\textup{\cite[\S 8.3]{BSModulus}}}>> H_\cD^1(S(\bC), \bZ(1)).
        \end{CD}
    \]
\end{prop}

\begin{proof}
Put \(F\coloneqq \bP^1_S\backslash \square_S\).
Also let \(\tau\colon \square_S\to \bP^1_S\) be the open immersion.
Take \(f\in \cO_S(S)^\times\).
The morphism it induces is denoted by \(f\colon S\to \bA^1_S\to \bP^1_S\).
Let \(S'\to \square_S\) be its base change.
This is the image of \(f\) by the left vertical map.

In this paragraph, we calculate
\[
    cl_{S}^{1,1}(S')\in H_{f(S)(\bC)}^2(\bP^1_S(\bC), \bZ(1)^\cD_{(\bP^1_S, F)}),
\]
the Deligne cycle class characterized \cite[Theorem 8.5]{BSModulus} by its images in
\[
    H^2_{f(S)(\bC)}(\bP^1_S(\bC), R\tau_*\bZ(1)_{\square_S(\bC)}), \quad H^2_{f(S)(\bC)}(\bP^1_S(\bC), \sigma_{\geqq 1}\Omega_{\bP^1_S(\bC)}(\log F(\bC))).
\]
Here, the corresponding morphism
\begin{align} \label{D2BdR}
    \bZ(1)^\cD_{(\bP^1_S, F)}\to R\tau_*\bZ(1)_{\square_S(\bC)}\oplus \sigma_{\geqq 1} \Omega_{\bP^1_S(\bC)}(\log F(\bC))
\end{align}
is the negative of the shift of the projection from a cone, in compliance with the axiom regarding shifting distinguished triangles in triangulated categories.
The image in the Betti cohomology is the image of \(T_f \coloneqq (T - f)/(T - 1)\in \cO_{\square_S(\bC)}^\times((\square_S\backslash S')(\bC))\) for the coordinate function \(T\) on \(\square_S\cap\bA^1_S\) by connecting homomorphisms
\begin{align*}
    &\cO_{\square_S(\bC)}^\times((\square_S\backslash S')(\bC))\xrightarrow{\delta} H^1_{S'(\bC)}(\square_S(\bC), \cO^\times_{\square_S(\bC)})\to H^2_{S'(\bC)}(\square_S(\bC), \bZ(1)_{\square_S(\bC)}) \\
    \cong& H^2_{f(S)(\bC)}(\bP^1_S(\bC), R\tau_*\bZ(1)_{\square_S(\bC)}).
\end{align*}
The de Rham side admits further injections
\begin{align*}
    &H^2_{f(S)(\bC)}(\bP^1_S(\bC), \sigma_{\geqq 1}\Omega_{\bP^1_S(\bC)}(\log F(\bC)))\xrightarrow{\textup{\cite[(7.5)]{BSModulus}}} H^1_{f(S)(\bC)}(\bP^1_S(\bC), \Omega^1_{\bP^1_S(\bC)}(\log F(\bC))) \\
    \xrightarrow{\textup{\cite[(5.6)]{BSModulus}}} &H^2_{S'(\bC)}(\square_S(\bC), \Omega^1_{\square_S(\bC)}).
\end{align*}
The image in the most right side is that of \(d\log T_f\) \cite[Theorem 6.3]{EVDB}.
Namely, the image is \(d\log \delta (T_f)\).
We conclude that \(cl_{S}^{1,1}(S') = - \delta (T_f)\) by the convention on \cref{D2BdR}.

For an integer \(p\geqq 0\), we take the Godement resolution \(I_p(1)\) of \(\cO_{\square^p_S(\bC)}^\times[-1]\), quasi-isomorphic to Deligne complex.
Also take the projection \(\pr\colon \square_S^p\to S\).
Let \(\alpha\) be the lift of \(T_f\) to \(I_1(1)^1\).
By the previous paragraph, \(\partial_{\mathrm{II}} \alpha = -cl_S^{1,1}(S')\), where \(\partial_{\mathrm{II}}\) is induced by the differential of the Godement resolution.
The image of \(f\) by the left and bottom homomorphisms of the statement of the proposition is given by \(\partial_{\mathrm{I}} \alpha\), the differential with respect to the cubical structure on \(\pr_*I_{\bullet}(1)\), because the last element is cohomologous to \(\partial_{\mathrm{II}} \alpha\) in \(\pr_*I_{-\bullet}(1)^\bullet\).
However, \(\partial_{\mathrm{I}} \alpha = f/1 = f\), where our calculation builds on the fact that \(S\times_\bC\{0\}, S\times_\bC\{\infty\} \square_S \backslash S'\).
\end{proof}

Assume that \(\pi\) is smooth until the end of this section.
We now introduce the Deligne cohomology analog of \cref{biext} and compare it with the cubical analog.
For a nonnegative integer \(p\), put
\[
    \bZ(p)^\cD_{X/S} \colon \dots \to 0\to \bZ(p)_{X(\bC)}\to \cO_{X(\bC)}\to \Omega_{X(\bC)/S(\bC)}^1\to \cdots \to \Omega_{X(\bC)/S(\bC)}^{p - 1}\to 0 \to \cdots,
\]
where \(\bZ(p)_{X(\bC)}\) is placed in degree \(0\).
Also define the subcomplex
\[
    \tau_{\leqq 2p}^{\hom} R\pi_*\bZ(p)_{X/S}^\cD \to \tau_{\leqq 2p} R\pi_*\bZ(p)_{X/S}^\cD
\]
so that these complexes share the same terms except the \(2p\)-th one and so that
\[
    H^{2p}(\tau_{\leqq 2p}^{\hom} R\pi_*\bZ(p)_{X/S}^\cD) = \Ker (R^{2p}\pi_*\bZ(p)_{X/S}^\cD \to R^{2p}\pi_*\bZ(p)_{X(\bC)}).
\]
This equals \(Ext^1_{\pVMHS}(\bZ, R^{2p - 1}\pi_* \bZ(p)/\mathrm{tors})\) as in the case \(S = \Spec \bC\) using \cite[\S C.3]{GZFam}, with The sign convention similar to \cite[(2.2)]{HainBiext}.
The genuine analogy to \(t\circ \cref{prelimTimes}\) would give
\begin{align} \label{prelimD}
    &R\pi_*\bZ(m)_{X/S}^\cD[2m]\otimes_\bZ^\bL R\pi_*\bZ(n)_{X/S}^\cD[2n] \to \Delta_S^{-1} R(\pi^2)_*\bZ(d + 1)_{X^2/S^2}^\cD[2d + 2] \\
    \to& \Delta_S^{-1} R(\pi^2)_*R\Delta_{X, *} \bZ(d + 1)_{X/S}^\cD [2d + 2] = \Delta_S^{-1} R\Delta_{S, *}R\pi_* \bZ(d + 1)_{X/S}^\cD [2d + 2] \nonumber \\
    \to& R\pi_{*} \bZ(d + 1)_{X/S}^\cD[2d + 2] \to \bZ(1)_{S}^{\cD}[2] = \bG_{\mathrm{m}, S}^{\an}[1]. \nonumber
\end{align}
We compose \cref{prelimD} with \((\tau_{\leqq 2m}^{\hom} R\pi_*\bZ(m)_{X/S}^\cD)[2m]\otimes_\bZ^\bL (\tau_{\leqq 2n}^{\hom} R\pi_*\bZ(n)_{X/S}^\cD)[2n]\).
After taking a quotient via the fiberwise application of \cref{lem:thetaDel} for \cref{lem:theta} and after taking \(\tau_{\geqq -1}\), we can take advantage of \cref{lem:HomAlg} to get
\begin{align} \label{biextD}
    Ext^1_{\pVMHS}(\bZ, R^{2m - 1}\pi_* \bZ(m)/\mathrm{tors})\otimes_\bZ^\bL Ext^1_{\pVMHS}(\bZ, R^{2n - 1}\pi_* \bZ(n)/\mathrm{tors})\to \bG_{\mathrm{m}, S}^{\an}[1].
\end{align}

\begin{prop}
The following diagram commutes:
\[
    \begin{CD}
        \epsilon^{-1}\bA^m(X/S)\otimes_\bZ^\bL \epsilon^{-1}\bA^n(X/S) @>{\cref{biext}}>> \epsilon^{-1}\Gm[1] \\
        @V{\textup{\cite[\S 8.3]{BSModulus}}\otimes \textup{\cite[\S 8.3]{BSModulus}}}VV                                           @VVV \\
        Ext^1_{\pVMHS}(\bZ, R^{2m - 1}\pi_* \bZ(m)_{X(\bC)}/\mathrm{tors})\otimes_\bZ^\bL Ext^1_{\pVMHS}(\bZ, R^{2n - 1}\pi_* \bZ(n)_{X(\bC)}/\mathrm{tors}) @>{\cref{biextD}}>> \bG_{\mathrm{m}, S}^\an[1].
    \end{CD}
\]
\end{prop}

\begin{proof}
Our \(\Delta_S^{-1}\cref{DefTCub}\circ \cref{prelimTimesCub}\) is compatible with
\begin{align*}
    &R\pi_*\bZ(m)_{X}^\cD[2m]\otimes^\bL R\pi_*\bZ(n)_X^\cD[2n] \to \Delta_S^{-1} R(\pi^2)_*\bZ(d + 1)_{X^2}^\cD[2d + 2] \\
    \to& \Delta_S^{-1} R(\pi^2)_*R\Delta_{X, *} \bZ(d + 1)_{X}^\cD [2d + 2] = \Delta_S^{-1} R\Delta_{S, *}R\pi_* \bZ(d + 1)_{X}^\cD [2d + 2] \\
    \to& R\pi_{*} \bZ(d + 1)_X^\cD[2d + 2] \to \bZ(1)_S^\cD[2] = \bG_{\mathrm{m}, S}^\an[1]
\end{align*}
by \cref{prop:CubDelProd,prop:CubDelPull,prop:CubDelPush,prop:CubDelGm}.
The last morphism forms a commutative diagram with \cref{prelimD} via the obvious morphism \(\Omega_{X(\bC)}^1\to \Omega_{X(\bC)/S(\bC)}^1\), etc.
Note that
\[
\epsilon^{-1}z^m_{c, \hom}(*_{X/S}, -\bullet)_0 \to \epsilon^{-1}\pi_*z^m_{c}(*_{X}, -\bullet)_0 \to R\pi_*\bZ(m)_{X}^\cD \to R\pi_*\bZ(m)_{X/S}^\cD
\]
factors through \(\tau_{\leqq 2m}^{\hom} R\pi_*\bZ(m)_{X/S}^\cD\).
The same is true when we replace \(m\) by \(n\).
These imply the desired result.
\end{proof}

\newgeometry{margin=10truemm}
\subsection*{Comparison of sections}
Here, we will compare \(\langle W, Z\rangle\) and \(B_{W, Z}^\circ\) as well as the biextensions that have them as sections.
Our \(\pi\) is still smooth.
For this, let \(W, Z\) be disjoint homologically trivial cycles on \(X\) of codimension \(m, n\).

We recall the image of \(\langle W, Z\rangle\) in the biextension defined by \cref{biextD}.
Here, we switch the role of \(W, Z\) from the construction in \cref{sec:Calc} for the convenience around \cref{clm:Jac(EZDual)}.
Take a distinguished triangle
\[
    R\pi_*\bZ(m)_{X/S}^\cD[2m]\to \bG_{\mathrm{m}, S}^\an [1] \to C_\cD \xrightarrow{+1}
\]
whose first morphism is
\begin{align} \label{biextDZ}
    R\pi_*\bZ(m)_{X/S}^\cD[2m]\xrightarrow{[Z]_\cD} R\pi_*\bZ(d + 1)_{X/S}^\cD[2d + 2]\xrightarrow{\pi_*}\bZ(1)_S^\cD[2]\simeq \bG_{\mathrm{m}, S}^\an [1],
\end{align}
where \([Z]_\cD\in H^{2n}(X(\bC), \bZ(n)_{X/S}^\cD)\) is the image of \([Z]_\cD\in H^{2n}(X(\bC), \bZ(n)_{X}^\cD)\).
It induces the long exact sequence
\[
    R^{2m - 1} \pi_*\bZ(m)_{X/S}^\cD \xrightarrow{0} \bG_{\mathrm{m}, S}^\an \to H^{-1}(C_\cD)\to R^{2m} \pi_*\bZ(m)_{X/S}^\cD \to 0,
\]
where applying \cref{lem:thetaDel} fiberwise kills the leftmost morphism.
This short exact sequence is the fiber of the biextension \cref{biextD} at \([Z]_\cD = E_{-Z}\), where the last equality was mentioned in \cref{rem:sign}.

By \cref{prop:CubDelProd,prop:CubDelPull,prop:CubDelPush,prop:CubDelGm}, our \cref{biextDZ} is compatible with
\[
    \pi_*z^m_{c,\{Z\}}(*_{X}, -\bullet)\to \pi_*z^{d + 1}_{c}(*_{X}, -\bullet)\to z^1(*_S, -\bullet).
\]
We need a more precise statement.
For a smooth variety \(T\) over \(\bC\) and a nonnegative integer \(p,q\), let \(I_q(p)_T\) be the Godement resolution of \(\bZ (p)^\cD_{((\bP^1_T)^q, T\times ((\bP^1)^q\backslash \square^q))}\).
As a variant, set \(I_0(p)_{X/S}\) to be the Godement resolution of \(\bZ (p)^\cD_{X/S}\).
For a finite set \(\sW\) of closed subsets of \(T\), let \(\sS_{T, \sW}^{p,q}\) be the set of those closed subsets \(V\) of \(T\times \square^q\) of pure codimension \(p\) such that \(V\) and each face of \(W_0\times \square^q\) intersect properly for every \(W_0\in \sW\) or \(W_0 = T\).
Let \(U\) be an open subset of \(S\).
The proofs of \cref{prop:CubDelProd,prop:CubDelPull,prop:CubDelGm} tell us that the following is a commutative diagram of complexes without any homotopy:
\[
{\fontsize{7truept}{7truept}\selectfont
\begin{tikzcd}[cramped, column sep=tiny]
    z_{c, \{Z\cap \pi^{-1}(U)\}}^m(\pi^{-1}(U), -\bullet) \arrow[d] \arrow[r]          & \varinjlim_{\sS_{\pi^{-1}(U), \{Z\cap \pi^{-1}(U)\}}^{m, -\bullet}} \frac{H^{2m}_{\overline{W_0}}\left(\bP_{\pi^{-1}(U)}, \bZ(m)^\cD_{(\bP_X, \partial)}\right)}{\mathrm{degn}} \arrow[d]                  & \varinjlim_{\sS_{U, \{Z\cap \pi^{-1}(U)\}}^{m, -\bullet}} \frac{\tau_{\leqq 2m}\Gamma_{\overline{W_0}}\left(\bP_{\pi^{-1}(U)}, I_{-\bullet}(m)_{X}\right)}{\mathrm{degn}}[2m] \arrow[l, "\simeq"] \arrow[d] \\
    z_{c, \{\Delta \pi^{-1}(U)\}}^{d + 1}(\pi^{-1}(U)^2, -\bullet) \arrow[r] \arrow[d] & \varinjlim_{\sS_{\pi^{-1}(U)^2, \{\Delta\pi^{-1}(U)\}}^{d + 1, -\bullet}} \frac{H^{2d^+}_{\overline{W_0}}\left(\bP_{\pi^{-1}(U)^2}, \bZ(d + 1)^\cD_{(\bP_{X^2}, \partial)}\right)}{\mathrm{degn}} \arrow[d] & \varinjlim_{\sS_{\pi^{-1}(U)^2, \{\Delta\pi^{-1}(U)\}}^{d + 1, -\bullet}} \frac{\tau_{\leqq 2d^+}\Gamma_{\overline{W_0}}\left(\bP_{\pi^{-1}(U)^2}, I_{-\bullet}(d + 1)_{X^2}\right)}{\mathrm{degn}}[2d^+] \arrow[l, "\simeq"] \arrow[d] \\
    z_c^{d + 1}(\pi^{-1}(U), -\bullet)  \arrow[r]                             & \varinjlim_{\sS_{\pi^{-1}(U), \emptyset}^{d + 1, -\bullet}} \frac{H^{2d^+}_{\overline{W_0}}\left(\bP_{\pi^{-1}(U)}, \bZ(d + 1)^\cD_{(\bP_X, \partial)}\right)}{\mathrm{degn}}                  & \varinjlim_{\sS_{\pi^{-1}(U), \emptyset}^{d + 1, -\bullet}} \frac{\tau_{\leqq 2d^+}\Gamma_{\overline{W_0}}\left(\bP_{\pi^{-1}(U)}, I_{-\bullet}(d + 1)_{X}\right)}{\mathrm{degn}}[2d^+] \arrow[l, "\simeq"]
\end{tikzcd}}
\]
\[
{\fontsize{7truept}{7truept}\selectfont
\begin{tikzcd}[cramped, column sep=tiny]
    \varinjlim_{\sS_{U, \{Z\cap \pi^{-1}(U)\}}^{m, -\bullet}} \frac{\tau_{\leqq 2m}\Gamma_{\overline{W_0}}\left(\bP_{\pi^{-1}(U)}, I_{-\bullet}(m)_{X}\right)}{\mathrm{degn}}[2m] \arrow[r] \arrow[d]                               & \frac{I_{-\bullet}(m)_{X}(\bP_{\pi^{-1}(U)})}{\mathrm{degn}}[2m] \arrow[d] & I_0(m)_{X}(\pi^{-1}(U))[2m] \arrow[l, "\simeq"] \arrow[r] \arrow[d] & I_0(m)_{X/S}(\pi^{-1}(U))[2m] \arrow[d]\\
    \varinjlim_{\sS_{\pi^{-1}(U)^2, \{\Delta\pi^{-1}(U)\}}^{d + 1, -\bullet}} \frac{\tau_{\leqq 2d^+}\Gamma_{\overline{W_0}}\left(\bP_{\pi^{-1}(U)^2}, I_{-\bullet}(d + 1)_{X^2}\right)}{\mathrm{degn}}[2d^+] \arrow[r] \arrow[d] & \frac{I_{-\bullet}(d + 1)_{X^2}(\bP_{\pi^{-1}(U)^2})}{\mathrm{degn}}[2d^+] \arrow[d] & I_0(d + 1)_{X^2}(\pi^{-1}(U)^2)[2d^+] \arrow[l, "\simeq"] \arrow[r] \arrow[d] & I_0(d + 1)_{X^2/S^2}(\pi^{-1}(U)^2)[2d^+] \arrow[d]\\
    \varinjlim_{\sS_{\pi^{-1}(U), \emptyset}^{d + 1, -\bullet}} \frac{\tau_{\leqq 2d^+}\Gamma_{\overline{W_0}}\left(\bP_{\pi^{-1}(U)}, I_{-\bullet}(d + 1)_{X}\right)}{\mathrm{degn}}[2d^+] \arrow[r]                        & \frac{I_{-\bullet}(d + 1)_{X}(\bP_{\pi^{-1}(U)})}{\mathrm{degn}}[2d^+]  & I_0(d + 1)_{X}(\pi^{-1}(U))[2d^+] \arrow[l, "\simeq"] \arrow[r]  & I_0(d + 1)_{X/S}(\pi^{-1}(U))[2d^+], \\    
\end{tikzcd}}
\]
with the following explanations.
\begin{itemize}
    \item In all of the injective limits, the running variables are \(W_0\).
    \item We mean the degenerate part of cubical objects by \(\mathrm{degn}\).
    \item In the second, third and fourth columns, we are taking the total complex of the double complex induced by the cubical structure, leading to the first index, and the complexes in each term, giving rise to the second index.
    \item The morphisms from the first to the second row are given by choosing an element of
    \begin{align} \label{LiftDel}
        (\tau_{\leqq 2n}\Gamma_{|Z|(\bC)}\left(X(\bC), I_0(n)_{X}\right))^{2n}
    \end{align}
    that lifts \([Z]_\cD\) and by using the monoidal property of the Godement resolutions \cite[\S 6.4, e)]{Godement} \cite[\S 3.1]{LevCub}.
    \item Put \(\bP\coloneqq (\bP^1)^{-\bullet}\).
\end{itemize}
\restoregeometry
\begin{itemize}
    \item Let \(\partial\) be the product of suitable schemes with \((\bP^1)^{-\bullet}\backslash \square^{-\bullet}\).
    \item Let \(\Gamma\) be the global section. If it has a subscript, then that means the support.
    \item Put \(d^+\coloneqq d + 1\).
    \item We omit the notation of taking the topological space of the \(\bC\)-valued points of schemes.
    \item Quasi-isomorphisms are meant by \(\simeq\).
    \item We do not have to think about \(\pi_{*} \colon R\pi_{*}\bZ(d + 1)_{X/S}^\cD[2d + 2] \to \bZ(1)_{S}^{\cD}[2]\) because this is isomorphic after applying \(\tau_{\geqq -1}\).
\end{itemize}
These form morphisms of complexes of sheaves for varying \(U\), too.
Namely, there is a commutative diagram, again with no homotopy, of the following form, where \(\pr\) are some base changes of \(\bP_\bC\to \Spec \bC\):
\begin{equation} \label{BigNoHtpy}
{\fontsize{7truept}{7truept}\selectfont
\begin{tikzcd}[cramped, column sep=small]
    \epsilon^{-1}\pi_*z_{c, \{Z\}}^m(*_X, -\bullet) \arrow[d] \arrow[r]                                 & * \arrow[d] & * \arrow[l, "\simeq"] \arrow[r] \arrow[d] & \frac{\pi_*\pr_*I_{-\bullet}(m)_{X}}{\mathrm{degn}}[2m] \arrow[d]                       & \pi_*I_0(m)_{X}[2m] \arrow[l, "\simeq"] \arrow[r] \arrow[d] & \pi_*I_0(m)_{X/S}[2m] \arrow[d]\\
    \epsilon^{-1}\Delta_S^{-1}\pi^2_*z_{c, \{\Delta X\}}^{d + 1}(*_{X^2}, -\bullet) \arrow[r] \arrow[d] & * \arrow[d] & * \arrow[l, "\simeq"] \arrow[r] \arrow[d] & \Delta_S^{-1}\frac{\pi^2_*\pr_*I_{-\bullet}(d + 1)_{X^2}}{\mathrm{degn}}[2d^+] \arrow[d] & \Delta_S^{-1}\pi^2_*I_0(d + 1)_{X^2}[2d^+] \arrow[l, "\simeq"] \arrow[r] \arrow[d] & \Delta_S^{-1}\pi^2_*I_0(d + 1)_{X^2/S^2}[2d^+] \arrow[d]\\
    \epsilon^{-1}\pi_* z_c^{d + 1}(*_X, -\bullet)  \arrow[r]                         & *  & * \arrow[l, "\simeq"] \arrow[r]  & \frac{\pi_*\pr_*I_{-\bullet}(d + 1)_{X}}{\mathrm{degn}}[2d^+]  & \pi_*I_0(d + 1)_{X}[2d^+] \arrow[l, "\simeq"] \arrow[r]  & \pi_*I_0(d + 1)_{X/S}[2d^+]. \\
\end{tikzcd}}
\end{equation}
By \cref{lem:NoHtpy}, we have some commutative diagram
\begin{align} \label{biextZVSbiextDZ}
\begin{CD}
\epsilon^{-1}\pi_*z_{c, \{Z\}}^m(*_X, -\bullet) @>>> * @<{\simeq}<< \pi_*I_0(m)_{X/S}[2m] \\
@VVV @VVV @VVV \\
\epsilon^{-1}\Delta_S^{-1}\pi^2_*z_{c, \{\Delta X\}}^{d + 1}(*_{X^2}, -\bullet) @>>> * @<{\simeq}<< \Delta_S^{-1}\pi^2_*I_0(d + 1)_{X^2/S^2}[2d + 2] \\
@VVV @VVV @VVV \\
\epsilon^{-1}\pi_* z_c^{d + 1}(*_X, -\bullet) @>>> * @<{\simeq}<< \pi_*I_0(d + 1)_{X/S}[2d + 2]
\end{CD}
\end{align}
of complexes, not up to homotopy, with \(\simeq\) quasi-isomorphisms.
By taking the mapping cone of the left and middle columns of this diagram, we obtain an isomorphism of the analytification of the fiber \(\bE_Z\coloneqq \bE_{[Z]}\) and the fiber of the biextension \cref{biextD} at \([Z]_\cD\) as a section of \(\Ker (R^{2n}\pi_*\bZ(n)^\cD_{X/S}\to R^{2n}\pi_*\bZ(n)_{X(\bC)})\).
We have sections \(\langle W, Z\rangle\) of \(\bE_Z\) for each \(W\).
Let \(\langle W, Z\rangle_\cD\) be the image of these in the fiber of the biextension \cref{biextD}.
Note the asymmetry of \(W, Z\) here.
We do not pursue the symmetry of these sections further.

We prefer the following description of \(\langle W, Z\rangle_\cD\).
Let \(i_W, i_Z\) be as in \cref{sec:Arch}.
Since \([Z]_\cD\) can be seen as a global section of \(R^{2n}\pi_*i_{Z, *}Ri_Z^! \bZ(n)_{X/S}^\cD\), the restriction of \cref{biextDZ} to \(R\pi_*i_{W, *}Ri_W^! \bZ(m)_{X/S}^\cD[2m]\) vanishes.
Therefore, we have
\begin{align} \label{CycCl<WZ>D}
    R^{2m} \pi_*i_{W, *}Ri_W^! \bZ(m)_{X/S}^\cD\to H^{-1}(C_\cD).
\end{align}
We want to show that this carries \([W]_\cD\) as an element of \(H^{2m}_{|W|(\bC)}(X(\bC), \bZ(m)_{X/S}^\cD)\) to \(\langle W, Z\rangle_\cD\).

As before, we need to kill the unnecessary homotopy for our purpose.
The diagram
\[
{\fontsize{7truept}{7truept}\selectfont
    \begin{tikzcd}[cramped, column sep=small]
        \mathbb{Z} \arrow[d] \arrow[r, "\text{[}W\text{]}_{\mathcal{D}}"] & H^{2m}_{|W|}(\pi^{-1}(U), \bZ(m)_{X}^\cD) \arrow[d]                                                                                                                         & \tau_{\leqq 2m}\Gamma_{|W|}(\pi^{-1}(U), I_0(m)_{X})[2m] \arrow[d] \arrow[l, "\simeq"]\\
        z_{c, \{Z\cap \pi^{-1}(U)\}}^m(\pi^{-1}(U), -\bullet)  \arrow[r] & \displaystyle\varinjlim_{\sS_{\pi^{-1}(U), \{Z\cap \pi^{-1}(U)\}}^{m, -\bullet}} \frac{H^{2m}_{\overline{W_0}}\left(\bP_{\pi^{-1}(U)}, \bZ(m)^\cD_{(\bP_X, \partial)}\right)}{\mathrm{degn}} & \displaystyle\varinjlim_{\sS_{U, \{Z\cap \pi^{-1}(U)\}}^{m, -\bullet}} \frac{\tau_{\leqq 2m}\Gamma_{\overline{W_0}}\left(\bP_{\pi^{-1}(U)}, I_{-\bullet}(m)_{X}\right)}{\mathrm{degn}}[2m] \arrow[l, "\simeq"]
    \end{tikzcd}}
\]
\[
{\fontsize{7truept}{7truept}\selectfont
    \begin{tikzcd}[cramped, column sep=tiny]
        \tau_{\leqq 2m}\Gamma_{|W|}(\pi^{-1}(U), I_0(m)_{X})[2m] \arrow[d] \arrow[r]                                                                                                        & \Gamma_{|W|}(\pi^{-1}(U), I_0(m)_{X})[2m] \arrow[d] & \Gamma_{|W|}(\pi^{-1}(U), I_0(m)_{X})[2m] \arrow[d] \arrow[l, equal] \arrow[r] & \Gamma_{|W|}(\pi^{-1}(U), I_0(m)_{X/S})[2m] \arrow[d]\\
        \displaystyle\varinjlim_{\sS_{U, \{Z\cap \pi^{-1}(U)\}}^{m, -\bullet}} \frac{\tau_{\leqq 2m}\Gamma_{\overline{W_0}}\left(\bP_{\pi^{-1}(U)}, I_{-\bullet}(m)_{X}\right)}{\mathrm{degn}}[2m] \arrow[r] & \frac{I_{-\bullet}(m)_{X}(\bP_{\pi^{-1}(U)})}{\mathrm{degn}}[2m]  & I_0(m)_{X}(\pi^{-1}(U))[2m] \arrow[l, "\simeq"] \arrow[r]  & I_0(m)_{X/S}(\pi^{-1}(U))[2m]
    \end{tikzcd}}
\]
in the category of complexes, not just the homotopy category, is commutative for each open subset \(U\subseteq S\) and with notations as before.
This diagram in turn gives
\[
\begin{tikzcd}[cramped, column sep=tiny]
    \bZ_{X(\bC)} \arrow[d] \arrow[r, "\text{[}W\text{]}_{\mathcal{D}}"] & R^{2m}\pi_*i_{W, *}Ri_W^! \bZ(m)_{X}^\cD \arrow[d] & \tau_{\leqq 2m}\pi_*i_{W, *}i_W^! I_0(m)_{X}[2m] \arrow[d] \arrow[l, "\simeq"] \arrow[r] & \pi_*i_{W, *}i_W^! I_0(m)_{X}[2m] \arrow[d] \\
    \epsilon^{-1}\pi_*z_{c, \{Z\}}^m(*_X, -\bullet)  \arrow[r] & *                                                    & * \arrow[l, "\simeq"] \arrow[r]                                                          & \frac{\pi_*\pr_*I_{-\bullet}(m)_{X}}{\mathrm{degn}}[2m]
\end{tikzcd}
\]
\[
    \begin{tikzcd}
        \pi_*i_{W, *}i_W^! I_0(m)_{X}[2m] \arrow[d]             & \pi_*i_{W, *}i_W^! I_0(m)_{X}[2m] \arrow[l, equal] \arrow[d] \arrow[r] & \pi_*i_{W, *}i_W^! I_0(m)_{X/S}[2m] \arrow[d] \\
        \frac{\pi_*\pr_*I_{-\bullet}(m)_{X}}{\mathrm{degn}}[2m] & \pi_*I_0(m)_{X}[2m] \arrow[l, "\simeq"] \arrow[r]                     & \pi_*I_0(m)_{X/S}[2m],
    \end{tikzcd}
\]
commutative morphisms of complexes of sheaves, again without allowing homotopy.
Since all the vertical morphisms here vanish when composed with the corresponding three vertical morphisms in \cref{BigNoHtpy}, we can enhance \cref{biextZVSbiextDZ} to a commutative diagram
\[
\begin{tikzcd}[cramped, column sep=tiny]
\bZ_{X(\bC)} \arrow[dddd] \arrow[rd, "\text{[}W\text{]}_{\cD}"] \arrow[rr] & & * \arrow[d] \arrow[dddd, bend left] & & \pi_*i_{W, *}i_W^! I_0(m)_{X/S}[2m] \arrow[ll, "\simeq"'] \arrow[dddd] \arrow[ld] \\
 & \epsilon^{-1}\pi_* z_{c,\{Z\}}^{m}(*_X,-\bullet)
    \arrow[r]
    \arrow[d] &
* \arrow[d] &
\pi_* I_0(m)_{X/S}[2m]
    \arrow[l, "\simeq"']
    \arrow[d] & \\
 & \epsilon^{-1}\Delta_S^{-1}\pi^2_* z_{c,\{\Delta X\}}^{d+1}(*_{X^2},-\bullet)
    \arrow[r]
    \arrow[d] &
* \arrow[d] &
\Delta_S^{-1}\pi^2_* I_0(d+1)_{X^2/S^2}[2d+2]
    \arrow[l, "\simeq"']
    \arrow[d] & \\
 & \epsilon^{-1}\pi_* z_{c}^{d+1}(*_X,-\bullet)
    \arrow[r] &
* &
\pi_* I_0(d+1)_{X/S}[2d+2]
    \arrow[l, "\simeq"'] & \\
0 \arrow[ru] \arrow[rr] & & 0 \arrow[u] & & 0 \arrow[lu] \arrow[ll, "\simeq"]
\end{tikzcd}
\]
of complexes, with no allowance for homotopy.
Taking the \(-1\)-st cohomology of the mapping cone of the leftmost and the second most left columns, respectively consisting of one and two morphisms, we obtain \(\bZ_{X(\bC)}\to \bE_Z\), which maps \(1\) to \(\langle W, Z\rangle\).
Doing the same for the right two columns, we recover \cref{CycCl<WZ>D} for certain choices of homotopy.
Our \cref{CycCl<WZ>D} gives \(\langle W, Z\rangle_\cD\) by looking at the middle column.

\begin{prop} \label{prop:<WZ>DVSBWZ}
    Suppose that \(Z\) intersects properly with each fiber of \(\pi\).
    There exists an isomorphism of the fibers of the biextensions \cref{biextD} and \(\cB(R^{2m-1}\pi_*\bZ(m)_{X(\bC)}/\mathrm{tors})\) over \(E_{-Z}\) that does not preserve \(\mathbb{G}_{\mathrm{m}, S}^{\an}\) in those fibers but does up to sign and that carries \(\langle W, Z\rangle_\cD\) to \(B_{W, Z}^{\circ}\) for each \(W\) intersecting properly with the fibers of \(\pi\).
\end{prop}

\begin{proof}
We mimic the proof of \cite[Theorem 7.11]{EVDB}.
However, we first write down everything in terms of Godement resolutions to avoid the ambiguity of homotopy.
Set \(j_W\colon X\backslash |W|\to X\) to be the open immersion.
Also, put \(\pi_W\coloneqq \pi\circ j_W\).
Besides \(I_0(p)_{X}, I_0(p)_{X/S}\) for an integer \(p\), we take the Godement resolutions
\[
    \bZ(p)_{X(\bC)}\to I_\bZ(p)_{X}, \quad \sigma_{\geqq p}\Omega_{X(\bC)/S(\bC)}\to I_F(p)_{X/S}, \quad \pi^{-1}\cO_S\to I_{\cO, X/S}.
\]
Let \(C_0, C_{\bZ, F}, C_\cO, C_{\bZ, F, *}, C_{\cO, *}\) be the cones of
\begin{align}
    &\pi_*I_0(m)_{X/S}[2m]\to \pi_*I_0(d + 1)_{X/S}[2d + 2], \nonumber \\
    &\pi_*I_\bZ(m)_{X}[2m]\oplus \pi_*I_F(m)_{X/S}[2m]\to \pi_*I_\bZ(d + 1)_{X}[2d + 2], \label{C_ZF} \\
    &\pi_*I_{\cO, X/S}[2m]\to \pi_*I_{\cO, X/S}[2d + 2], \nonumber \\
    &\pi_{W, *}j_W^{-1}I_\bZ(m)_X[2m]\oplus \pi_{W, *}j_W^{-1}I_F(m)_{X/S}[2m]\to \pi_*I_\bZ(d + 1)_{X}[2d + 2], \label{C_ZFj_W}\\
    &\pi_{W, *}j_W^{-1}I_{\cO, X/S}[2m]\to \pi_*I_{\cO, X/S}[2d + 2], \nonumber
\end{align}
defined by multiplying the lifts in \cref{LiftDel} of \([Z]_\cD\) and \([Z]_\sing \in H^{2n}(X(\bC), \bC)\) to
\[
    (\tau_{\leqq 2n}\Gamma_{|Z|(\bC)}(X(\bC), I_0(n)_{X}))^{2n}, \quad (\tau_{\leqq 2n}\Gamma_{|Z|(\bC)}(X(\bC), I_{\cO, X/S}))^{2n}.
\]
There are short exact sequences
\[
    \begin{tikzcd}[cramped, column sep=tiny]
                & 0                             & 0                       & 0                              &   \\
    0 \arrow[r] & C_{\cO,*} \arrow[u] \arrow[r] & {*} \arrow[r] \arrow[u] & C_{\bZ, F, *}[1] \arrow[r] \arrow[u] & 0 \\
    0 \arrow[r] & C_{\cO} \arrow[r] \arrow[u] & C_{0}[1] \arrow[r] \arrow[u] & C_{\bZ, F}[1] \arrow[r] \arrow[u] & 0 \\
    0 \arrow[r] & \pi_* i_{W,*} i_W^! I_{\cO, X/S}[2m+1] \arrow[r] \arrow[u] &
    \pi_* i_{W,*} i_W^! I_{0}(m)_{X/S}[2m+2] \arrow[r] \arrow[u] &
    \begin{matrix}
    \pi_* i_{W,*} i_W^! I_{\bZ}(m)_X[2m+2]\\
    \oplus \pi_* i_{W,*} i_W^! I_{F}(m)_{X/S}[2m+2]
    \end{matrix}
    \arrow[r] \arrow[u] & 0 \\
        & 0 \arrow[u] & 0 \arrow[u] & 0, \arrow[u] &
    \end{tikzcd}
\]
where all the vertical morphisms are by the functoriality, inclusion, projection, or \cref{lem:3x3}, and where all the squares are commutative without homotopy except the right bottom one, which is anti-commutative due to the same lemma.
This induces the commutative diagram
\begin{equation} \label{EVDB7.11}
    \begin{tikzcd}[cramped, column sep=tiny]
    0 \arrow[r] &
    \displaystyle\frac{H^{-2}(C_\cO)}{H^{-2}(C_{\bZ, F})} \arrow[r, "-1"] &
    H^{-1}(C_0) \arrow[r] &
    H^{-1}(C_{\bZ, F}) \arrow[r] &
    H^{-1}(C_\cO) \\
    &
    0 \arrow[u] \arrow[r] &
    H^{2m}(\pi_* i_{W,*} i_W^! I_{0}(m)_{X/S}) \arrow[u] \arrow[r, "-1"] &
    \begin{matrix}
    H^{2m}(\pi_* i_{W,*} i_W^! I_\bZ(m)_X) \oplus \\
    H^{2m}(\pi_* i_{W,*} i_W^! I_F(m)_{X/S})
    \end{matrix}
    \arrow[u] \arrow[r] &
    H^{2m}(\pi_* i_{W,*} i_W^! I_{\cO, X/S}) \arrow[u] \\
    & & &
    H^{-2}(C_{\bZ, F, *}) \arrow[u] \arrow[r] &
    H^{-2}(C_{\cO,*}) \arrow[u] \\
     & & & & H^{-2}(C_\cO) \arrow[u]
    \end{tikzcd}
\end{equation}
with exact rows and columns, which we explain as below.
\begin{itemize}
    \item The symbol \(-1\) indicates that the morphism is that number times the one induced from the morphism in the previous diagram.
        We explain why we need \(-1\) in the top row later.
        The morphism with \(-1\) in the second row is the correct morphism from the Deligne cohomology to the Betti cohomology and the filtration of the de Rham cohomology because of the definition of the Deligne cohomology using mapping fiber and the convention regarding distinguished triangles.
        Multiplying \(-1\) here also switches the anticommutativity to the commutativity.
    \item The rightmost horizontal morphisms are induced by the functoriality rather than being the negatives of them.
        The latter would be induced by the triangles obtained out of the horizontal short exact sequences.
        Rather than that, for those horizontal morphisms in the diagram, we are following the convention of the snake lemma.
        Because of this, we need \(-1\) in the top row so that the signs in that row is induced from some distinguished triangle.
    \item The vertical morphisms from cohomology of the cones with \({}_*\) in their notations do follow the signs of the triangles induced by vertical short exact sequences, and are different from the signs in the proof of the snake lemma.
\end{itemize}

We have \(H^{-1}(C_{\bZ, F})\simeq R^{2m}\pi_*\bZ(m)_{X(\bC)}\oplus F^mR^{2m}\pi_*\pi^{-1}\cO_{S(\bC)}\).
The kernel of
\[
    H^{-1}(C_0) \to H^{-1}(C_{\bZ, F})
\]
in \cref{EVDB7.11} is the fiber of \cref{biextD} over \(E_{-Z}\).

\begin{clm} \label{clm:Jac(EZDual)}
There exists an isomorphism \(H^{-2}(C_\cO)/H^{-2}(C_{\bZ, F})\simeq Ext^1_{\pVMHS}(\bZ, E_{Z}^\vee)\).
\end{clm}

\begin{proof}
The dual of \cref{E_Z} gives commuting exact sequences
\begin{equation} \label{E_ZDual}
    \begin{tikzcd}[cramped, column sep=tiny]
    0 &
    {R^{2m-1}\pi_*\bZ(m)_{X(\bC)}} \arrow[l] \arrow[d] &
    {R^{2m-1}\pi_{Z,!}\bZ(m)_{(X\backslash |Z|)(\bC)}} \arrow[l] \arrow[d] &
    {R^{2m-2}\pi_* i_{Z,*}\bZ(m)_{|Z|(\bC)}} \arrow[l,"+"] \arrow[d,"{\cdot [Z]_{\sing}}"] &
    {R^{2m-2}\pi_*\bZ(m)_{X(\bC)}} \arrow[l] \arrow[d] \\
    0 &
    \displaystyle\frac{R^{2m-1}\pi_*\bZ(m)_{X(\bC)}}{\mathrm{tors}} \arrow[l] &
    {E_{Z}^\vee} \arrow[l] &
    {\bZ(1)_{S(\bC)}} \arrow[l] &
    0, \arrow[l]
    \end{tikzcd}
\end{equation}
where the morphism labeled by \(+\) does follow the usual convention of the snake lemma as we now explain.
This sign change as well as the multiplication by \([Z]_{\sing}\) rather than \(-[Z]_{\sing}\) happens because
\begin{align*}    
    &R^{2n-1}\pi_{Z, *}\bZ \times R^{2m-2}\pi_*i_{Z, *}i_Z^*\bZ (m) \to R^{2n-1}\pi_{Z, *}\bZ (n)\times R^{2m-1}\pi_{Z, !}\bZ (m) \\
    \to &R^{2d}\pi_{Z, !} \bZ(d + 1) \to \bZ(1)_{S(\bC)}
\end{align*}
is \(-1\) times
\begin{align*}
    &R^{2n-1}\pi_{Z, *}\bZ \times R^{2m-2}\pi_*i_{Z, *}i_Z^*\bZ (m) \to R^{2n}\pi_*i_{Z, *}Ri_Z^!\bZ (n) \times R^{2m-2}\pi_*i_{Z, *}\bZ (m) \\
    \to &R^{2d}\pi_*i_{Z, *}Ri_Z^!\bZ(d + 1) \to \bZ(1)_{S(\bC)}.
\end{align*}
Note the morphism
\[
    \begin{CD}
        \pi_* I_\cO(m)_{X/S}[2m] @>{\cref{C_ZF}}>> \pi_*I_\cO(d + 1)_{X/S}[2d + 2] @>>> C_\cO @>{+1}>> \\
        @|                                      @AAA                                @AAA \\
        \pi_*I_\cO(m)_{X/S}[2m] @>>> \pi_*i_{Z, *}i_Z^*I_\cO(m)_{X/S}[2m] @>{+}>> \pi_*j_{Z, !}j_Z^*I_\cO(m)_{X/S}[2m + 1] @>{+1}>>
    \end{CD}
\]
of triangles in the derived category, with the following explanations.
\begin{itemize}
    \item The morphism labeled \(+\) induces the morphism in the usual proof of the snake lemma.
    \item The rightmost vertical morphism is induced by the inclusion of the source to
    \[
        \pi_*I_\cO(m)_{X/S}[2m+1]
    \]
    and \(0\colon \pi_*j_{Z, !}j_Z^*I_\cO(m)_{X/S}[2m + 1] \to \pi_*I_\cO(d + 1)_{X/S}[2d + 2]\).
\end{itemize}
By \cref{E_ZDual}, the rightmost vertical morphism in the above diagram induces
\begin{align} \label{H-2C_O}
E_{Z, \cO}^\vee\xrightarrow{\simeq}H^{-2}(C_\cO).
\end{align}
We can argue similarly about the integral part of \(H^{-2}(C_{\bZ, F})\).
For the other part taking care of Hodge filtration, we need more argument since a priori these parts of the cohomology of cones do not come from a variation of mixed Hodge structures.
The part of \(H^{-2}(C_{\bZ, F})\) that we need is \(F^mR^{2m-1}\pi_*\pi^{-1}\cO_{S(\bC)}\) by definition.
We can compare this with the exact sequence obtained as the Hodge filtration
\[
    0\leftarrow F^mR^{2m-1}\pi_*\pi^{-1}\cO_{S(\bC)} \leftarrow F^0E_{Z, \cO}^\vee \leftarrow 0 \leftarrow 0
\]
of the second row of \cref{E_ZDual}.
Combining these considerations including \cref{H-2C_O}, we win by \cite[\S C.3]{GZFam}.
\end{proof}

Our \(E_Z^\vee\) corresponds to \(E_{-Z}\) via \cref{Dual}.
Consider commuting exact sequences
\begin{equation} \label{B_W^d(Z)}
    \begin{tikzcd}[cramped, column sep=small]
        0 \arrow[r] &
        R^{2m-1}\pi_{Z, !}\bZ(m) \arrow[r] \arrow[d] &
        R^{2m-1}\pi_{Z, !}Rj_*\bZ(m) \arrow[r, "{-1}"] \arrow[d] &
        R^{2m}\pi_* i_{W, *}Ri_W^!\bZ(m) \arrow[r] \arrow[d, equal] &
        R^{2m}\pi_{Z, !} \bZ(m) \arrow[d, equal] \\
        0 \arrow[r] &
        E_{Z}^\vee \arrow[r] \arrow[d, equal] &
        B_Z^\mathrm{d}(W) \arrow[r] &
        R^{2m}\pi_* i_{W, *}Ri_W^!\bZ(m) \arrow[r] &
        R^{2m}\pi_{Z, !} \bZ(m) \\
        0 \arrow[r] &
        E_{Z}^\vee \arrow[r] &
        B_{W, Z}^\circ \arrow[r] \arrow[u] &
        \bZ_{S(\bC)} \arrow[r] \arrow[u, "{[W]_{\sing}}"] &
        0, \arrow[u]
    \end{tikzcd}
\end{equation}
where we need the explanations below.
\begin{itemize}
    \item We introduce \(B_Z^\mathrm{d}(W)\), which may be related with the dual of \(B_Z(W)\).
        However, we do not discuss this unnecessary point.
    \item We recall that \(j\colon X\backslash (|W|\cup |Z|)\to X\backslash |Z|\) has been defined.
\end{itemize}

To copy the proof of \cite[Theorem 7.11]{EVDB}, we need another claim for
\[
    H^{-2}(C_{\cO, *}), \quad H^{-2}(C_{\bZ, F, *}).
\]
Let \(C_{\bZ, *}\) be the cone of
\[
    \pi_{W, *}j_W^{-1}I_\bZ(m)_X[2m]\to \pi_*I_\bZ(d + 1)_{X}[2d + 2].
\]
We have \(H^{-2}(C_{\bZ, F, *})\simeq H^{-2}(C_{\bZ, *})\oplus R^{2m - 1}\pi_{W, *}\sigma_{\geqq m}\Omega_{(X\backslash |W|)(\bC)/S(\bC)}\), where the last term is the cohomology of the term in \cref{C_ZFj_W} not in the definition of \(C_{\bZ, *}\).

\begin{clm} \label{clm:H-2C_*}
    \begin{enumerate}
        \item\label{H-2C_O*} There is an isomorphism \(H^{-2}(C_{\cO, *})\xleftarrow{\simeq} B_Z^\mathrm{d}(W)_\cO\), and
        \item\label{H-2C_Z*} There is a surjection \(H^{-2}(C_{\bZ, *})\twoheadrightarrow B_Z^\mathrm{d}(W)\) with torsion kernel.
        \item\label{H-2C_ZF*remain} There is a morphism \(R^{2m - 1}\pi_{W, *}\sigma_{\geqq m}\Omega_{(X\backslash |W|)(\bC)/S(\bC)}\leftarrow F^mB_Z^\mathrm{d}(W)_{\cO}\) such that through this and the part
                            \begin{align} \label{EVDB7.11_VertHodgeLift}
                                R^{2m - 1}\pi_{W, *}\sigma_{\geqq m}\Omega_{(X\backslash |W|)(\bC)/S(\bC)}\to R^{2m}\pi_*i_{W, *}Ri_W^!\sigma_{\geqq m}\Omega_{X(\bC)/S(\bC)}
                            \end{align}
                            of a vertical morphism in \cref{EVDB7.11}, the de Rham cycle class of \(W\) seen as the section of the target of \cref{EVDB7.11_VertHodgeLift} lifts to a section of \(F^mB_Z^\mathrm{d}(W)_{s, \bC}\) at each \(s\in S(\bC)\).
    \end{enumerate}
    Morphisms in \cref{H-2C_O*,H-2C_Z*,H-2C_ZF*remain} are compatible with each other.
    The isomorphism in \ref{H-2C_O*} is compatible with \cref{H-2C_O}, \(H^{-2}(C_\cO)\to H^{-2}(C_{\cO, *})\) in \cref{EVDB7.11} and \(E_{Z}^\vee\to B_Z^\mathrm{d}(W)\) in \cref{B_W^d(Z)}.
    The morphisms in \cref{H-2C_Z*,H-2C_ZF*remain} is compatible with
    \[
        H^{-2}(C_{\bZ, F, *})\to R^{2m}\pi_*i_{W, *}Ri_W^!\bZ(m)_{X(\bC)}\oplus F^m R^{2m}\pi_*i_{W, *}Ri_W^!\pi^{-1}\cO_{S(\bC)}
    \]
    in \cref{EVDB7.11}, and \(B_Z^\mathrm{d}(W)\to R^{2m}\pi_*i_{W, *}Ri_W^!\bZ(m)_{X(\bC)}\) in \cref{B_W^d(Z)}.
\end{clm}

\begin{proof}
    Define \(j'\colon X\backslash (|W|\cup |Z|)\to X\backslash |W|\).
    We have the isomorphism \(j_{Z, !}Rj_*j^*j_Z^*\simeq Rj_{W, *}j'_!j'^*j_W^*\) by using the snake lemma to
    \[
        \begin{CD}
            0 @>>> i_{W, *}i_W^! @>>> \id @>>> j_{W, *}j_W^* @. \\
            @.      @VVV            @|         @VVV         @. \\
            0 @>>> j_{Z, !}j_Z^* @>>> \id @>>> i_{Z, *}i_Z^* @>>> 0,
        \end{CD}
    \]
    where the top rightmost horizontal arrow is surjective when applied to injective sheaves.
    The right side of the isomorphism fits in commuting short exact sequences
    \[
        \begin{CD}
            0 @>>> j_{Z, !}j_Z^*        @>>> \id @>>> i_{Z, *}i_Z^* @>>> 0 \\
            @.      @VVV                    @VVV                @|         @. \\
            0 @>>> j_{W, *}j'_!j'^*j_W^* @>>> j_{W, *}j_W^* @>>> i_{Z, *}i_Z^* @>>> 0.
        \end{CD}
    \]
    These considerations and \cref{E_ZDual} describe \(B_Z^\mathrm{d}(W)\) as a quotient of \(B_Z^\mathrm{d}(W)'\) by its torsion subsheaf, where the new notation is as in the pushout
    \begin{equation} \label{B_Z^d(W)'}
        {\fontsize{7truept}{7truept}\selectfont
        \begin{tikzcd}[cramped, column sep=tiny]
        {R^{2m-2}\pi_{W,*}\bZ(m)_{X\backslash |W|}}
            \arrow[r]
            \arrow[d]
        &
        {R^{2m-2}\pi_* i_{Z,*}\bZ(m)_{|Z|}}
            \arrow[r, "{+}"]
            \arrow[d, "{\cdot [Z]_{\sing}}"]
        &
        {R^{2m-1}\pi_{W,*} j'_!\bZ(m)_{X\backslash (|W|\cup |Z|)}}
            \arrow[r]
            \arrow[d]
        &
        {R^{2m-1}\pi_{W,*}\bZ(m)_{X\backslash |W|}}
            \arrow[r]
            \arrow[d, equal]
        & 0 \\
        0
            \arrow[r]
        &
        {\bZ_{S}}
            \arrow[r]
        &
        {B_Z^\mathrm{d}(W)'}
            \arrow[r]
        &
        {R^{2m-1}\pi_{W,*}\bZ(m)_{X\backslash |W|}}
            \arrow[r]
        & 0
        \end{tikzcd}}
    \end{equation}
    of exact sequences, with the comments below.
    \begin{itemize}
        \item We omitted the notation of taking the set of \(\bC\)-points.
        \item The morphism labeled \(+\) is the one that appears in the usual proof of the snake lemma.
        \item The leftmost square commutes since \(R^{2m-2}\pi_{W, *}\bZ(m)_{(X\backslash |W|)(\bC)}\simeq R^{2m-2}\pi_*\bZ(m)_{X(\bC)}\) by purity of the Betti cohomology, with discussions around \cref{rem:PL?} for similar details.
    \end{itemize}
    Arguing similarly to \cref{clm:Jac(EZDual)} gives \cref{H-2C_O*,H-2C_Z*}.

    We need more arguments for \ref{H-2C_ZF*remain}.
    As in \cref{clm:Jac(EZDual)}, we take the short exact sequence that we get as the Hodge filtration
    \[
        0\to 0\to F^0B_Z^\mathrm{d}(W)_\cO \to F^mR^{2m - 1}\pi_{W, *}\pi_{W}^{-1}\cO_{S(\bC)} \to 0
    \]
    of the second row of \cref{B_Z^d(W)'}.
    The \(m\)-th Hodge filtration above does not equal
    \[
        R^{2m - 1}\pi_{W, *}\sigma_{\geqq m}\Omega_{(X\backslash |W|)(\bC)/S(\bC)},
    \]
    the part of the cohomology of \(C_{\bZ, F, *}\) that we intend to discuss.
    However, \cite[Proposition 6.10]{EVDB} says that, at each \(s\in S(\bC)\), the de Rham cycle class of \(W\) in \(H_{|W_s(\bC)|}^{2m}(X_s(\bC), \sigma_{\geqq m}\Omega_{X_s(\bC)})\) comes from \(F^mH^{2m - 1}((X\backslash |W|)_s(\bC), \bC)\).
\end{proof}

By arguing as in the proof of \cite[Theorem 7.11]{EVDB} with the help of \cite[\S C.3]{GZFam}, the section \(B^\circ_{W, Z}\) maps to \(\langle W, Z\rangle_\cD\) via \cref{EVDB7.11,clm:Jac(EZDual),clm:H-2C_*}.
It suffices to prove that
\[
    \begin{CD}
        H^{-2}(C_\cO)/H^{-2}(C_{\bZ, F}) @>{\cref{EVDB7.11}}>> H^{-1}(C_0) \\
        @A{\textup{\cref{clm:Jac(EZDual)}}}AA        @AAA \\
        Ext^1_{\pVMHS}(\bZ, E_{Z}^\vee) @. H^{2d + 1}(\pi_*I_0(d + 1)_{X/S}) \\
        @AAA                                @| \\
        Ext^1_{\pVMHS}(\bZ, \bZ(1)) @>{\textup{\cite[\S C.3]{GZFam}, \cite[(2.2)]{HainBiext}}}>> \cO_{S(\bC)}^\times
    \end{CD}
\]
anticommutes, where the right bottom vertical morphism is by seeing its source and target as a quotient of \(R^{2d}\pi_*\pi^{-1}\cO_{S(\bC)} = \cO_{S(\bC)}\).
This sign inconsistency comes from \(-1\) in the definition of \cref{EVDB7.11}.
\end{proof}

\begin{thm} \label{thm:ArchCompSm}
There exists an isomorphism of the biextension by \(\bG_{\mathrm{m}, S}^{\an}\) induced by \cref{clm:descent} and \(\cB(R^{2m-1}\pi_*\bZ_{X(\bC)}/\mathrm{tors})\) twisted by \(-1\colon \bG_{\mathrm{m}, S}^\an\to \bG_{\mathrm{m}, S}^\an\) that carries \(\langle W, Z\rangle\) to the image of \(B_{W, Z}^\circ\) for an open subset \(U\subseteq S\) and disjoint cycles \(W\in z^m_{\hom, \equi}(U_{X/S}, 0)\) and \(Z\in z^n_{\hom, \equi}(U_{X/S}, 0)\).
\end{thm}

\begin{proof}
    By the comparison of \(\langle W, Z\rangle\) and \(\langle W, Z\rangle_\cD\) as well as \cref{prop:<WZ>DVSBWZ}, it follows that for an open subset \(U\subseteq S\), a cycle \(Z\in z^n_{\hom, \equi}(U_{X/S}, 0)\) and \(W, W'\in z^m_{\hom, \equi}(U_{X/S}, 0)\) disjoint from \(Z\), we have
    \[
        B_{W, Z}^\circ = \sigma_Z(W - W')^{-1} B_{W', Z}^\circ.
    \]
    We have a similar formula for the second variable of \(B_{W, Z}^\circ\) by \cite[(3.3.4)]{HainBiext}.
    Namely,
    \[
        B_{W, Z}^\circ = \sigma_W(Z - Z')^{-1} B_{W, Z'}^\circ
    \]
    for an open subset \(U\subseteq S\), a cycle \(W\in z^m_{\hom, \equi}(U_{X/S}, 0)\) and \(Z, Z'\in z^n_{\hom, \equi}(U_{X/S}, 0)\) disjoint from \(W\).
    These and the fact that \(\langle W, Z \rangle\) and \(B_{W, Z}^\circ\) are bilinear with respect to \(W, Z\) show that there is a unique map \(\langle W, Z\rangle \mapsto B_{W, Z}^\circ\) of biextensions as in the statement since \(\langle W, Z\rangle\) generate the biextension in \cref{clm:descent}.
\end{proof}

\begin{rem}
    The proof here is a version of that of \cite[Proposition 3.1]{deJongGrossSchoen}.
    However, we pay attention to the well-definedness of \(\langle W, Z\rangle\mapsto B_{W, Z}^\circ\), which is omitted in \emph{loc. cit}.
    We do not formulate similar theorems involving \cref{biextD} because we do not have made \(\langle W, Z\rangle_\cD\) symmetric.
\end{rem}

\section{The case of semistable reductions over curves} \label{sec:ArchComp}
We now relate the \(\bQ\)-line bundles of \cite[Remark 242]{BPJump} with those out of \cite{BloBiext,SeiBier}.
We work in the case of strong semistable reduction over a curve.
Namely, in this section, we suppose that \(S\) is a smooth irreducible curve over \(\bC\).
We also suppose that the fibers of \(\pi\colon X\to S\) are smooth or strong normal crossing divisors.
In this case, the monodromy of \(R^{2m-1}\pi_*\bZ(m)/\mathrm{tors}\) around \(s\in S\) with singular fiber is unipotent \cite[Corollary 11.19]{PSMHS}.

To prepare for our comparison, we describe how to modify generically homologically trivial cycles to those that are also homologically trivial in the singular fibers in some sense.
\begin{lem} \label{lem:Hodge}
Take \(W\in z^m_{\hom}(S_{X/S}, 0)\).
Assume the Hodge conjecture for each component of the singular fibers of \(\pi\).
Then there exists \(\hat{W}\in z^m(S_{X/S}, 0) \otimes_\bZ \bQ\) whose pullback to \(\pi^{-1}(S_\sm)\) coincides with that of \(W\), with the following property.
Take \(s\in S\backslash S_\sm\).
Let \(i\colon X_s\to X\) be the closed immersion.
Then
\begin{align} \label{bdryHomTriv}
i^*([\hat{W}]_\sing) = 0\in H^{2m}(X_s(\bC), \bQ(m)).
\end{align}
\end{lem}

\begin{proof}
Take \(s\in S\backslash S_\sm\).
By considering the weight of the morphism \(\bQ\to H^{2m}(X_s(\bC), \bQ(m))\) of mixed Hodge structures induced by \(i^*([\hat{W}]_\sing)\), it suffices to show that it vanishes in
\[
    \Gr_{\mathrm{W}}^{0}H^{2m}(X_s(\bC), \bQ(m)),
\]
where \(\mathrm{W}\) is the weight filtration of the mixed Hodge structure.

Fix a local parameter of \(S(\bC)\) at \(s\).
For an integer \(p\), let \(H^{2m}(X_\infty, \bQ(p))\) be the corresponding limit mixed Hodge structure with Tate twist.
Let \(i_s\colon X_s\to X\) be the closed immersion.
The image \(i_s^*[W]_\sing\in H^{2m}(X_s(\bC), \bQ(m))\) vanishes in the last term of the Clemens--Schmid exact sequence
\[
    \dots \xrightarrow{N} H^{2m - 2}(X_\infty, \bQ(m - 1)) \to H^{2m}_{X_s(\bC)}(X(\bC), \bQ(m)) \xrightarrow{L} H^{2m}(X_s(\bC), \bQ(m)) \to H^{2m}(X_\infty, \bQ(m)),
\]
since \(W\) is generically homologically trivial.
Hence, \(i_s^*[W]_\sing\) comes from \(H^{2m}_{X_s(\bC)}(X(\bC), \bQ(m))\).

More precisely, the short exact sequence
\[
    0\to \Ker \Gr_\mathrm{W}^{0}L\to \Gr_{\mathrm{W}}^{0}H^{2m}_{X_s(\bC)}(X(\bC), \bQ(m)) \to \Ima \Gr_\mathrm{W}^{0}L \to 0
\]
splits since these are polarizable pure Hodge structures.
Therefore, the Hodge class \(i_s^*[W]_\sing\) in \(\Ima \Gr_\mathrm{W}^{0}L\) lifts to that in \(\Gr_{\mathrm{W}}^{0}H^{2m}_{X_s(\bC)}(X(\bC), \bQ(m))\).
The lift can again be lifted to a Hodge class in
\[
\bigoplus_{Y_0\in \pi_0(Y)} H^{2m-1}_{Y_0(\bC)}(X(\bC), \bQ(m)).
\]
We win by using the Hodge conjecture there.
\end{proof}

\begin{thm} \label{thm:ArchComp}
Suppose that \(S\) is a smooth irreducible curve over \(\bC\).
Also suppose that the fibers of \(\pi\colon X\to S\) are smooth or strong normal crossing divisors.
Take \(W\in z^m_{\hom}(S_{X/S}, 0), Z\in z^n_{\hom}(S_{X/S}, 0)\).
Let
\begin{align*}
    &\nu\in Ext^1_{\AVMHS}(\bZ_{S_\sm (\bC)}, R^{2m-1}\pi_*\bZ(m)/\mathrm{tors})(S_\sm(\bC)), \\
    &\omega\in Ext^1_{\AVMHS}(R^{2m-1}\pi_*\bZ(m)/\mathrm{tors}, \bZ(1)_{S_\sm (\bC)})(S_\sm (\bC))
\end{align*}
be their images.
Suppose that there exists \(\hat{W}\in z^m(S_{X/S}, 0)\otimes_{\bZ} \bQ\) that satisfies \cref{bdryHomTriv}.
Then we have an isomorphism \(\bE_{\hat{W}, Z} \cong \ol{\cB}(R^{2m-1}\pi_*\bZ(m)/\mathrm{tors})_{\omega, \nu}\) of \(\bQ\)-line bundles which extend \cref{thm:ArchCompSm} over \(S_\sm\).
\end{thm}

\begin{proof}
Take a positive integer \(N\) such that \(N\hat{W}\) has integral coefficients.
Take \(s\in S\backslash S_\sm\).
By moving lemma, over some open neighborhood \(U\) of \(s\) in \(S\), there exists \(W'\in z^m_{\hom, \equi}(U_{X/S}, 0), Z'\in z^n_{\hom, \equi}(U_{X/S}, 0)\) rationally equivalent to \(N\hat{W}, Z\), respectively, such that \(W'\cap Z' = \emptyset\).
It suffices to show that \(\ol{\cB}(R^{2m-1}\pi_*\bZ(m)/\mathrm{tors})_{\omega, N\nu}\) comes from a line bundle and that the section \(B^\circ_{W', Z'}\in \cB(R^{2m-1}\pi_*\bZ(m)/\mathrm{tors})_{\omega, N\nu}(U\backslash \{s\})\) defines a section of \(\ol{\cB}(R^{2m-1}\pi_*\bZ(m)/\mathrm{tors})_{\omega, N\nu}\) since then \(\langle W', Z'\rangle \mapsto B^\circ_{W', Z'}\) defines the desired isomorphism locally around \(s\).
For all these, it is enough to show that \(\wt\tau_s (B^\circ_{W', Z'}) = 0\).

Recall the distinguished triangle
\[
    R\pi_{Z', !}\bQ(m)\to R\pi_*\bQ(m)\xrightarrow{-} R\pi_{*}i_{Z', *}\bQ(m)\xrightarrow{-} R\pi_{Z', !}\bQ(m)[1]
\]
that gives a variant of \cref{E_ZDual}, where \(-\) means the negative of the conventions in \cite{HartsRes}.
This induces an exact sequence
\[
    0\to \bQ(1)_{U(\bC)\backslash \{s\}}\to E_{Z', \bQ}\to R^{2m-1}\pi_*\bQ(m) \to 0
\]
of variations of \(\bQ\)-mixed Hodge structures on \(U(\bC)\backslash \{s\}\).
By \cite[Lemma 2.18]{BFNPANF}, this induces an exact sequence
\[
    0\to \bQ(1)_{U(\bC)}[1]\to j_{!*}(E_{Z', \bQ}[1])\to j_{!*}(R^{2m-1}\pi_*\bQ(m)[1]) \to 0
\]
of \(\bQ\)-mixed Hodge modules on \(U\), where \(j\colon U\backslash \{s\}\to U\) is the open immersion.

Also recall the distinguished triangles
\begin{align} \label{PervTri}
\begin{CD}
    R\pi_*i_{W', *}Ri_{W'}^!\bQ(m)@>{-}>> R\pi_{Z', !}\bQ(m)@>>> R\pi_{W', *}j'_!\bQ(m)@>{-}>> R\pi_*i_{W', *}Ri_{W'}^!\bQ(m)[1] \\
    @|                                      @VVV                     @VVV                     @| \\
    R\pi_*i_{W', *}Ri_{W'}^!\bQ(m)@>{-}>> R\pi_*\bQ(m)      @>>> R\pi_{W', *}\bQ(m)    @>{-}>> R\pi_*i_{W', *}Ri_{W'}^!\bQ(m)[1],
\end{CD}
\end{align}
where \(j'\colon X\backslash (|W'|\cup |Z'|)\to X\backslash |W'|\) is the open immersion.
Taking the perverse cohomology of the first row induces commuting exact sequences
\begin{equation} \label{PervBlendExt}
{\fontsize{9truept}{9truept}\selectfont
\begin{tikzcd}[cramped, column sep = tiny]
{}^p\!R^{2m}\pi_*i_{*}Ri^!\bQ(m) \arrow[r] \arrow[d] &
{}^p\! R^{2m}\pi_{Z', !}\bQ(m) \arrow[r] \arrow[d] &
{}^p\! R^{2m}\pi_{W', *}j'_!\bQ(m) \arrow[r] &
{}^p\! R^{2m + 1}\pi_* i_{*}Ri^!\bQ(m) \arrow[r] &
{}^p\! R^{2m + 1}\pi_{Z', !} \bQ(m) \\
0 \arrow[r] &
\displaystyle\frac{{}^p\! R^{2m}\pi_{Z', !}\bQ(m)}{{}^p\!R^{2m}\pi_*i_{*}Ri^!\bQ(m)} \arrow[r] &
{}^p\! B_{W'}(Z') \arrow[r] \arrow[u] &
\bQ[1] \arrow[r] \arrow[u] &
0 \arrow[u]
\end{tikzcd}
}
\end{equation}
of mixed Hodge modules, where \(i\coloneqq i_{W'}\), and where the last row below is defined by the pullback via a variant
\begin{align} \label{PervCyc}
    \bQ[1]\to {}^p\!R^1\pi_*\bQ\xrightarrow{-[W']_\sing} {}^p\!R^{2m + 1}\pi_*i_{W', *}Ri_{W'}^!\bQ(m)
\end{align}
of the cycle class, which has the following property.
\begin{clm}
    The composition of \cref{PervCyc} to \({}^p\! R^{2m + 1}\pi_{Z', !} \bQ(m)\) vanishes.
\end{clm}
\begin{proof}
By Ehresmann's lemma, the simple objects in the weight graded pieces of \({}^p\! R^{2m + 1}\pi_{Z', !} \bQ(m)\) are subquotients of \(j_{!*}(R^{2m}\pi_*i_{W', *}Ri_{W'}^!\bQ(m)[1])\) or of the form \(i_*V\), where \(i\colon \{s\}\to U\) and \(V\) is a pure Hodge structure.
We have \(\Hom(\bQ[1], i_*V) = \Hom(\bQ[1], V) = 0\).
Also, the usual cycle class is \(0\) in \(\Hom(\bQ[1], j_{!*}(R^{2m}\pi_*i_{W', *}Ri_{W'}^!\bQ(m)[1])) = \Hom(\bQ, R^{2m}\pi_*i_{W', *}Ri_{W'}^!\bQ(m))\).
Therefore, we have the desired vanishing.
\end{proof}

We start decoding the quotient in \cref{PervBlendExt}.
\begin{clm} \label{clm:NonQuotH^0}
    The morphism \(H^0({}^p\! R\pi_{Z', !}\bQ(m)) \to H^0({}^p\! R\pi_*\bQ(m))\) is isomorphic.
\end{clm}

\begin{proof}
Consider the distinguished triangle
\[
    R\pi_{Z', !}\bQ(m)\to R\pi_*\bQ(m)\to R\pi_*i_{Z', *}\bQ(m)\xrightarrow{+1}.
\]
It induces an exact sequence
\[
    {}^p\! R^{2m - 1}\pi_*i_{Z', *}\bQ(m)\to {}^p\! R^{2m}\pi_{Z', !}\bQ(m)\to {}^p\! R^{2m}\pi_*\bQ(m)\to {}^p\! R^{2m}\pi_*i_{Z', *}\bQ(m) = 0,
\]
where the last equality is by the decomposition theorem and the proper intersection of \(Z'\) with the fibers of \(\pi\).
Since \(H^0\) is right exact in the category of perverse sheaves, we get the surjectivity.
The injectivity is again by the decomposition theorem and the proper intersection.
\end{proof}

\begin{clm} \label{clm:QuotH^0}
    The morphism
    \[
        H^0({}^p\! R^{2m}\pi_{Z', !}\bQ(m)/{}^p\!R^{2m}\pi_*i_{W', *}Ri_{W'}^!\bQ(m))\to H^0({}^p\! R^{2m}\pi_*\bQ(m)/{}^p\!R^{2m}\pi_*i_{W', *}Ri_{W'}^!\bQ(m))
    \]
    is an isomorphism.
\end{clm}

\begin{proof}
    We define \(C, C'\) as cokernels in the commuting exact sequences
    \[
    \begin{CD}
        {}^p\! R^{2m}\pi_*i_{W', *}Ri_{W'}^!\bQ(m) @>>> {}^p\! R^{2m}\pi_{Z', !}\bQ(m) @>>> C @>>> 0\\
        @|                                              @VVV                      @VVV  @. \\
        {}^p\! R^{2m}\pi_*i_{W', *}Ri_{W'}^!\bQ(m) @>>> {}^p\! R^{2m}\pi_*\bQ(m) @>>> C' @>>> 0
    \end{CD}
    \]
    induced by \cref{PervTri}.
    We can apply \(H^0\) without losing the exactness since \(H^0\) is right exact in the category of perverse sheaves.
    We win by \cref{clm:NonQuotH^0} and the five lemma.
\end{proof}

By \cref{clm:QuotH^0}, the pushforward of the bottom row of \cref{PervBlendExt} via
\[
    {}^p\! R^{2m}\pi_{Z', !}\bQ(m)/{}^p\!R^{2m}\pi_*i_{W', *}Ri_{W'}^!\bQ(m) \to H^0({}^p\! R^{2m}\pi_{Z', !}\bQ(m)/{}^p\!R^{2m}\pi_*i_{W', *}Ri_{W'}^!\bQ(m))
\]
is the same as the similar construction applied to the bottom row of \cref{PervTri}.
By \cref{ex:HomTrivExSeq}, the extension is given by
\begin{align} \label{PervExt}
    \bQ[1]\to {}^p\!R^{2m}\pi_*\bQ(m)[1] \to H^0({}^p\!R^{2m}\pi_*\bQ(m))[1] \to H^0({}^p\!R^{2m}\pi_*\bQ(m)/{}^p\!R^{2m}\pi_*i_{W', *}Ri_{W'}^!\bQ(m))[1],
\end{align}
where the first morphism is induced by
\[
    \bQ[1]\to R\pi_*\bQ[1]\to R\pi_*i_{W', *}Ri_{W'}^!\bQ(m)[2m + 1]\to R\pi_*\bQ(m)[2m + 1],
\]
part of which gives \cref{PervCyc}.
Upon identifying \(H^0({}^p\!\tau_{\leqq 2m}R\pi_*\bQ(m)) \xrightarrow{\simeq} H^0({}^p\!R^{2m}\pi_*\bQ(m))\), the first two morphisms of \cref{PervExt} are those in
\[
    \bQ[1]\to {}^p\!\tau_{\leqq 2m}R\pi_*\bQ(m)[1] \to H^0({}^p\!\tau_{\leqq 2m}R\pi_*\bQ(m))[1]\hookrightarrow R^{2m}\pi_*\bQ(m)[1],
\]
where the injectivity of the last morphism is by the decomposition theorem.
The entire composition is \(0\) due to the definition of \(W'\) as well as because singular fibers of \(\pi\), for example, \(X_s(\bC)\), are \cite{ClemDegen} homotopic to the inverse image of a small disc around \(s\) by \(\pi\).
Therefore, \cref{PervExt} also vanishes.

The last paragraph shows that the bottom row of \cref{PervBlendExt} comes from an extension
\[
    0 \to H^{-1}({}^p\! R^{2m}\pi_{Z', !}\bQ(m)/{}^p\!R^{2m}\pi_*i_{W', *}Ri_{W'}^!\bQ(m))[1]\to B^{-1}_{W'}(Z') \to \bQ[1] \to 0.
\]
We can push this out by
\[
    H^{-1}({}^p\! R^{2m}\pi_{Z', !}\bQ(m)/{}^p\!R^{2m}\pi_*i_{W', *}Ri_{W'}^!\bQ(m))[1]\to j_{!*}(E_{Z', \bQ}[1])
\]
to get an extension
\begin{align} \label{PervB_W(Z)}
    0\to j_{!*}(E_{Z', \bQ}[1]) \to \sB_{W'}(Z') \to \bQ[1] \to 0.
\end{align}

We have \(\sB_{W'}(Z') = j_{!*}(B_{W'}(Z')[1])\) since the left side has a filtration consisting of intermediate extensions.
Let \(V\) be the rational limit mixed Hodge structure of \(B_{W'}(Z')\) at \(s\) for some local parameter.
Let \(N_V\) be its monodromy logarithm. 
As in \cite[Remark 20]{BPJump}, there exist \(v_0, v_{-2}\in V\) and a subspace \(U\subset V\) such that
\begin{itemize}
    \item \(V = \bQ v_0\oplus U\oplus \bQ v_{-2}\),
    \item \(N_V(U) \subseteq U\),
    \item \(N_V(v_0) \in \bQ v_{-2}\), and
    \item \(v_0\mapsto 1\), \(U, \bQ v_{-2}\to 0\) by the surjection in \cref{PervB_W(Z)},
\end{itemize}
and it suffices to show that \(N_V(v_0) = 0\).
The stalk of the surjection in \cref{PervB_W(Z)} is a surjection \(\Ker N_V\mapsto \bQ\).
This happens only if \(N_V(v_0) = 0\).
\end{proof}

As an application of Theorem \cref{thm:ArchComp}, we study the asymptotic behavior of the height of \(B^\circ_{W_{t}, Z_t}\) when \(t\in S\) is around a point of \(S\) with the singular fiber.
For the meaning of the height, we refer to \cite[(3.3.7)]{HainBiext}.

\begin{thm} \label{thm:Apply}
Suppose that \(S\) is a smooth irreducible curve over \(\bC\).
Also suppose that the fibers of \(\pi\colon X\to S\) are smooth or strong normal crossing divisors.
Take \(W\in z^m_{\hom}(S_{X/S}, 0), Z\in z^n_{\hom}(S_{X/S}, 0)\) that properly intersect.
Suppose that there exists \(\hat{W}\in z^m(S_{X/S}, 0)\otimes_{\bZ} \bQ\) that satisfies \cref{bdryHomTriv}.
Take a closed point \(s\in S\).
Let \(p\colon \{z\in \bC\mid |z|<1\}\to S(\bC)\) be a local coordinate in a sufficiently small region around \(s\).
For \(0\neq t\) in the source of \(p\), let \(h(t)\) be the height of \(B^{\circ}_{W_{p(t)}, Z_{p(t)}}\) in terms of \cite[Definition D.1]{GZFam}.
Then
\[
    h(t) - \langle W, Z\rangle_{\mathrm{a}, s}\log |t|
\]
is bounded and continuous around \(t = 0\), where \(\langle W, Z\rangle_{\mathrm{a}, s}\) is the local intersection index defined in \cite[\S 2]{BeiHeiCyc}.
\end{thm}

\begin{proof}
In the proof of \cref{thm:ArchComp}, we have seen that \(\wt\tau_s (B^\circ_{W', Z'}) = 0\).
This shows \cite[Property 18]{BPJump} that the height, namely the logarithm of the metric as in \cite[Definition D.1]{GZFam}, extends continuously around \(t = 0\).
By \cite[Lemma 6.1]{HdJAsym}, we have only to investigate the multiplicity of \(\langle \hat{W}, Z\rangle\) around \(s\), where \(N, W'\) are defined as in the proof of \cref{thm:ArchComp} and we consider suitable restrictions of \(\hat{W}, Z\) to a Zariski open subset of \(S\backslash \{s\}\).

We may assume that \(\hat{W}, Z\) properly intersect by moving lemma.
We can apply \cite[Lemma-Definition 2.1.1.c)]{BeiHeiCyc}, used in conjunction with a variant of \cite[Remark B.5]{LLMot}, to the following situation in the notation there since \(X_s(\bC)\) is \cite{ClemDegen} homotopic to the inverse image of the image of \(p\) by \(\pi\).
\begin{itemize}
    \item Let \(\cO_{C_v}(C_v)\) be the completion of \(\cO_{S, s}\).
    \item Put \(X_{C_v} \coloneqq X\times_S \Spec \cO_{S, s}\).
    \item Let \(\beta_1, \beta_2\) be suitable restrictions of \([\hat{W}/N]_\et, [Z]_\et\).
\end{itemize}
The outcome is that \(\langle W, Z\rangle_{\mathrm{a}, s}\) coincides with the multiplicity of \(\pi_*(\hat{W}\cdot Z)/N\) at \(s\), where the twist in \cref{thm:ArchCompSm} is absorbed into the anticommutativity of the diagram
\[
    \begin{CD}
        H^0_{\et}(C_v, \Gm) @>>> H^1_{\{s\}}(C_v, \Gm) \\
        @VVV           @VVV \\
        H^1_{\et}(X_{C_v}, \mu_{\ell^a}) @>>> H^2_{\{s\}}(C_v, \mu_{\ell^a})
    \end{CD}
\]
of etale cohomology for a prime number \(\ell\) and a nonnegative integer \(a\).
The multiplicity is exactly that of \(\langle \hat{W}, Z\rangle\) by \cref{prop:FullCalc}.
\end{proof}


\begin{thebibliography}{GAGP88}

\bibitem[BFNP09]{BFNPANF}
P.~Brosnan, H.~Fang, Z.~Nie, and G.~Pearlstein.
\newblock Singularities of admissible normal functions.
\newblock {\em Invent. Math.}, 177(3):599--629, 2009.
\newblock With an appendix by Najmuddin Fakhruddin.
\newblock \href {https://doi.org/10.1007/s00222-009-0191-9} {\path{doi:10.1007/s00222-009-0191-9}}.

\bibitem[Bi86]{BeiAbs}
A.~A. Be\u~ilinson.
\newblock Notes on absolute {H}odge cohomology.
\newblock In {\em Applications of algebraic {$K$}-theory to algebraic geometry and number theory, {P}art {I}, {II} ({B}oulder, {C}olo., 1983)}, volume~55 of {\em Contemp. Math.}, pages 35--68. Amer. Math. Soc., Providence, RI, 1986.
\newblock \href {https://doi.org/10.1090/conm/055.1/862628} {\path{doi:10.1090/conm/055.1/862628}}.

\bibitem[Blo]{BloCub}
S.~Bloch.
\newblock Some notes on elementary properties of higher chow groups, including functoriality properties and cubical chow groups.
\newblock URL: \url{https://math.uchicago.edu/~bloch/cubical_chow.pdf}.

\bibitem[Blo84]{BloHei}
S.~Bloch.
\newblock Height pairings for algebraic cycles.
\newblock In {\em Proceedings of the {L}uminy conference on algebraic {$K$}-theory ({L}uminy, 1983)}, volume~34, pages 119--145, 1984.
\newblock \href {https://doi.org/10.1016/0022-4049(84)90032-X} {\path{doi:10.1016/0022-4049(84)90032-X}}.

\bibitem[Blo86]{BloBei}
S.~Bloch.
\newblock Algebraic cycles and the {B}e\u ilinson conjectures.
\newblock In {\em The {L}efschetz centennial conference, {P}art {I} ({M}exico {C}ity, 1984)}, volume~58 of {\em Contemp. Math.}, pages 65--79. Amer. Math. Soc., Providence, RI, 1986.
\newblock \href {https://doi.org/10.1090/conm/058.1/860404} {\path{doi:10.1090/conm/058.1/860404}}.

\bibitem[Blo89]{BloBiext}
S.~Bloch.
\newblock Cycles and biextensions.
\newblock In {\em Algebraic {$K$}-theory and algebraic number theory ({H}onolulu, {HI}, 1987)}, volume~83 of {\em Contemp. Math.}, pages 19--30. Amer. Math. Soc., Providence, RI, 1989.
\newblock \href {https://doi.org/10.1090/conm/083/991974} {\path{doi:10.1090/conm/083/991974}}.

\bibitem[BP19]{BPJump}
P.~Brosnan and G.~Pearlstein.
\newblock Jumps in the {Archimedean} height.
\newblock {\em Duke Mathematical Journal}, 168(10), jul 2019.
\newblock \href {https://doi.org/10.1215/00127094-2018-0056} {\path{doi:10.1215/00127094-2018-0056}}.

\bibitem[Bry02]{BryPL}
J.~L. Bryant.
\newblock Piecewise linear topology.
\newblock In {\em Handbook of geometric topology}, pages 219--259. North-Holland, Amsterdam, 2002.

\bibitem[BS19]{BSModulus}
F.~Binda and S.~Saito.
\newblock Relative cycles with moduli and regulator maps.
\newblock {\em J. Inst. Math. Jussieu}, 18(6):1233--1293, 2019.
\newblock \href {https://doi.org/10.1017/s1474748017000391} {\path{doi:10.1017/s1474748017000391}}.

\bibitem[Car80]{CarExt}
J.~A. Carlson.
\newblock Extensions of mixed {H}odge structures.
\newblock In {\em Journ\'ees de {G}\'eometrie {A}lg\'ebrique d'{A}ngers, {J}uillet 1979/{A}lgebraic {G}eometry, {A}ngers, 1979}, pages 107--127. Sijthoff \& Noordhoff, Alphen aan den Rijn---Germantown, Md., 1980.

\bibitem[Cle77]{ClemDegen}
C.~H. Clemens.
\newblock Degeneration of {K}\"ahler manifolds.
\newblock {\em Duke Math. J.}, 44(2):215--290, 1977.
\newblock URL: \url{http://projecteuclid.org/euclid.dmj/1077312231}.

\bibitem[Con00]{ConBaseChange}
B.~Conrad.
\newblock {\em Grothendieck duality and base change}, volume 1750 of {\em Lecture Notes in Mathematics}.
\newblock Springer-Verlag, Berlin, 2000.
\newblock \href {https://doi.org/10.1007/b75857} {\path{doi:10.1007/b75857}}.

\bibitem[dJ14]{deJongGrossSchoen}
R.~de~Jong.
\newblock Normal functions and the height of {G}ross-{S}choen cycles.
\newblock {\em Nagoya Math. J.}, 214:53--77, 2014.
\newblock \href {https://doi.org/10.1215/00277630-2413391} {\path{doi:10.1215/00277630-2413391}}.

\bibitem[EHN25]{EHNPoin}
K.~Eder, J.~Huerta, and S.~Noja.
\newblock Poincar\'e duality for supermanifolds, higher {C}artan structures and geometric supergravity, 2025.
\newblock \href {https://arxiv.org/abs/2312.05224} {\path{arXiv:2312.05224}}.

\bibitem[EV88]{EVDB}
H.~Esnault and E.~Viehweg.
\newblock Deligne-{B}e\u ilinson cohomology.
\newblock In {\em Be\u ilinson's conjectures on special values of {$L$}-functions}, volume~4 of {\em Perspect. Math.}, pages 43--91. Academic Press, Boston, MA, 1988.

\bibitem[FF14]{FFVMHS>0}
O.~Fujino and T.~Fujisawa.
\newblock Variations of mixed {H}odge structure and semipositivity theorems.
\newblock {\em Publ. Res. Inst. Math. Sci.}, 50(4):589--661, 2014.
\newblock \href {https://doi.org/10.4171/PRIMS/145} {\path{doi:10.4171/PRIMS/145}}.

\bibitem[Ful98]{FulInt}
W.~Fulton.
\newblock {\em Intersection theory}, volume~2 of {\em Ergebnisse der Mathematik und ihrer Grenzgebiete. 3. Folge. A Series of Modern Surveys in Mathematics [Results in Mathematics and Related Areas. 3rd Series. A Series of Modern Surveys in Mathematics]}.
\newblock Springer-Verlag, Berlin, second edition, 1998.
\newblock \href {https://doi.org/10.1007/978-1-4612-1700-8} {\path{doi:10.1007/978-1-4612-1700-8}}.

\bibitem[GAGP88]{GNPPHyp}
F.~Guill\'en, V.~Navarro Aznar, P.~Pascual Gainza, and F.~Puerta.
\newblock {\em Hyperr\'esolutions cubiques et descente cohomologique}, volume 1335 of {\em Lecture Notes in Mathematics}.
\newblock Springer-Verlag, Berlin, 1988.
\newblock Papers from the Seminar on Hodge-Deligne Theory held in Barcelona, 1982.
\newblock \href {https://doi.org/10.1007/BFb0085054} {\path{doi:10.1007/BFb0085054}}.

\bibitem[GL01]{GLBK}
T.~Geisser and M.~Levine.
\newblock The {B}loch-{K}ato conjecture and a theorem of {S}uslin-{V}oevodsky.
\newblock {\em J. Reine Angew. Math.}, 530:55--103, 2001.
\newblock \href {https://doi.org/10.1515/crll.2001.006} {\path{doi:10.1515/crll.2001.006}}.

\bibitem[God58]{Godement}
R.~Godement.
\newblock {\em Topologie alg\'ebrique et th\'eorie des faisceaux}, volume No. 13 of {\em Publications de l'Institut de Math\'ematiques de l'Universit\'e{} de Strasbourg [Publications of the Mathematical Institute of the University of Strasbourg]}.
\newblock Hermann, Paris, 1958.
\newblock Actualit\'es Scientifiques et Industrielles, No. 1252. [Current Scientific and Industrial Topics].

\bibitem[Gon02]{GonReg}
A.~B. Goncharov.
\newblock Explicit regulator maps on polylogarithmic motivic complexes.
\newblock In {\em Motives, polylogarithms and {H}odge theory, {P}art {I} ({I}rvine, {CA}, 1998)}, volume 3, I of {\em Int. Press Lect. Ser.}, pages 245--276. Int. Press, Somerville, MA, 2002.

\bibitem[Gor09]{GorBiext}
S.~Gorchinskiy.
\newblock Notes on the biextension of {C}how groups.
\newblock In {\em Motives and algebraic cycles}, volume~56 of {\em Fields Inst. Commun.}, pages 111--148. Amer. Math. Soc., Providence, RI, 2009.
\newblock \href {https://doi.org/10.1090/fic/056/05} {\path{doi:10.1090/fic/056/05}}.

\bibitem[GZ24]{GZFam}
Z.~Gao and S.-W. Zhang.
\newblock Heights and periods of algebraic cycles in families, 2024.
\newblock \href {https://arxiv.org/abs/2407.01304v3} {\path{arXiv:2407.01304v3}}.

\bibitem[Hai90]{HainBiext}
R.~Hain.
\newblock Biextensions and heights associated to curves of odd genus.
\newblock {\em Duke Math. J.}, 61(3), 1990.
\newblock \href {https://doi.org/10.1215/S0012-7094-90-06133-2} {\path{doi:10.1215/S0012-7094-90-06133-2}}.

\bibitem[Har66]{HartsRes}
R.~Hartshorne.
\newblock {\em Residues and duality}, volume No. 20 of {\em Lecture Notes in Mathematics}.
\newblock Springer-Verlag, Berlin-New York, 1966.
\newblock Lecture notes of a seminar on the work of A. Grothendieck, given at Harvard 1963/64, With an appendix by P. Deligne.

\bibitem[Har05]{HarBiext}
C.~Hardouin.
\newblock {\em Structure galoisienne des extensions it{\'e}r{\'e}es de modules diff{\'e}rentiels}.
\newblock PhD thesis, Paris 6, 2005.

\bibitem[Hat02]{Hatcher}
A.~Hatcher.
\newblock {\em Algebraic topology}.
\newblock Cambridge University Press, Cambridge, 2002.

\bibitem[HdJ15]{HdJAsym}
D.~Holmes and R.~de~Jong.
\newblock Asymptotics of the {N}\'eron height pairing.
\newblock {\em Math. Res. Lett.}, 22(5):1337--1371, 2015.
\newblock \href {https://doi.org/10.4310/MRL.2015.v22.n5.a5} {\path{doi:10.4310/MRL.2015.v22.n5.a5}}.

\bibitem[Hof09]{HofTri}
K.~R. Hofmann.
\newblock {\em Triangulation of locally semi-algebraic spaces}.
\newblock ProQuest LLC, Ann Arbor, MI, 2009.
\newblock Thesis (Ph.D.)--University of Michigan.
\newblock URL: \url{http://gateway.proquest.com/openurl?url_ver=Z39.88-2004&rft_val_fmt=info:ofi/fmt:kev:mtx:dissertation&res_dat=xri:pqdiss&rft_dat=xri:pqdiss:3382214}.

\bibitem[ili87]{BeiHeiCyc}
A.~Be\u ilinson.
\newblock Height pairing between algebraic cycles.
\newblock In {\em {$K$}-theory, arithmetic and geometry ({M}oscow, 1984--1986)}, volume 1289 of {\em Lecture Notes in Math.}, pages 1--25. Springer, Berlin, 1987.
\newblock \href {https://doi.org/10.1007/BFb0078364} {\path{doi:10.1007/BFb0078364}}.

\bibitem[Jan88]{JannDel}
U.~Jannsen.
\newblock Deligne homology, {H}odge-{${\mathscr{D}}$}-conjecture, and motives.
\newblock In {\em Be\u ilinson's conjectures on special values of {$L$}-functions}, volume~4 of {\em Perspect. Math.}, pages 305--372. Academic Press, Boston, MA, 1988.

\bibitem[Kas86]{KasVMHS}
M.~Kashiwara.
\newblock A study of variation of mixed {H}odge structure.
\newblock {\em Publ. Res. Inst. Math. Sci.}, 22(5):991--1024, 1986.
\newblock \href {https://doi.org/10.2977/prims/1195177264} {\path{doi:10.2977/prims/1195177264}}.

\bibitem[KL08]{KLAddHighChow}
A.~Krishna and M.~Levine.
\newblock Additive higher {C}how groups of schemes.
\newblock {\em J. Reine Angew. Math.}, 619:75--140, 2008.
\newblock \href {https://doi.org/10.1515/CRELLE.2008.041} {\path{doi:10.1515/CRELLE.2008.041}}.

\bibitem[KLMS06]{KLM}
M.~Kerr, J.~Lewis, and S.~Müller-Stach.
\newblock The {A}bel-{J}acobi map for higher {C}how groups.
\newblock {\em Compositio Math.}, 142(2):374--396, 2006.
\newblock URL: \url{https://www.cambridge.org/core/journals/compositio-mathematica/article/abeljacobi-map-for-higher-chow-groups/96E7B429B8B0E96C08890BE3BAC41724}.

\bibitem[Lan91]{LandRelChow}
S.~E. Landsburg.
\newblock Relative {C}how groups.
\newblock {\em Illinois J. Math.}, 35(4):618--641, 1991.
\newblock URL: \url{http://projecteuclid.org/euclid.ijm/1255987675}.

\bibitem[Lea90]{LeaExt}
D.~A. Lear.
\newblock {\em Extensions of normal functions and asymptotics of the height pairing}.
\newblock PhD thesis, 1990.
\newblock Copyright - Database copyright ProQuest LLC; ProQuest does not claim copyright in the individual underlying works; Last updated - 2023-09-12.
\newblock URL: \url{https://www.proquest.com/dissertations-theses/extensions-normal-functions-asymptotics-height/docview/303887228/se-2}.

\bibitem[Lev94]{LevRevisit}
M.~Levine.
\newblock Bloch's higher {C}how groups revisited.
\newblock Number 226, pages 10, 235--320. 1994.
\newblock $K$-theory (Strasbourg, 1992).

\bibitem[Lev01]{LevTec}
M.~Levine.
\newblock Techniques of localization in the theory of algebraic cycles.
\newblock {\em J. Algebraic Geom.}, 10(2):299--363, 2001.

\bibitem[Lev09]{LevCub}
Marc Levine.
\newblock Smooth motives.
\newblock In {\em Motives and algebraic cycles}, volume~56 of {\em Fields Inst. Commun.}, pages 175--231. Amer. Math. Soc., Providence, RI, 2009.
\newblock \href {https://doi.org/10.3390/a2041327} {\path{doi:10.3390/a2041327}}.

\bibitem[LL21]{LLMot}
Chao Li and Yifeng Liu.
\newblock Chow groups and {$L$}-derivatives of automorphic motives for unitary groups.
\newblock {\em Ann. of Math. (2)}, 194(3):817--901, 2021.
\newblock \href {https://doi.org/10.4007/annals.2021.194.3.6} {\path{doi:10.4007/annals.2021.194.3.6}}.

\bibitem[MS95]{MulBiext}
S.~M\"uller-Stach.
\newblock {$\bold C^*$}-extensions of tori, higher {C}how groups and applications to incidence equivalence relations for algebraic cycles.
\newblock {\em $K$-Theory}, 9(4):395--406, 1995.
\newblock \href {https://doi.org/10.1007/BF00961471} {\path{doi:10.1007/BF00961471}}.

\bibitem[Par23]{ParLoc}
J.~Park.
\newblock On localization for cubical higher {Chow} groups.
\newblock {\em Tohoku Mathematical Journal}, 75(2), jun 2023.
\newblock \href {https://doi.org/10.2748/tmj.20211221} {\path{doi:10.2748/tmj.20211221}}.

\bibitem[PS08]{PSMHS}
C.~Peters and J.~Steenbrink.
\newblock {\em Mixed {H}odge structures}, volume~52 of {\em Ergebnisse der Mathematik und ihrer Grenzgebiete. 3. Folge. A Series of Modern Surveys in Mathematics [Results in Mathematics and Related Areas. 3rd Series. A Series of Modern Surveys in Mathematics]}.
\newblock Springer-Verlag, Berlin, 2008.

\bibitem[Sai96]{SaitoANF}
M.~Saito.
\newblock Admissible normal functions.
\newblock {\em J. Algebraic Geom.}, 5(2):235--276, 1996.

\bibitem[Sei07]{SeiBier}
M.~Seibold.
\newblock {\em Bierweiterungen f{\"u}r algebraische Zykel und Poincar{\'e}b{\"u}ndel}.
\newblock PhD thesis, 2007.

\bibitem[SGA72]{SGA7}
{\em Groupes de monodromie en g\'eom\'etrie alg\'ebrique. {I}}, volume Vol. 288 of {\em Lecture Notes in Mathematics}.
\newblock Springer-Verlag, Berlin-New York, 1972.
\newblock S\'eminaire de G\'eom\'etrie Alg\'ebrique du Bois-Marie 1967--1969 (SGA 7 I), Dirig\'e{} par A. Grothendieck. Avec la collaboration de M. Raynaud et D. S. Rim.

\bibitem[Son25]{SongAsym}
Y.~Song.
\newblock Parametrization of geometric beilinson--bloch heights via adelic line bundles, 2025.
\newblock \href {https://arxiv.org/abs/2406.19912v3} {\path{arXiv:2406.19912v3}}.

\bibitem[SZ85]{SZVMHS}
J.~Steenbrink and S.~Zucker.
\newblock Variation of mixed {H}odge structure. {I}.
\newblock {\em Invent. Math.}, 80(3):489--542, 1985.
\newblock \href {https://doi.org/10.1007/BF01388729} {\path{doi:10.1007/BF01388729}}.

\bibitem[Tot92]{TotMilnorK}
B.~Totaro.
\newblock Milnor {K}-theory is the simplest part of algebraic {K}-theory.
\newblock {\em K-Theory}, 6(2):177--189, March 1992.
\newblock \href {https://doi.org/10.1007/BF01771011} {\path{doi:10.1007/BF01771011}}.

\bibitem[Wan98]{WangAsym}
B.~Wang.
\newblock Asymptotics of the {A}rchimedean height pairing.
\newblock {\em Amer. J. Math.}, 120(2):229--249, 1998.
\newblock URL: \url{http://muse.jhu.edu/journals/american_journal_of_mathematics/v120/120.2wang.pdf}.

\bibitem[Zei78]{ElZeiFond}
F.~El Zein.
\newblock Complexe dualisant et applications \`a{} la classe fondamentale d'un cycle.
\newblock {\em Bull. Soc. Math. France M\'em.}, (58):93, 1978.

\end{thebibliography}
\end{document}